\documentclass[12pt]{amsart}

\usepackage[all]{xy}
\usepackage{fullpage}
\usepackage{latexsym}
\usepackage{amsmath}
\usepackage{amsfonts}
\usepackage{amssymb}
\usepackage{amsthm}
\usepackage{eucal}
\usepackage{enumerate,yfonts}
\usepackage{mathrsfs}
\usepackage{graphicx}
\usepackage{graphics}
\usepackage{epstopdf}
\usepackage{amscd}
\usepackage{bbm}
\usepackage{hyperref}
\usepackage{url}
\usepackage{color}
\usepackage{bbm}
\usepackage{cancel}
\usepackage{enumerate}
\usepackage{amsmath,amsthm}
\usepackage{amssymb}
\usepackage{epsfig}
\usepackage{pstricks}
\usepackage{xy}
\usepackage{xypic}
\usepackage{bbm}

\newtheorem{thm}{Theorem}[section]
\newtheorem{corollary}[thm]{Corollary}
\newtheorem{lemma}[thm]{Lemma}

\newtheorem{proposition}[thm]{Proposition}
\newtheorem{prop}[thm]{Proposition}

\newtheorem{thm-dfn}[thm]{Theorem-Definition}

\theoremstyle{definition}
\newtheorem{definition}[thm]{Definition}
\newtheorem{remark}[thm]{Remark}
\newtheorem{example}[thm]{Example}

\numberwithin{equation}{section}

\newcommand{\fg}{{\mathfrak g}}
\newcommand{\ft}{{\mathfrak t}}

\newcommand{\fb}{{\mathfrak b}}
\newcommand{\fu}{{\mathfrak u}}
\newcommand{\fp}{{\mathfrak p}}

\newcommand{\fa}{{\mathfrak a}}

\newcommand{\fc}{{\mathfrak c}}

\newcommand{\rW}{{\mathrm W}}

\newcommand{\bC}{{\mathbb C}}

\newcommand{\bG}{{\mathbb G}}
\newcommand{\bZ}{{\mathbb Z}}

\newcommand{\mE}{\mathcal{E}}
\newcommand{\mF}{\mathcal{F}}

\newcommand{\mO}{\mathcal{O}}
\newcommand{\mL}{\mathcal{L}}

\newcommand{\on}{\operatorname}
\newcommand{\D}{\on{D}}

\newcommand{\lra}{\longrightarrow}

\newcommand{\ra}{\rightarrow}
\newcommand{\la}{\leftarrow}

\newcommand{\is}{\simeq}

\newcommand{\Loc}{\on{LocSys}}
\newcommand{\Bun}{\on{Bun}}

\newcommand{\quash}[1]{}  
\newcommand{\nc}{\newcommand}

\newcommand{\bbC}{{\mathbb C}}

\newcommand{\bbH}{{\mathbb H}}

\newcommand{\bbP}{{\mathbb P}}

\newcommand{\bbR}{{\mathbb R}}

\newcommand{\bbZ}{{\mathbb Z}}

\newcommand{\calB}{{\mathcal B}}
\newcommand{\calC}{{\mathcal C}}

\newcommand{\calF}{{\mathcal F}}

\newcommand{\calS}{{\mathcal S}}
\newcommand{\calT}{{\mathcal T}}

\newcommand{\calY}{{\mathcal Y}}

\nc{\al}{{\alpha}} \nc{\be}{{\beta}} \nc{\ga}{{\gamma}}
\nc{\ve}{{\varepsilon}} \nc{\Ga}{{\Gamma}} 
\nc{\La}{{\Lambda}}
\nc{\RHom}{\on{RHom}}

\nc{\ad }{{\on{ad }}}

\nc{\aff}{{\on{aff}}} \nc{\Aff}{{\mathbf{Aff}}}

\nc{\der}{{\on{der}}}

\nc{\diag}{{\on{diag}}}
\newcommand{\End}{{\on{End}}}
\nc{\Fl}{{\calF\ell}}

\newcommand{\Gr}{{\on{Gr}}}
\nc{\Hg}{{\on{Higgs}}}
\newcommand{\Hom}{{\on{Hom}}}

\nc{\Id}{{\on{Id}}}

\nc{\Ind}{{\on{Ind}}}
\newcommand{\Lie}{{\on{Lie}}}
\nc{\Op}{{\on{Op}}}

\nc{\res}{{\on{res}}}

\nc{\tr}{{\on{tr}}}

\newcommand{\GL}{{\on{GL}}}
\nc{\GSp}{{\on{GSp}}} \nc{\GU}{{\on{GU}}} \nc{\SL}{{\on{SL}}}
\nc{\SU}{{\on{SU}}} \nc{\SO}{{\on{SO}}}

\nc{\nh}{{\Loc_{J^p}(\tau')}}
\nc{\bnh}{{\Loc_{\breve J^p}(\tau')}}

\nc{\bU}{{\overline{U}}} \nc{\IC}{{\on{IC}}}

\newcommand{\dota}{\ddot{\text a}}

\newcommand{\beqn}{\begin{equation*}}
\newcommand{\eeqn}{\end{equation*}}

\newcommand{\beq}{\begin{equation}}
\newcommand{\eeq}{\end{equation}}

\nc{\QM}{QM}
\nc{\eval}{\textup{ev}}

\quash{
\topmargin-0.5cm \textheight22cm \oddsidemargin1.2cm \textwidth14cm}

\begin{document}
\title{Quaternionic Satake equivalence
}

       \author{Tsao-Hsien Chen} 
        
       \address{School of Mathematics, University of Minnesota, Minneapolis, Vincent Hall, MN, 55455}
       \email{chenth@umn.edu}

       \author{Mark Macerato} 
        
        \address{Department of Mathematics, UC Berkeley, Evans Hall,
Berkeley, CA 94720}
        \email{macerato@berkeley.edu}
        
        \author{David Nadler} 
        
        \address{Department of Mathematics, UC Berkeley, Evans Hall,
Berkeley, CA 94720}
        \email{nadler@math.berkeley.edu}

         \author{John O'Brien} 
        
       \address{School of Mathematics, University of Minnesota, Minneapolis, Vincent Hall, MN, 55455}
       \email{obri0741@umn.edu}
       
\maketitle
\begin{abstract} 
We establish a derived geometric Satake equivalence for 
the quaternionic general linear group $\GL_{n}(\bbH)$. By applying the 
real-symmetric correspondence for affine Grassmannians, we obtain a 
derived geometric Satake equivalence for the symmetric variety 
$\GL_{2n}/\on{Sp}_{2n}$. We explain how these equivalences fit into 
the general framework of a
geometric Langlands correspondence for real groups and 
the relative Langlands duality conjecture.
As an application, we compute the stalks of the IC-complexes 
for spherical orbit closures in 
the quaternionic affine Grassmannian and the loop space 
of $\GL_{2n}/\on{Sp}_{2n}$. We show the stalks are given by 
the Kostka-Foulkes polynomials for $\GL_n$ but with all degrees doubled. 

\end{abstract}
\setcounter{tocdepth}{2} 
\tableofcontents

\newcommand{\bA}{{\mathbb A}}

\nc{\qQM}{\mathscr{QM}}

\nc{\sw}{\textup{sw}}
\nc{\inverse}{\textup{inv}}
\nc{\negative}{\textup{neg}}

\nc{\Sym}{\textup{Sym}}
\nc{\sym}{\mathit{sym}}
\nc{\Orth}{\textup{O}}
\nc{\Maps}{\textup{Maps}}
\nc{\sm}{\textup{sm}}

\nc{\image}{\textup{image}}

\nc{\risom}{\stackrel{\sim}{\to}}

\section{Introduction}

\subsection{Real-symmetric correspondence}
Let $G_\bbR$ be a real form of a  connected complex reductive group $G$.
Let  $X=G/K$ be the associated symmetric variety under 
Cartan's bijection, where $K$ is the complexification of a maximal compact subgroup $K_\bbR\subset G_\bbR$.

A fundamental 
feature of the representation theory 
of the real group $G_\bbR$
is that 
many results of an analytic nature 
have equivalent purely algebraic formulations in terms of the corresponding 
symmetric variety $X$.
We will call this broad phenomenon the 
\textit{real-symmetric 
correspondence}.
It allows one to 
use  algebraic tools on the symmetric side to 
study questions on the real side, and conversely, to use analytic tools on the 
real side to study questions on the symmetric side.
Famous examples include 
Harish-Chandra's reformulation of 
admissible representations of real groups in terms of $(\fg,K)$-modules, the
Kostant-Sekiguchi correspondence between real and symmetric nilpotent 
orbits, and the Matsuki correspondence between 
$G_\bbR$ and $K$-orbits on the flag manifold of $G$.

In \cite{CN1}, the first and third authors established a real-symmetric correspondence 
relating the dg derived category 
of  spherical constructible sheaves on the 
real affine Grassmannian $\Gr_{G_\bbR}$
 of $G_\bbR$ and the dg derived category 
 of 
spherical constructible sheaves on the loop space $\frak LX$ of $X$.
We are interested in 
applying this 
real-symmetric correspondence 
to study questions in the real and relative geometric Langlands programs.

In the present paper, we consider the question
of a geometric Satake equivalence for real groups and symmetric varieties.
We focus on the case where the real group 
is the quaternionic group $G_\bbR=\GL_n(\bbH)$ with associated symmetric variety 
 $X=\GL_{2n}/\on{Sp}_{2n}$.
We prove the
derived geometric Satake equivalence for $\GL_n(\bbH)$
relating the dg constructible derived category of the quaternionic affine Grassmannian 
with the dg coherent derived category of a  quotient stack 
associated to the Gaitsgory-Nadler dual group $\check G_X$
of $X$ (which is $\check G_X=\GL_n$ in this case).
Via the real-symmetric correspondence, we 
obtain a derived geometric Satake equivalence for the symmetric variety $\GL_{2n}/\on{Sp}_{2n}$. 
As an application, we compute the stalks of the IC-complexes 
for spherical orbit closures in 
the quaternionic affine Grassmannian and the loop space 
of $\GL_{2n}/\on{Sp}_{2n}$. We show the stalks are given by 
the Kostka-Foulkes polynomials for $\GL_n$ but with all degrees doubled.

We explain how these equivalences fit into 
the general framework of 
a geometric Langlands correspondence for real groups,  
due to Ben-Zvi and the third author, 
and 
of the relative Langlands duality conjecture, due to 
Ben-Zvi, Sakellaridis, and Venkatesh.

From the point of view of real groups, 
the quaternionic group $\GL_n(\bbH)$ offers in some sense the simplest possible geometry: just as complex Grassmannians are simpler than real Grassmannians (Schubert cells are $2d$ versus $d$ real-dimensional), quaternionic Grassmannians are  simpler still than complex Grassmannians  (Schubert cells are $4d$ versus $2d$ real-dimensional).   
 On the other hand,
the  geometry of the symmetric variety $\GL_{2n}/\on{Sp}_{2n}$ is more complicated than that of $\GL_{2n}$.
The real-symmetric correspondence allows us to use the
simpler quaternionic geometry of $\GL_n(\bbH)$ to 
answer questions about the more complicated  
geometry of $\GL_{2n}/\on{Sp}_{2n}$.

We now describe the paper in more details.

\subsection{Reminder on derived Satake for $\GL_{2n}$}
Let $\frak L\GL_{2n}$ and $\frak L^+\GL_{2n}$ be the Laurent loop group and Taylor arc group of $\GL_{2n}$.
The affine Grassmannian
$\Gr_{2n}=\frak L\GL_{2n}/\frak L^+\GL_{2n}$ for $\GL_{2n}$
is the ind-variety classifying $\bC[[t]]$-lattices in $\bC((t))^{n}$.
The arc group $\frak L^+\GL_{2n}$  acts naturally on $\Gr_{2n}$,
and we denote by 
$D^b(\frak L^+\GL_{2n}\backslash\Gr_{2n})$  the monoidal dg-category of $\frak L^+\GL_{2n}$-equivariant constructible complexes
on $\Gr_{2n}$ with monoidal structure given by convolution.

Let
$\frak{gl}_{2n}$ be the Lie algebra of $\GL_{2n}$. 
Write $\on{Sym}(\frak{gl}_{2n}[-2])$ for the symmetric algebra of $\frak{gl}_{2n}[-2]$
viewed as a dg-algebra with trivial differential.
The group $\GL_{2n}$ acts on $\on{Sym}(\frak{gl}_{2n}[-2])$ via the adjoint action,
and  we
denote by $D^{\GL_{2n}}_{\on{perf}}(\on{Sym}(\frak{gl}_{2n}[-2]))$
the monoidal dg-category of perfect $\GL_{2n}$-equivariant 
dg-modules over $\on{Sym}(\frak{gl}_{2n}[-2]))$ with monoidal structure given by 
(derived) tensor product of dg-modules.

One of the versions of 
the derived Satake equivalence in \cite{BF} says that there is 
an equivalence of monoidal dg-categories
\[\Psi:D^b(\frak L^+\GL_{2n}\backslash\Gr_{2n})\is D^{\GL_{2n}}_{\on{perf}}(\on{Sym}(\frak{gl}_{2n}[-2]))\]
extending the  geometric Satake equivalence 
$\on{Perv}(\Gr_{2n})\is\on{Rep}(\GL_{2n})$ where 
$\on{Perv}(\Gr_{2n})\subset D^b(\frak L^+\GL_{2n}\backslash\Gr_{2n})$ is the subcategory
$\frak L^+\GL_{2n}$-equivariant perverse sheaves on $\Gr_{2n}$
and $\on{Rep}(\GL_{2n})\subset D^{\GL_{2n}}_{\on{perf}}(\on{Sym}(\frak{gl}_{2n}[-2]))$
is the subcategory of representations of $\GL_{2n}$.\footnote{
The embedding $\on{Rep}(\GL_{2n})\subset D^{\GL_{2n}}_{\on{perf}}(\on{Sym}(\frak{gl}_{2n}[-2])$ is given by 
$\on{V}\mapsto \on{Sym}(\frak{gl}_{2n}[-2])\otimes_\bC\on{V}$.
}

\subsection{Derived  Satake for the quaternionic group $\GL_n(\bbH)$}
Let $\GL_{n}(\bbH)\subset\GL_{2n}$ be the real form given by the 
 rank $n$ quaternionic group.
Let  
$\frak L\GL_n(\bbH)$ and $\frak L^+\GL_n(\bbH)$ be the 
real Laurent loop group and Taylor arc group for $\GL_n(\bbH)$.
By the real affine Grassmannian 
for the quaternionic group $\GL_n(\bbH)$, we will mean the ind semi-analytic variety
$\Gr_{n,\bbH}=\frak L\GL_n(\bbH)/\frak L^+\GL_n(\bbH)$ 
classifying $\bbH[[t]]$-lattices in $\bbH((t))^{n}$.\footnote{
By definition a $\bbH[[t]]$-lattice $\Lambda$ in $\bbH((t))^{n}$ is  a finitely generated 
right $\bbH[[t]]$-submodule of $\bbH((t))^{n}$ such that 
$\Lambda\otimes_{\bbH[[t]]}\bbH((t))=\bbH((t))^{n}$.
}
The real arc group $\frak L^+\GL_n(\bbH)$ acts naturally on 
$\Gr_{n,\bbH}$, and we denote by 
$D^b(\frak L^+\GL_{n}(\bbH)\backslash\Gr_{n,\bbH})$  the monoidal dg-category of $\frak L^+\GL_{n}(\bbH)$-equivariant constructible complexes
on $\Gr_{n,\bbH}$ with monoidal structure given by convolution.
The $\frak L^+\GL_n(\bbH)$-obits on $\Gr_{n,\bbH}$ are all  even real-dimensional
(in fact, $4d$ real-dimensional; see Section \ref{affine grassmannian}), and 
hence middle perversity makes sense.
We denote by 
$\on{Perv}_{}(\Gr_{n,\bbH})$ the category of $\frak L^+\GL_{n}(\bbH)$-equivariant perverse sheaves on $\Gr_{n,\bbH}$.
In \cite{Na}, the third  author established a real geometric Satake equivalence, giving an equivalence of 
monoidal abelian categories 
$\on{Perv}(\Gr_{n,\bbH})\is\on{Rep}(\GL_n)$ in the case at hand.

The first main result of this paper is the following equivalence of monoidal dg derived categories, to be called  
derived Satake for $\GL_{n}(\bbH)$:
\begin{thm}[see Theorem \ref{main}]\label{quaternionic intro}
There is an equivalence of monoidal dg-categories
\[\Psi_\bbH:D^b(\frak L^+\GL_{n}(\bbH)\backslash\Gr_{n,\bbH})\is D^{\GL_n}_{\on{perf}}(\on{Sym}(\mathfrak{gl}_n[-4]))\]
extending the real geometric Satake equivalence 
$\on{Perv}(\Gr_{n,\bbH})\is\on{Rep}(\GL_n)$.
\end{thm}

\quash{
An interesting and important property of the 
the quaternionic geometric Satake equivalence is its compatibility with
the complex geometric Sateke equivalence (see Proposition \ref{abelian Satake}):
We have a 
commutative diagram
\[\xymatrix{\on{Perv}(\Gr_{2n})
\ar[r]^{\mathrm R\ \ }\ar[d]^{\simeq}&\on{Perv}(\Gr_{n,\bbH})_\bbZ\ar[d]^{\simeq}\\
\on{Rep}(\GL_{2n})
\ar[r]^{\Phi\ \ }&\on{Rep}(\GL_n\times\bG_m)
}
\]
where $\Phi:\on{Rep}(\GL_{2n})\to\on{Rep}(\GL_n\times\bG_m)$ is the restriction functor 
along the map
 \beq\label{Arthur intro}
\GL_n\times\bG_m\subset \GL_n\times\on{SL}_2\stackrel{\psi_X}\to\GL_{2n}
 \eeq 
 where $\psi_X$ is the "Arthur parameter" associated to $X$, see~\eqref{SL_2},
 the category
$\on{Perv}(\Gr_{n,\bbH})_\bbZ:=\oplus_i\on{Perv}(\Gr_{n,\bbH})[i]$
is the subcategory of $
D^b(\frak L^+\GL_{n}(\bbH)\backslash\Gr_{n,]\bbH})$
consisting of 
direct sum of shifts of equivalent perverse sheaves on $\Gr_{n,\bbH}$,
and the functor 
\beq\label{nearby intro}
\mathrm R:D^b(\frak L^+\GL_{2n}\backslash\Gr_{2n})\to D^b(\frak L^+\GL_{n}(\bbH)\backslash\Gr_{n,]\bbH})
\eeq
is the 
nearby cycles functor 
along  a real form of the Beilinson-Drinfeld Grassimannian with generic fibers isomorphic to
the 
complex affine Grassimannian $\Gr_{2n}$
and special fiber the
the quaternionic affine Grassimannian $\Gr_{n,\bbH}$ (see Section \ref{Real nearby cycles}).
Note that, unlike the complex algebraic setting, 
the nearby cycles functor $\mathrm R$ is not $t$-exact: it 
maps perverse sheaves to direct sum of shifts of perverse sheaves. 
}

A key ingredient in the proof of Theorem \ref{quaternionic intro} (as in the proof of the abelian quaternionic 
geometric Satake) is 
a nearby cycles functor 
\beq\label{nearby intro}
\mathrm R:D^b(\frak L^+\GL_{2n}\backslash\Gr_{2n})\to D^b(\frak L^+\GL_{n}(\bbH)\backslash\Gr_{n,\bbH})
\eeq
associated to a real form of the Beilinson-Drinfeld Grassmannian with generic fibers isomorphic to
the 
complex affine Grassimannian $\Gr_{2n}$
and special fiber isomorphic to the
quaternionic affine Grassimannian $\Gr_{n,\bbH}$ (see Section \ref{Real nearby cycles}).
Note that, unlike the complex algebraic setting, 
the nearby cycles functor $\mathrm R$ is not $t$-exact: it 
maps perverse sheaves to direct sums of shifts of perverse sheaves (see Proposition~\ref{abelian Satake}).
As a corollary of the proof, we obtain the following spectral description of the  
nearby cycles functor.

Consider the  graded scheme 
\[\widetilde{\frak{gl}}_{2n}=\{\begin{pmatrix}
A[0]&B[-2]\\
C[2]&D[0]
\end{pmatrix}
|A,B,C,D\in\frak{gl}_n\}.\]
We have the natural embedding of (even graded) schemes 
\beq\label{tau}
\tau:\frak{gl}_n[4]\to\widetilde{\frak{gl}}_{2n}[2]\ \ \ \ \tau(C[4])=\begin{pmatrix}
0&I_n\\
C[4]&0
\end{pmatrix}
\eeq
where $I_n$ is the rank $n$ identity matrix. 
Note  the map $\tau$ is $\GL_n$ adjoint-equivariant where 
$\GL_n$ acts on $\widetilde{\frak{gl}}_{2n}[2]$ via the diagonal embedding
$\GL_n\to\GL_{2n}$. Hence  pullback along $\tau$ provides a functor 
\[\tau^*:D^{\GL_{n}}_{\on{perf}}(\on{Sym}(\widetilde{\frak{gl}}_{2n}[-2]))\to D^{\GL_{2n}}_{\on{perf}}(\on{Sym}(\frak{gl}_{2n}[-4]))\]
Here we view the rings of functions 
on $\frak{gl}_n[4]$ and $\widetilde{\frak{gl}}_{2n}[2]$
 as the dg symmetric algebras $\on{Sym}(\frak{gl}_{2n}[-4])$ and
$\on{Sym}(\widetilde{\frak{gl}}_{2n}[-2])$ with trivial differential.
Introduce the  functor
\[\Phi:D^{\GL_{2n}}_{\on{perf}}(\on{Sym}(\frak{gl}_{2n}[-2]))\stackrel{}\lra D^{\GL_{n}}_{\on{perf}}(\on{Sym}(\widetilde{\frak{gl}}_{2n}[-2]))\stackrel{\tau^*}\lra D^{\GL_{2n}}_{\on{perf}}(\on{Sym}(\frak{gl}_{2n}[-4]))\]
where the first functor  
is the sheared forgetful functor 
associated to the $\bG_m$-action on 
$\frak{gl}_{2n}[-2]$ via the co-character 
$2\rho_L:\bG_m\to\GL_{2n}$ (see~\eqref{spectral}). Here $L$ is the complexification of the Levi subgroup of the minimal parabolic subgroup
of $\GL_{n}(\bbH)$.

\begin{thm}[see Theorem \ref{spectral nearby cycle}]\label{spectral nearby cycles, intro}
The following square is naturally commutative
\[\xymatrix{D^b(\frak L^+\GL_{2n}\backslash\Gr_{2n})
\ar[r]^{\mathrm R}\ar[d]_{\Psi}^{\simeq}&D^b(\frak L^+\GL_{n}(\bbH)\backslash\Gr_{n,\bbH})\ar[d]_{\Psi_\bbH}^{\simeq}\\
D^{\GL_{2n}}_{\on{perf}}(\on{Sym}(\frak{gl}_{2n}[-2]))
\ar[r]^{\Phi}&D^{\GL_n}_{\on{perf}}(\on{Sym}(\frak{gl}_{n}[-4]))
}
\]
where 
$\Psi$ and $\Psi_\bbH$ are the complex and quaternionic derived Satake equivalences respectively.
\end{thm}

Later in Section~\ref{ss:intro rel}, we will discuss how Theorem \ref{spectral nearby cycles, intro} fits into 
the general framework of duality for Hamiltonian spaces.

\subsection{Derived  Satake  for the symmetric variety $\GL_{2n}/\on{Sp}_{2n}$}
Let $\frak L\on{Sp}_{2n}$ be the Laurent loop group of the symmetric subgroup $\on{Sp}_{2n}\subset \GL_{2n}$.
There is a natural action of $\frak L\on{Sp}_{2n}$ on $\Gr_{2n}$, and we denote by
$D^b(\frak L\on{Sp}_{2n}\backslash\Gr_{2n})$ the 
dg-category of $\frak L\on{Sp}_{2n}$-equivariant constructible complexes on 
$\Gr_{2n}$.

In \cite[Theorem 8.1]{CN1} it is shown that there is an equivalence 
of dg-categories 
\beq\label{real-symmetric equ, intro}
D^b(\frak L\on{Sp}_{2n}\backslash\Gr_{2n})\is D^b(\frak L^+\GL_{n}(\bbH)\backslash\Gr_{n,\bbH})
\eeq
compatible with the 
natural monoidal
actions of 
$D^b(\frak L^+\GL_{2n}\backslash\Gr_{2n})$, where
the action on
the right hand side 
is through the nearby cycles functor~\eqref{nearby intro}. 
One can view the above equivalence as an example of the real-symmetric correspondence for 
the affine Grassmannian $\Gr_{2n}$.
Combining this with Theorem \ref{quaternionic intro}, we obtain a derived Satake equivalence for $\GL_{2n}/\on{Sp}_{2n}$:

\begin{thm}\label{derived Satake for X, intro}
There is an equivalence of dg-categories 
\[\Psi_X:D^b(\frak L\on{Sp}_{2n}\backslash\Gr_{2n})\is D^{\GL_n}_{\on{perf}}(\on{Sym}(\frak{gl}_{n}[-4]))\]
compatible with the monoidal actions of 
$D^b(\frak L^+\GL_{2n}\backslash\Gr_{2n})\is D^{\GL_n}_{\on{perf}}(\on{Sym}(\frak{gl}_{2n}[-2]))$. 
\end{thm}

\begin{remark}
In general, the $\frak L\on{Sp}_{2n}$-orbits on $\Gr_{2n}$ are 
neither finite-dimensional nor finite-codimensional.
Thus there is not a naive
approach to sheaves on  $\frak L\on{Sp}_{2n}\backslash\Gr_{2n}$
with traditional methods.
To overcome this, we use the observation in \cite{CN1} that
the based loop group $\Omega\on{Sp}(n)$ of the compact real form $\on{Sp}(n)$ 
of $\on{Sp}_{2n}$
acts freely on 
$\Gr_{2n}$ and 
the quotient $\Omega \on{Sp}(n)\backslash\Gr_{2n}$ 
is 
a semi-analytic space of ind-finite type, i.e.,
an inductive limit of 
real analytic schemes of finite type. We define
$D^b(\frak L\on{Sp}_{2n}\backslash\Gr_{2n})$ to be  
the category of  sheaves on $\Omega \on{Sp}(n)\backslash\Gr_{2n}$ constructible
with respect to the stratification coming from the descent of the 
$\frak L\on{Sp}_{2n}$-orbits stratification on $\Gr_{2n}$, see \cite[Definition 1.3]{CN1} and also Remark \ref{def of IC-Stalks}.
\end{remark}

\subsection{Geometric Langlands correspondence for real groups} We discuss here how one might interpret 
our results in terms of geometric Langlands for real groups~\cite{BZN2}. Our results specifically relate to the curve $\mathbb P^1$ with its standard real structure with real points $\mathbb R\mathbb P^1$ (whereas connections to Langlands parameters have been explored~\cite{BZN1}  for $\mathbb P^1$ with its antipodal real structure with no real points).

For complex reductive groups, it is known that 
 the derived Satake equivalence implies the
 geometric Langlands correspondence over the projective line $\mathbb P^1$ via a Radon transform.
To state a version of this in the setting at hand,  let $\Bun_{\GL_{2n}}(\mathbb P^1)$ be the moduli stack of $\GL_{2n}$-bundles over $\mathbb P^1$,
 and let $\on{Loc}_{\GL_{2n}}(\on{S}^2)$ be the moduli 
  stack of  Betti $\GL_{2n}$-local systems on the two sphere $\on{S}^2$.
  Let
  $D_!(\Bun_{2n}(\mathbb P^1))$ be the dg-category of 
constructible complexes on $\Bun_{2n}(\mathbb P^1)$ that are extensions by zero off of a finite-type 
substack, and let $\on{Coh}(\Loc_{2n}(\on{S}^2))$ be the 
dg-category of coherent complexes on $\on{Loc}_{\GL_{2n}}(\on{S}^2)$
with bounded cohomology.

In this setting, the 
geometric Langlands correspondence for $\mathbb P^1$ constructed in  \cite{La} takes the form of an equivalence
\beq\label{GL}
\xymatrix{
D_!(\Bun_{\GL_{2n}}(\mathbb P^1))\ar[r]^-\sim & \on{Coh}(\on{Loc}_{\GL_{2n}}(\on{S^2}))
}\eeq
Moreover, it fits into a 
 commutative diagram of equivalences
\beq\xymatrix{
D_!(\Bun_{\GL_{2n}}(\mathbb P^1))\ar[r]^{\simeq}\ar[d]^{\simeq}&\on{Coh}(\on{Loc}_{\GL_{2n}}(\on{S^2}))\ar[d]^{\simeq}\\
D^b(\frak L^+\GL_{2n}\backslash\Gr_{2n})\ar[r]^{\simeq}&D^{\GL_{2n}}_{\on{perf}}(\on{Sym}(\frak{gl}_{2n}[-2]))
}
\eeq
where the left vertical equivalence 
\beq\label{Radon}
D_!(\Bun_{\GL_{2n}}(\mathbb P^1))\is D^b(\frak L^+\GL_{2n}\backslash\Gr_{2n})
\eeq
is given by the 
Radon transform (see \cite[Proposition 2.1]{La}), and the right vertical equivalence is given by the 
the Koszul duality equivalence 
\beq
\on{Coh}(\frak{gl}_{2n}[-1]/\GL_{2n})\stackrel{}\is
D^{\GL_{2n}}_{\on{perf}}(\on{Sym}(\frak{gl}_{2n}[-2]))
\eeq 
under the isomorphisms
$\on{Loc}_{\GL_{2n}}(\on{S}^2)\is\on{pt}/\GL_{2n}\times_{\frak{gl}_{2n}/\GL_{2n}}\on{pt}/\GL_{2n}\is\frak{gl}_{2n}[-1]/\GL_{2n}$.

As a special case of the affine Matsuki correspondence established in~\cite{CN1},
we have a real group version of the equivalence~\eqref{Radon} taking the form 
\beq\label{affine Matsuki}
D_!(\Bun_{\GL_{n}(\bbH)}(\mathbb{RP}^1))\is D^b(\frak L^+\on{Sp}_{2n}\backslash\Gr_{2n})
\eeq
Here $\Bun_{\GL_{n}(\bbH)}(\mathbb{RP}^1)$ is the real form of 
$\Bun_{\GL_{2n}}(\mathbb{P}^1)$ classifying 
$\GL_{n}(\bbH)$-bundles on the real projective line $\mathbb{RP}^1$.
Combining this with the derived Satake equivalence 
for $\GL_{2n}/\on{Sp}_{2n}$ in Theorem \ref{derived Satake for X, intro}, we obtain the following  
geometric Langlands correspondence for $\GL_n(\bbH)$.

Let $\on{Loc}_{\GL_n}(\on{S^4})$ be the moduli stack of Betti  $\GL_n$-local systems on
the $4$-sphere $\on{S}^4$.
Note that the
presentation $\on{S}^4=\on{D}^4\cup_{\on{S}^3}\on{D}^4$
(where $\on{D}^4$ is the 4-dimensional disk in $\mathbb R^4$)
gives an isomorphism of dg-stacks:
\[\on{Loc}_{\GL_n}(\on{S^4})\is\on{pt}/\GL_{n}\times_{\frak{gl}_{n}[-2]/\GL_{n}}\on{pt}/\GL_{n}\is\frak{gl}_{n}[-3]/\GL_{n}.\]
From  the Koszul duality $\on{Coh}(\frak{gl}_{n}[-3]/\GL_{n})\stackrel{}\is D^{\GL_{n}}_{\on{perf}}(\on{Sym}(\frak{gl}_{n}[-4]))$, 
we obtain 
\beq\label{spectral S^4}
\on{Coh}(\on{Loc}_{\GL_n}(\on{S^4}))\is \on{Coh}(\frak{gl}_{n}[-3]/\GL_{n})\is D^{\GL_{n}}_{\on{perf}}(\on{Sym}(\frak{gl}_{n}[-4]))
\eeq

\begin{thm}\label{quaternionic GL, intro}
There is an equivalence 
\[
D_!(\Bun_{\GL_{n}(\bbH)}(\mathbb{RP}^1))\is\on{Coh}(\on{Loc}_{\GL_n}(\on{S^4}))
\]
that fits into a 
 commutative diagram of equivalences
\beq\xymatrix{
D_!(\Bun_{\GL_{n}(\bbH)}(\mathbb{RP}^1))\ar[r]^{\simeq}\ar[d]_{}^{\simeq}&\on{Coh}(\on{Loc}_{\GL_n}(\on{S^4}))\ar[d]^{\simeq}_{}\\
D^b(\frak L^+\on{Sp}_{2n}\backslash\Gr_{2n})\ar[r]^{\simeq}&D^{\GL_{n}}_{\on{perf}}(\on{Sym}(\frak{gl}_{n}[-4]))}
\eeq
Here the left and right vertical equivalence are the 
affine Matsuki correspondence~\eqref{affine Matsuki}
and Koszul duality~\eqref{spectral S^4} respectively, and
the bottom equivalence is the derived Satake equivalence for 
$\GL_{2n}/\on{Sp}_{2n}$.
\end{thm}

\begin{remark}
The appearance of the $4$-sphere $\on{S}^4$ in the above version of 
geometric Langlands for $\GL_n(\bbH)$
is quite mysterious (at least to the authors of this paper).
It suggests a connection with  twistor theory 
but at the moment we do not have a good explanation. One could note that 
the twistor fibration 
$\mathbb P^3\to\on{S}^4$ arises naturally in the proof of Theorem \ref{quaternionic intro},
(see Section \ref{twistor}).
From the perspective of geometric Langlands for real groups, we expect the spectral side to be expressible in terms of
$\GL_{2n}$-connections on a disk with a partial oper structure along the boundary. This should reflect the usual Satake $\GL_{2n}$-Hecke operators in the bulk and the real Satake $\GL_{n}$-Hecke operators along the boundary.

\end{remark}
\begin{remark}
More generally, the real-symmetric correspondence~\eqref{real-symmetric equ, intro}
 and affine Matsuki correspondence~\eqref{affine Matsuki}
hold for any real group $G_\bbR$.  
It follows that 
a derived Satake equivalence for real groups
or symmetric varieties will imply a version of geometric Langlands correspondence   over $\bbR \bbP^1$ for 
real groups and vice versa.
\end{remark}

\subsection{Relative Langlands duality conjectures}\label{ss:intro rel}
A far-reaching program of Ben-Zvi, Sakellaridis and Venkatesh  proposes 
relative Langlands duality conjectures 
between periods and L-functions (see, e.g., \cite{S}).
A fundamental conjecture in the program predicts  that given a 
complex reductive group $G$ and a 
homogeneous
spherical $G$-variety $X$, there exists 
a (graded) Hamiltonian $\check G$-variety 
$\check M$ together with a moment map 
$\mu:\check M\to\check\fg^*$ equipped with 
a commuting $\bG_m$-action of weight $2$, and an equivalence 
\beq\label{eq:rel sat}
D^b(\frak LX/\frak L^+G)\is \on{Coh}(\check M/\check G)
\eeq
where $\on{Coh}(\check M/\check G)$ is the dg-category 
of $\check G$-equivariant perfect dg-modules over the 
ring of functions on $\check M$ viewed as a dg-algebra
with trivial differential and grading given by the above $\bG_m$-action. 
Moreover, this equivalence  should be 
compatible with the derived Satake equivalence 
$D^b(\frak L^+G\backslash\Gr_G)\is D^{\check G}_{\on{perf}}(\on{Sym}(\check\fg[-2]))\is
\on{Coh}(\check\fg^*[2]/\check G)$, in the sense that
the  right convolution action of $D^b(\frak L^+G\backslash\Gr_G)$
on  $D^b(\frak LX/\frak L^+G)$ should correspond to the 
tensor product action of $\on{Coh}(\check\fg^*[2]/\check G)$ on $\on{Coh}(\check M/\check G)$  through the moment map $\mu$.

We now explain how the derived Satake equivalence for the symmetric variety
$X=\GL_{2n}/\on{Sp}_{2n}$ fits into the general setting of relative Langlands duality.
On the one hand, there are canonical bijections between orbits posets
\[|\frak LX/\frak L^+\GL_{2n}|\leftrightarrow|\frak L\on{Sp}_{2n}\backslash\frak L\GL_{2n}/\frak L^+\GL_{2n}|\leftrightarrow|\frak L\on{Sp}_{2n}\backslash\Gr_{2n}|\]
One can show that there an upgrade of this to an equivalence of categories
\beq\label{identification}
D^b(\frak LX/\frak L^+\GL_{2n})\is D^b(\frak L\on{Sp}_{2n}\backslash\Gr_{2n})
\eeq
where $D^b(\frak LX/\frak L^+\GL_{2n})$ is the dg-category of 
$\frak L^+\GL_{2n}$-equivariant constuctible complexes on the loop space 
$\frak LX$ of $X$, see \cite{CN3}.

On the other hand,
it is expected that the Hamiltonian $\check G$-space 
$\check M$ associated to the symmetric variety $X=\GL_{2n}/\on{Sp}_{2n}$ (note that symmetric varieties are 
spherical) is given by $\check M=T^*(\GL_{2n}/\GL_n^{\on{}}\ltimes U,\psi)$, 
the partial Whittaker reduction of $T^*\GL_{2n}$ with respect to 
the generic homomorphism $\psi$ of the Shalika subgroup $\GL^{\on{}}_n\ltimes U$
of $\GL_{2n}$:
\beq\label{Shalika}
\GL^{\on{}}_n\ltimes U=\{\begin{pmatrix}
A&0\\
0&A
\end{pmatrix}\begin{pmatrix}
I_n&0\\
C&I_n
\end{pmatrix}|\ A\in\GL_n, C\in\frak{gl}_n\},\ \ \ \psi(\begin{pmatrix}
A&0\\
0&A
\end{pmatrix}\begin{pmatrix}
I_n&0\\
C&I_n
\end{pmatrix})=-\tr(C)
\eeq
(see the list of examples of relative duality in \cite{W}).
   
 By Lemma \ref{partial reduction}, there is an isomorphism 
$\check M\is\GL_{2n}\times^{\GL_n^{\on{}}}\frak{gl}_n$
such that the induced isomorphism
$\check M/\GL_{2n}\is\frak{gl}_n/\GL_n$
fits into a commutative diagram
\[\xymatrix{\check M/\GL_{2n}\ar[r]^{\simeq}\ar[d]^{\mu}&\frak{gl}_n/\GL_n\ar[d]^\tau\\
\frak{gl}^*_{2n}/\check\GL_{2n}\ar[r]^{\simeq}&\frak{gl}_{2n}/\GL_{2n}}\]
Here $\mu$ is the moment map,  $\tau$ is the embedding
in~\eqref{tau} (disregarding the cohomological grading), and 
we identify $\frak{gl}^*_{2n}/\GL_{2n}\is\frak{gl}_{2n}/\GL_{2n}$ via the trace form. Thus in view of~\eqref{identification}, the equivalence of Theorem \ref{derived Satake for X, intro}
gives an instance of~\eqref{eq:rel sat} of the form\footnote{
Here we have not been precise about cohomological degrees on the right hand side.} \[D^b(\frak LX/\frak L^+\GL_{2n})\is
\on{Coh}(\check M/\GL_{2n}).\]

\begin{remark}
Our work suggests an interesting relationship between 
real groups and
periods of automorphic forms associated to 
the corresponding symmetric varieties
and we plan to investigate this relationship in more details.
The case of quaternionic group $\on{GL}_n(\mathbb H)$ is related to
 the so-called Symplectic Periods and Jacquet-Shalika Periods
 \cite{JR, JS}.

\end{remark}

\subsection{IC-stalks and Kostka-Foulkes polynomials}
As an application of the proof of Theorem \ref{quaternionic intro}, 
we determine the stalk cohomology of
the IC-complexes 
for the $\frak L^+\GL_n(\bbH)$-orbit closures in 
the quaternionic affine Grassmannian 
$\Gr_{n,\bbH}$
and the 
$\frak L\on{Sp}_{2n}$-orbit closures in the comlplex
affine Grassmannian $\Gr_{2n}$.

The $\frak L^+\GL_n$-orbits (resp.
$\frak L^+\GL_n(\bbH)$ and $\frak L\on{Sp}_{2n}$-orbits)
on $\Gr_n$ (resp. $\Gr_{n,\bbH}$ and $\Gr_{2n}$) are in bijection  with the 
 set 
dominant  co-weights $\Lambda^+_n$
of $\GL_n$, 
see Section \ref{affine grassmannian}. For any 
$\lambda\in\Lambda^+_n$ we denote 
by $\Gr_n^\lambda$ (resp.
$\Gr_{n,\bbH}^\lambda$ and $\Gr^\lambda_{2n,X}$) the corresponding 
orbit and 
by $\IC(\overline{\Gr_n^\lambda})$ (resp. $\IC(\overline{\Gr_{n,\bbH}^\lambda})$ and $\IC(\overline{\Gr^\lambda_{2n,X}})$) the  intersection cohomology complex 
on the orbit closure.
We will write $\mathscr H^i(\mF)$
for the $i$-th cohomology sheaf of 
a complex $\mF$ 
and $\mathscr H_x^i(\mF)$
its stalk at a point $x$.
 
For any pair of dominant coweights $\lambda,\mu\in\Lambda^+_n$, 
we denote by 
$K_{\lambda,\mu}(q)$ the assoicated Kostka-Foulkes polynomial with variable 
$q$. Denote by $\rho_n$ the half-sum of positive roots of $\GL_n$. 
A well-known result of Lusztig \cite{Lu} says that 
we have 
$\mathscr H^{i-2\langle\lambda,\rho_n\rangle}(\IC(\overline{\Gr_n^\lambda}))=0$ for $2\nmid i$
and 
\[\sum_i\on{dim}\mathscr H^{2i-2\langle\lambda,\rho_n\rangle}_x(\IC(\overline{\Gr_n^\lambda}))q^i
=q^{\langle\lambda-\mu,\rho_n\rangle}K_{\lambda,\mu}(q^{-1}) \ \ \ \ 
\text{for}\ \ x\in\Gr_{n}^\mu.\]
We have the following real and symmetric analogue:
\begin{thm}[see Corollary \ref{parity vanishing}, Theorem \ref{IC-stalks}]\label{IC stalks, intro}
Let  $\lambda,\mu\in\Lambda_n^+$. For any $x\in\Gr_{n,\bbH}^\mu$ and $y\in\Gr_{2n,X}^\mu$, we have 
\begin{enumerate}
\item
$\mathscr H^{i-4\langle\lambda,\rho_n\rangle}(\IC(\overline{\Gr^\lambda_{n,\bbH}}))=
\mathscr H^{i-4\langle\lambda,\rho_n\rangle}(\IC(\overline{\Gr_{2n,X}^\lambda}))= 0$
for $4\nmid i$,
\item
$\underset{i}\sum\on{dim}\mathscr H^{4i-4\langle\lambda,\rho_n\rangle}_x(\IC(\overline{\Gr^\lambda_{n,\bbH}}))q^i=
\underset{i}\sum\on{dim}\mathscr H^{4i-4\langle\lambda,\rho_n\rangle}_y(\IC(\overline{\Gr^\lambda_{2n,X}}))q^i=q^{\langle\lambda-\mu,\rho_n\rangle}K_{\lambda,\mu}(q^{-1})$.
\end{enumerate}
\end{thm}

In other words, the theorem above says that 
the IC-complex for the $\frak L^+\GL_{n}(\bbH)$ and $\frak LK$-orbit closures 
$\overline{\Gr^\lambda_{n,\bbH}}$ and $\overline{\Gr^\lambda_{2n,X}}$
 have the same stalk cohomology
as the ones $\overline{\Gr_n^\lambda}$ for $\GL_n$, but with all degrees doubled.

\begin{remark}\label{def of IC-Stalks}
To define the  IC-stalk $\mathscr H_y^i\IC(\overline{\Gr_{2n,X}^\lambda})$ at 
$y\in\Gr_{2n,X}^\mu$
we use the observation that $\Gr_{2n,X}^\mu$ has finite codimension in 
$\overline{\Gr_{2n,X}^\lambda}$ and hence the IC-stalk makes sense.
This can be made precise using the observation in
\cite{CN1} that 
the image $\frak L\on{Sp}_{2n}\backslash\Gr_{2n,X}^\lambda$
of the 
$\frak L\on{Sp}_{2n}$-orbits $\Gr_{2n,X}^\lambda$
in the quotient $\Omega\on{Sp}(n)\backslash\Gr_{2n}$ is 
finite dimensional with even real dimension and 
the collection $\{\frak L\on{Sp}_{2n}\backslash\Gr_{2n,X}^\lambda\}_{\lambda\in\Lambda_n^+}$
forms a nice stratification of $\Omega\on{Sp}(n)\backslash\Gr_{2n}$.
This allow us to define the IC-complex
$\IC(\overline{\Gr_{2n,X}^\lambda})$ of $\overline{\Gr_{2n,X}^\lambda}$ (and hence the IC-stalks) as the IC-complex  for the 
orbit closure $\Omega\on{Sp}(n)\backslash\overline{\Gr_{2n,X}^\lambda}$
of $\Omega\on{Sp}(n)\backslash\Gr_{2n,X}^\lambda$ inside $\Omega\on{Sp}(n)\backslash\Gr_{2n}$.

\end{remark}

\begin{remark}
We first prove Theorem \ref{IC stalks, intro}
in the real case using the nice  geometry of quaternionic 
Grassmannian $\Gr_{n,\bbH}$ and 
deduce the symmetric case via the real-symmetric correspondence.
At the moment, we do not know a direct argument in the symmetric case.
\end{remark}

\subsection{Outline of the proof}
We briefly explain the proof of Theorem \ref{quaternionic intro}.
Similar to the proof of the derived Satake for complex reductive groups \cite{BF}, the desired equivalence follows from the following two statements:
(1) the  
de-equivariantized Extension algebra
$\on{Ext}^*_{D^b_{}(\frak L^+G_{n,\bbH}\backslash\Gr_{n,\bbH})}(\IC_0,\IC_0\star\mO(G_n))$ is isomorphic to the dg symmetric algebra 
$\on{Sym}(\frak{gl}_n[-4])$, see Proposition \ref{computation of Ext}, and 
(2)
the
dg algebra $\on{RHom}_{D^b_{}(\frak L^+G_{n,\bbH}\backslash\Gr_{n,\bbH})}(\IC_0,\IC_0\star\mO(G_n))$ is formal, see Proposition \ref{formality}.

We deduce (1) from a fully-faithfulness property of the equivariant cohomology functor
In \cite{BF}, this was proved using a general result of Ginzburg \cite{G2}
whose proof uses Hodge theory  and hence can not apply directly to the real 
analytic setting.
We observe that 
in the situation of quaternionic affine Grassmannian 
the stalks of the IC-complexes satisfy a parity vanishing property and,
as observed in \cite{AR},
one can use parity considerations in place of Hodge theory.
To prove the the parity vanishing result of the IC-stalks 
we show that 
fibers of certain convolution Grassmannians (which are basically 
quaternionic Springer fibers)
admit pavings by quaternionic affine spaces.

The proof of (2) in \cite{BF} also
relies on Hodge theory (or theory of weights) and hence must be modified in the real setting.  We observe that 
the nearby cycles functor~\eqref{nearby intro}
induces a surjective homomorphism form the dg algebra 
$\on{RHom}_{D^b_{}(\frak L^+G_{2n}\backslash\Gr_{2n})}(\IC_0,\IC_0\star\mO(G_{2n}))$
associated to the complex affine Grassimannian $\Gr_{2n}$ to  the dg algebra 
$\on{RHom}_{D^b_{}(\frak L^+G_{n,\bbH}\backslash\Gr_{n,\bbH})}(\IC_0,\IC_0\star\mO(G_n))$.
Since the former 
dg algebra  is  formal, thanks to \cite{BF},
the desired claim follows from a 
general criterion of formality, see Lemma \ref{criterion}.

\begin{remark}
We expect that the proof strategy outlined above are applicable for 
general real groups:  the parity vanishing, fully-faithfulness, and formality results should hold in general.

\end{remark}

\subsection{Organization}
We briefly summarize the main goals of each sections.
In Section \ref{sheaves}, we recall some notations and results on
constructible sheaves on a semi-analytic stack.
In Section \ref{Spectral side}, we study the spectral side of the 
quaternionic Satake equivalence including 
results on quaternionic groups, symplectic groups,
regular centralizers group schemes, and Whittaker reduction.
In Section \ref{constructible side}, we study the constructible side of the 
equivalence including the study of 
nearby cycles functors, parity vanishing results,
fully-faithfulness of the equivariant cohomology functor,
 and the computation of 
 the IC-stalks and
the de-equivariantized extension algebra.
Finally, in Section \ref{Main results}, 
we prove  the formality result and deduce the 
derived Satake equivalence for quaternionic groups 
including a version involving nilpotent singular supports (Theorem \ref{main})
and also the spectral description of the nearby cycles functor (Theorem \ref{spectral nearby cycle}).

\subsection{Acknowledgements}

We would like to thank David Ben-Zvi, 
Jonathan Wang and Geordie Williamson  for useful discussions.

T.-H. Chen would also like to 
thank the Institute for Advanced Study in Princeton for support, hospitality, and a nice research environment.

T.-H. Chen
is supported by NSF grant DMS-2001257 and the S. S. Chern Foundation.
D. Nadler is supported by NSF grant DMS-2101466. M. Macerato is supported by an NSF Graduate Research Fellowship.

\section{Constructible sheaves on a semi-analytic stack}\label{sheaves}
We will be working with $\bC$-linear dg-categories (see, e.g., \cite[Section 0.6]{DG} for a concise summary of the theory of dg-categories). Unless specified otherwise, all dg-categories will be assumed cocomplete, i.e., containing all small colimits,
and all functors between dg-categories will be assumed continuous, i.e., preserving all small colimits.

We collect some facts about constructible sheaves on a semi-analytic stack.
Recall that a subset $Y$ of a 
real analytic manifold $M$ 
is called semi-analytic if
any point $y\in Y$ has a open neighbourhood $U$ 
such that the intersection $Y\cap U$ is a finite union of sets of the form
\[\{y\in U|f_1(y)=\cdot\cdot\cdot=f_r(y)=0, g_1(y)>0,..., g_l(y)>0\},\]
where the $f_i$ and $g_j$ are real analytic functions on 
$U$.
A map $f:Y\to Y'$ between two semi-analytic sets is called semi-analytic if
it is continuous and its graph is a semi-analytic set.

For any semi-analytic set $S$, 
we define $D(Y)=\on{Ind}(D^b(Y))$ to be the ind-completion
of the 
 bounded dg-category $D^b(Y)$ of $\bC$-constructible sheaves on 
$Y$. For any semi-analytic stack $\calY$ 
we define $D(\calY):=\underset{I}{\on{lim}}\ D(Y)$
where the index category is that of semi-analytic sets equipped with a semi-analytic map
to $\calY$, and the transition functors are given by $!$-pullback.
Since we are in the constructible context, $!$-pullback admits a left adjoint, given by 
$!$-pushforward, and it follows that 
$D(\calY)=\underset{S}{\on{colim}}D(S)$. In particular,
$D(\calY)$ is compactly generated. 
We let $D(\calY)^c$ be the full subcategory of compact objects
and $D^b(\calY)\subset D(\calY)$ be the full subcategory
of objects that pull back to an object of $D^b(Y)$ for any $Y$ mapping to $\calY$.
Note that we have natural inclusion $D(\calY)^c\subset D^b(\calY)$
but the inclusion is in general
not an equality. For example, the constant sheaf $\bC_{\calY}\in D^b(\calY)$  
for the classifying stack $\calY=B(\GL_1(\bC))$ is not compact.

Let $f:\calY\to\calY'$ be a map between semi-analytic stacks.
We have the usual six functor formalism
$f_*$, $f^!$, $f_*$, $f_!$, $\otimes$, $\underline{\on{Hom}}$.

For a semi-analytic stack $\calY$, with projection map $p:\calY\to\on{pt}$, and $\mF\in D(Y)$,
we will write 
$H^*(Y,\mF):=p_*(\mF)$ for the cohomology of $\mF$.
If $\calY$ is isomorphic to the quotient stack $\calY\is G\backslash Y$, where 
$G$ is a Lie group acting real analytically on a semi-analytic set $Y$, 
we will write 
$H^*_G(Y,\mF):=(p_{BG})_*(\mF)$ for the 
the equivariant cohomology of $\mF$ 
where $p_{BG}:\calY=G\backslash Y\to BG$ is the projection map.
When it is clear from the content we will abbreviate 
$H^*(Y,\mF)$ and $H^*_G(Y,\mF)$ by $H^*(\mF)$
and $H^*_G(\mF)$.

For an ind semi-analytic stack $\calY=\underset{I}{\on{colim}}\ \calY_i$ we define 
$D(\calY)=\underset{I}{\on{lim}}\ D(\calY_i)$ where the limit is taking with respect to 
the $!$-pull-back along the closed embedding $\iota_{i,i'}:\calY_i\to\calY_{i'}, i,i'\in I$.

\section{Spectral side}\label{Spectral side}
\subsection{Quaternion group}
For any positive integer $n$,
we denote by $G_n=\GL_n(\bC)$ the complex Lie group of $n\times n$ invertible matrices 
and $\fg_n=\frak{gl}_n(\bC)$ its Lie algebra of $n\times n$ matrices.
We write 
$B_n$, $N_n$, and $T_n$ for the subgroups of $G_n $ consisting of upper triangular matrices,
upper triangular unipotent matrices,
 and 
diagonal matrices and 
$\fb_n$, $\frak n_n$, and $\ft_n$ for their Lie algebras.
We denote by $\rW_n$ the Weyl group of $G_n$ acting on $\ft_n$
 by the permutation action.
 We let $\frak c_n=\ft_n//\rW_n$.
 We will identify $\frak c_n$ with the space of 
 degree $n$ monomials in such a way that, under the above identification, the 
 Chevalley map
$\chi_n:\fg_n\to \fg_n//G_n\is\fc_n$
becomes the map sending a matrix to its 
characteristic polynomial.
We will identify $\fg_n\is\fg_n^*$ using the trace paring 
$\fg_n\times\fg_n\to\bC$,
$(A,B)\mapsto\tr(AB)$.

 Let $\mathbb H= \{a+ib+jc+kd\}$ denote the quaternions. 
Consider the quaternionic vector space $\bbH^n$ where
$\mathbb H$ acts via right multiplication.
Let 
$G_{n,\bbH}$ be the Lie group of 
$\bbH$-linear invertible endomorphisms of $\bbH^n$,
which can be identified with the space $\GL_n(\bbH)$ of $n\times n$-invertible 
quaternionic matrices, and let 
$\frak g_{n,\bbH}$ be the Lie algebra of $\bbH$-linear endomorphisms of 
$\bbH^n$ which can be identified with the space $\frak{gl}_n(\bbH)$ of 
$n\times n$ 
quaternionic matrices.

Using the identification 
$\bC^{2n}\is\bbH^n$ sending $(z,w)\to q=z+jw$ , $z,w\in\bC^n$, one can realize 
$G_{n,\bbH}$ as a real form of $G_n$.
Namely, the endomorphism of $\bbH^n$ sending 
$q\to qj$ corresponds to the endomorphism 
of $\bC^{2n}$ sending 
\beq\label{j}
(z,w)\to (-\bar w,\bar z)
\eeq and we can identify 
$\fg_{n,\bbH}$ and $G_{n,\bbH}$ as the subsets of 
$\fg_n$ and $G_n$ consisting of 
$\bC$-linear endomorphisms of $\bC^{2n}$ that commute with the 
map~\eqref{j}. Equivalently, consider the $2n\times 2n$-matrix
\[S_n=\begin{pmatrix} 0 & -I_n  \\ 
I_n &0  \end{pmatrix}\]
where $I_n $ is the $n\times n$ identity matrix.
Then the endomorphism $\eta$ of $\fg_{2n}$ (resp. $G_{2n}$) sending 
$X\in\fg_{2n}$ (resp. $X\in G_{2n}$) to 
\[\eta(X)=S_n\overline XS_n^{-1}\]
defines a real form of $\fg_{2n}$ (resp. $G_{2n}$), that is, an
anti-holomorphic conjugation on $\fg_{2n}$ (resp. $G_{2n}$)
and $\fg_{n,\bbH}$ and $G_{n,\bbH}$ are the $\eta$-fixed points 
in $\fg_{2n}$ and $G_{2n}$. Concretely, $\fg_{n,\bbH}$ (resp. $G_{n,\bbH}$)
consists of $2n\times 2n$-matrices (resp. invertible matrices) of the form
\[\begin{pmatrix} A & B  \\ 
-\overline B &\overline A  \end{pmatrix}\]
where $A,B\in\fg_n$.

We denote by 
$\ft_{n,\bbH}\subset\fg_{n,\bbH}$ (resp.
$T_{n,\bbH}\subset G_{n,\bbH}$) the Cartan subalgebra (resp. Cartan subgroup) 
consisting of matrices (resp. invertible matrices)
\[\begin{pmatrix} A & 0  \\ 
0 &\overline A  \end{pmatrix}\]
where $A\in\ft_n$.

We denote by 
$P_{n,\bbH}=M_{n,\bbH} A_{n,\bbH} N_{n,\bbH}$  the standard minimal parabolic subgroup of 
$G_{n,\bbH}$ consisting of invertible upper triangular quaternionic matrices
and $\fp_{n,\bbH}=\frak m_{n,\bbH}\oplus\fa_{n,\bbH}\oplus\frak n_{n,\bbH}$ its Lie algebra.

\subsection{Symplectic group}
According to the Cartan classification of real forms, the conjugation $\eta$
corresponds to a holomorphic involution $\theta$ on 
$G_{2n}$ (resp. $\bC$-linear involution of $\fg_{2n}$) characterized by the property that 
$\eta\circ\theta=\theta\circ\eta$ is a compact real form, that is,  the fixed points 
subgroup (resp. subalgebra)
of $\eta_c:=\eta\circ\theta$:
\[G_c=(G_{2n})^{\eta_c}\ \ \ \ \ \  (resp.\ \ \  \fg_c=(\fg_{2n})^{\eta_c})\]
is compact. In our case, we will take
$\theta$ to be
\[\theta(X)=S_n(X^t)^{-1}S_n^{-1}\ \ \ \ \ (resp.\ \ \ \  \theta(X)=-S_n(X^t)S_n^{-1}),\]
where $X\in G_{2n}$ (resp. $X\in\fg_{2n}$), and we have 
\[\eta_c(X)=(\overline X^t)^{-1}\ \ \ \ \ \ (resp.\ \ \ \  \theta(X)=-\overline X^t)\]
and the corresponding compact subgroup $G_c=(G_{2n})^{\eta_c}$
is the group 
of $2n\times 2n$-unitary matrices.

The $\theta$-fixed point subgroup $K=(G_{2n})^\theta=\on{Sp}_{2n}$ 
is the symplectic group  of rank $n$
and the intersection  
\[K_c:=
\on{Sp}_{2n}\cap G_c=\on{Sp}(n)\] is the compact symplectic group. 
The Lie algebras $\frak k_{}$ and $\frak k_c$ consist of matrices
\[\begin{pmatrix}
A&B\\
 C&-A^t
\end{pmatrix},\]
where $A,B, C\in\frak g_n$
satisfying $B=B^t$ and $C=C^t$
for $\frak k_{}$, and the additional condition
$A=-\bar A^t$ and $C=-\overline B$
for $\frak k_c$.

Recall the Cartan decompostion of the Lie algebra
 $\fg_{2n}=\frak k\oplus\frak p$
 where $\frak p$ is the $(-1)$-eigenspace of $\theta$.
The Cartan decompostion induces a decomposition 
$\ft_{2n}=\ft\oplus\frak s$
where 
$\ft=\ft_{2n}\cap\frak k$ is a Cartan subalgebra of $\frak k$ consisting of diagonal matrices of the form
\[\on{diag}(h_1,...,h_n,-h_1,...,-h_n)\]
and $\frak s=\frak t_{2n}\cap\fp\subset\fp$ is a maximal abelian subspace of $\fp$ 
consisting of diagonal matrices of the form
\[\on{diag}(s_1,...,s_n,s_1,...,s_n).\] 
We denote by 
$\rW$ the Weyl group of $K$ and $\rW_{\fa}=N_K(\fa)/Z_K(\fa)$ the little Weyl group.
We have $\rW\is\rW_n\ltimes\{\pm1\}^n$ and $\rW_\fa\is\rW_n$.
We let 
$\fc=\frak h//\rW$. Then the natural inclusion 
$\ft\to\ft_{2n}$ gives rise to an embedding 
\beq\label{char poly}
\frak c=\ft//\rW\lra\ft_{2n}//\rW_{2n}=\frak c_{2n}
\eeq
whose image consists of even monomials of degree $2n$.

Finally, we denote by 
$X_{}=G_{2n}/K_{}$
and $X_c=G_c/K_c$ the symmetric space and compact symmetric space 
associated to $K_{}$ and $K_c$.

\subsection{Notation related to root structure}
Let $\Lambda_{n}=\Hom(\bC^\times, T_n)$ be the coweight lattice of $T_n$
and let $\Lambda_n^+$ be the set of dominant coweights with respect to $B_n$.  
Let $2\rho_n$ be the sum of the positive roots of $G_n$
and let $\langle\lambda,2\rho_n\rangle\in\bZ$ be the natural paring for a coweight $\lambda\in\Lambda_n$.

Let $S\subset T$ be the maximal split torus corresponds to the maxmial abelian subspace $\frak s\subset\fp$
and 
let $\Lambda_S=\Hom(\bC^\times,S)$ be the set of real coweights
and $\Lambda_S^+=\Lambda_S\cap\Lambda_{2n}^+$ be the set of dominant real coweights.
There is natural identification 
$S\is T_n$ sending $\on{diag}(s_1,...s_n,s_1,...,s_n)$ to $\on{diag}(s_1,...,s_n)$
and hence natural identification $\Lambda_S\is\Lambda_n$
and $\Lambda_S^+\is\Lambda^+_n$.

\subsection{Regular centralizers}
\subsubsection{}
Consider the following embedding
\beq\label{tau embedding}
\tau:\fg_n\to \fg_{2n}\ \ \ \ \tau(C)=\begin{pmatrix}0&I_n\\
C&0\end{pmatrix}.
\eeq
Note that the map $\tau$ is
$G_n$-equivariant where $G_n$ acts on $\fg_{2n}$
via diagonal embedding
$\delta:G_n\to G_{2n}$.
Thus it indues 
an embedding  on the invariant quotients (denoted again by $\tau$)
\beq\label{c_X}
\tau:\fc_n=\fg_n//G_n\to \fg_{2n}//G_{2n}\is\fc_{2n}\ \ \ 
\tau(c_1,...,c_n)=(0,c_1,0,c_2,...,0,c_{n})
\eeq
whose image consists of even monomials of degree $2n$.
Note that the image of $\tau$ is equal to the image of the map
$\fc=\ft//\rW\to\fc_{2n}=\ft_{2n}//\rW_{2n}$ in~\eqref{char poly}, and hence there is 
natural identification 
\beq\label{c=c_n}
\fc\is\fc_n
\eeq
such that 
$\tau:\fc_n\is\fc\to\fc_{2n}$

Recall the   the group scheme of centralizers $I_{n}$ (resp. $I_{2n}$)
 over $\fg_{n}$ (resp. $\fg_{2n}$)
and  the group scheme of regular centralizers $J_n$ (resp. $J_{2n}$)
 over $\fc_n$ (resp. $\fc_{2n}$), see \cite[Section 3]{Ng}.

\begin{lemma}
There is a natural closed embedding of 
groups schemes 
$J_n\to J_{2n}$ fits into the following commutative diagram
\beq\label{key diagram}
\xymatrix{J_n\ar[r]\ar[d]&J_{2n}\ar[d]\\
\fc_n\ar[r]^{}&\fc_{2n}}
\eeq
where the bottom arrow is the map in~\eqref{c_X}.
\end{lemma}
\begin{proof}
We first claim  that 
$\tau(\fg_n^{reg})=\fg_{2n}^{reg}\cap\tau(\fg_n)$.
Let $x=\tau(C)=\begin{pmatrix}0&I_n\\
C&0\end{pmatrix}$.
If $x$
 is in $\fg_{2n}^{reg}$, then 
the centralizer $Z_{G_{2n}}(x)$ of $x$ in $G_{2n}$ is abelian.
It follows that  
$Z_{G_n}(C)$ is also abelian 
and the characterization of regular elements 
for $\fg_n$
 implies that $C\in\fg_{n}^{\on{reg}}$. 
 On the other hand, if $C\in\fg_n^{reg}$ then without loss of generality we can assume 
$C$ is a companion matrix
and easy computation shows that $x$ is in $\fg_{2n}^{reg}$ (see \eqref{matrix}). 
The claim follows.

Let $I_n^{reg}=I_n|_{\fg_n^{reg}}$ and 
$I_{2n}^{reg}=I_{2n}|_{\fg_{2n}^{reg}}$. 
Then the claim implies that 
we have a commutative diagram
\beq\label{reg}
\xymatrix{
I_n^{reg}\ar[r]\ar[d]& I_{2n}^{reg}\ar[d]\\
\fg_n^{reg}\ar[r]^{\tau}&\fg_{2n}^{reg}}
\eeq
Since 
$J_{2n}\is I_{2n}^{reg}//G_{2n}$ is the descent of 
$I_{2n}^{reg}$
along the map
$\fg_{2n}^{reg}\to\fc_{2n}$, 
the restriction $J_{2n}|_{\fc_n}$ is the descent of 
$I_{2n}^{reg}|_{\fg_n^{reg}}$ along the map
$\fg_{n}^{reg}\to\fc_n$:
\[J_{2n}|_{\fc_n}\is(I_{2n}^{reg}//G_{2n})|_{\fc_n}\is I_{2n}^{reg}|_{\fg_n^{reg}}// G_n.\]
Since the maps in~\eqref{reg} are compatible with the natural 
$G_n$-action and the desired map is the map 
on the GIT quotients
\[J_n\is I_n^{reg}//G_n\lra 
I_{2n}^{reg}|_{\fg_n^{reg}}// G_n\is J_{2n}|_{\fc_n}.\]

\end{proof}

\subsubsection{Kostant sections}\label{Kostant section}
We give an alternative construction of the map 
$J_n\to J_{2n}$ in~\eqref{key diagram} using Kostant sections.

Consider the following two ordered bases of 
$\bC^{2n}$:
the standard basis $\{e_1=(1,0,..,0),...,e_{2n}=(0,...,0,1)\}$
and the basis 
 $\{w_1=e_1, w_2=e_{3},...,w_n=e_{2n-1}, w_{n+1}=e_2,w_{n+2}=e_4,...,w_{2n}=e_{2n}\}$.
 Let $P\in G_{2n}$ be the matrix associated to the 
 linear map $w_i\to v_i$ in the basis $w_1,...,w_{2n}$.

For any positive integer $s$,
consider the Kostant section $\kappa_s:\fc_s\to\fg_s$
for $G_s$
given by 
\[
 \kappa_s(c)=\begin{pmatrix} 0 & 1 & \\ 
 \vdots& 0 &\ddots &  \\ 
\vdots &  & \ddots & 1 \\
-c_{s}& -c_{s-1}&\hdots  &-c_1  \end{pmatrix}\ \ \ \ c=(c_1,...,c_s)\in\fc_s
\]
A direct computation shows that 
\[
\begin{pmatrix} 0 & 1 & \\ 
 \vdots& 0 &\ddots &  \\ 
\vdots &  & \ddots & 1 \\
-c_{2n}& -c_{2n-1}& \hdots &-c_1  \end{pmatrix}=
P\begin{pmatrix}0&I_n\\
C&D\end{pmatrix}P^{-1}
\]
where
\[C=\begin{pmatrix} 0 & 1 & \\ 
 \vdots& 0 &\ddots &  \\ 
 \vdots &  & \ddots & 1 \\
-c_{2n}& -c_{2n-2}& \hdots &-c_2  \end{pmatrix}\ \ \ \ 
D=\begin{pmatrix} 0 & 0 &\hdots&0 \\ 
 \vdots& \vdots &\hdots &0  \\ 
0 & 0 & \hdots & 0 \\
-c_{2n-1}& -c_{2n-3}& \hdots &-c_1\end{pmatrix}\]
It follows that for any 
$\fc=(c_1,...,c_n)\in\fc_n$ with $\tau(c)=(0,c_1,0,c_2,...,0,c_{n})
\in\fc_{2n}$
we have
\beq\label{matrix}
\begin{pmatrix} 0 & 1 &\hdots&\hdots& \hdots&\hdots&0\\ 
 \vdots& 0 &\ddots &  \\ 
\vdots &  & \ddots & \ddots \\
\vdots&  &\ddots &\ddots & &\\
\vdots&  &\ddots &\ddots & \ddots&&1\\
-c_{n}& 0& -c_{n-1} &0&\hdots&0&-c_1  \end{pmatrix}=P\begin{pmatrix}0&I_n\\
\kappa_n(c)&0\end{pmatrix}P^{-1}.
\eeq
Thus
there is a commutative diagram
\beq\label{diagram ks}
\xymatrix{\fg_n^{reg}\ar[r]^{\tau}&\fg_{2n}^{reg}\\
\fc_n\ar[u]^{\kappa_n}\ar[r]^{\tau}&\fc_{2n}\ar[u]_{\on{Ad}_{P}^{-1}\circ\kappa_{2n}}}
\eeq
In particular, we have 
\[\tau\circ\kappa_n:\fc_n\to \fg_{2n}^{reg}.\] 
 The pull-back of the group scheme $I_{2n}^{reg}$ along the map above $\tau\circ\kappa_n$
is isomorphic to
\[(\tau\circ\kappa_n)^*(I_{2n}^{reg})\is ((\kappa_{2n})^*\on{Ad}_{P^{-1}}^*(I_{2n}^{reg}))|_{\fc_n}\is
(\kappa_{2n})^*(I_{2n}^{reg})|_{\fc_n}\is J_{2n}|_{\fc_n}\]
and the desired map is given by pull-back of~\eqref{reg} along the 
map $\tau\circ\kappa_n$:
\beq\label{alternative description}
J_n\is\kappa_n^*(I_n^{reg})\lra\kappa_n^{*}(I^{reg}_{2n}|_{\fg_n^{reg}})\is
(\tau\circ\kappa_n)^*(I_{2n}^{reg})\is J_{2n}|_{\fc_n}
\eeq

\subsubsection{}
The identification $\fc\is\fc_n$ in~\eqref{c=c_n} gives rise to a map
$\ft\to\fc\is\fc_n$ and 
we shall give a description 
of the pull-back 
\beq\label{pull-back}
J_n\times_{\fc_n}\ft\to J_{2n}\times_{\fc_{2n}}\ft
\eeq
of~\eqref{key diagram} 
along  $\ft\to\fc_n$.

Consider the map 
\beq\label{e^T_2n}
e^{T_{2n}}:\ft_{2n}\to\fg_{2n},\ \ \ \ \ e^{T_{2n}}(t)=
\begin{pmatrix} t_1 & 1 & \\ 
 0& t_2 &\ddots &  \\ 
 \vdots &  & \ddots & 1 \\
0&  \hdots&0 &t_{2n} \end{pmatrix}\ \ \ t=\on{diag}(t_1,...,t_{2n})
\eeq
Note that the image of $e^{T_{2n}}$ consist of regular elements.
We have the following commutative diagram
\[\xymatrix{\ft_{2n}\ar[r]^{e^{T_{2n}}}\ar[d]&\fg_{2n}\ar[d]\\
\fc_{2n}\ar[r]&\fc_{2n}}\]
where the vertical arrows are the natural adjoint quotient maps.
If follows that there is a canonical isomorphism
\[J_{2n}\times_{\fc_{2n}}\ft_{2n}\is  (e^T)^*I_{2n}=(G_{2n}\times\ft_{2n})^{e^{T_{2n}}}\]
where $(G_{2n}\times\ft_{2n})^{e^{T_{2n}}}$ the centralizer of $e^{T_{2n}}$ in $G_{2n}\times\ft_{2n}$.

Consider the restriction $e^T=e^{T_{2n}}|_{\ft}:\ft\to\fg_{2n}$.
Concretely, we have 
\beq\label{e^T}
e^{T}:\ft_{}\to\fg_{2n},\ \ \ \ \ e^{T_{}}(t)=
\begin{pmatrix} t_1 & 1 &0&\hdots&\hdots&0 \\ 
0&  \ddots&\ddots & && \vdots\\ 
\vdots&  \ddots& t_n & \ddots& &\vdots\\ 
\vdots& & \ddots& -t_1&\ddots&0\\ 
\vdots &  &  & \ddots & \ddots &1\\
0& \hdots&\hdots  &\hdots  &0 &-t_n
\end{pmatrix}\ \ \ t=\on{diag}(t_1,...,t_n,-t_1,...,-t_{2n})
\eeq
It is clear that 
\[J_{2n}\times_{\fc_{2n}}\ft\is  (e^T)^*I_{2n}=(G_{2n}\times\ft_{})^{e^{T_{}}}.\]

Consider the map
\beq\label{e^T_X}
e^T_X:\ft\to \fg_n,\ \ \ \ \ \ e^T_X(t)=\begin{pmatrix} t_1^2 & 1 & \\ 
 \vdots& t_2^2 &\ddots &  \\ 
 \vdots &  & \ddots & 1 \\
0& 0& \hdots &t_n^2 \end{pmatrix}
\eeq
The image of $e_X^T$ consists of regular elements and we have 
 the following commutative diagram
\[\xymatrix{\ft\ar[r]^{e^T_X}\ar[d]&\fg_n\ar[d]\\
\fc_n\ar[r]^{id}&\fc_{n}}\]
It follows that we have a canonical isomorphism 
\[J_n\times_{\fc_n}\ft\is (e^T_X)^*I_n=(G_n\times\ft)^{e^T_X}\]
of groups schemes over $\ft$.
For any $t\in\ft$, 
we have 
\beq\label{}
\tau\circ e^T_X(t)=
\begin{pmatrix}0&I_n\\
C&0\end{pmatrix}
\ \ \ \ C=\begin{pmatrix} t_1^2 & 1 & \\ 
 \vdots& t_2^2 &\ddots &  \\ 
 \vdots &  & \ddots & 1 \\
0& 0& \hdots &t_n^2 \end{pmatrix}
\eeq
Note that the
elements $\tau\circ e^T_X(t)$ and $e^T(t)$ are regular and have the same characteristic polynomial
and hence lie in the same $G_{2n}$-orbit.
Pick an element $g_t\in G_{2n}$ such that 
\[e^T(t)=g_t(\tau\circ e^T_X(t)) g_t^{-1}\]
Then the conjugation map
$\on{Ad}_{g_t}:G_{2n}\to G_{2n}, g\to g_t g g_t^{-1}$
restricts to a map between  the centralizers
\[(G_n)^{e^T_X(t)}\stackrel{\delta}\to (G_{2n})^{e^T(t)}\stackrel{\on{Ad}_{g_t}}\to (G_{2n})^{e^T(t)}\]
Since  centralizers of a regular element is a commutative group, the map 
above is independent of the choice of the element $g_t$ and 
hene is canonical.
Then as $t$ varies over $\ft$, we obtain a map  between the corresponding centralizers group schemes.
\beq\label{base change to t}
J_n\times_{\fc_n}\ft\is (G_n\times\ft)^{e^T_X}\stackrel{}\to  (G_{2n}\times\ft)^{e^T}\is
J_{2n}\times_{\fc_X}\ft
\eeq
which 
is the map in~\eqref{pull-back}.

Alternatively, 
the assignment 
$t\to g_t$ gives rise to an element 
\[\Phi\in G_{2n}\otimes R_T\ \ \ \ \Phi(t)=g_t\]
and 
if we regard  the maps  $e^T$ and $\tau\circ e^T_X$ as elements 
in 
$\fg_{2n}\otimes R_T$
we have 
\beq\label{key property}
e^T=\Phi (\tau\circ e^T_X)\Phi^{-1}\in\fg_{2n}\otimes R_T.
\eeq
(we will give a canonical construction of 
the element $\Phi$, see Remark \ref{construction of Phi}).
Then the
composition
\[\on{Ad}_\Phi\circ\delta: G_{n}\times\ft\to G_{2n}\times\ft\ \ \ \ (g,t)\to \Phi(t)(\delta(g))\Phi^{-1}(t)\]
restricts to the map~\eqref{base change to t}
between the corresponding centralizers group schemes.

\subsection{Dual group}\label{Nadler dual group}
In \cite{Na}, the author associated to each 
real from 
$G_\bbR$ of a complex reductive group $G$,
equivalently a
symmetric space $X$ of $G$, 
a complex reductive  group 
$\check G_X$ together with
a homomorphism $\delta: \check G_X\to \check G$.
In the case $G=G_{2n}$ and $G_\bbR=G_{n,\bbH}$, equivalently $X=G_{2n}/\on{Sp}_{2n}$,
we have $\check G=G_{2n}$
and the homomorphism is the diagonal embedding 
\beq
\delta:G_n\to G_{2n}\ \ \ \ 
\delta(A)=\begin{pmatrix}
A&0\\
0&A
\end{pmatrix}
\eeq
Let $P=LN$ be the complexification of the minimal parabolic
$P_{n,\bbH}$. The Levi subgroup $L$ consisting of matrices of the form
\[L=\{\begin{pmatrix}
A&B\\
C&D
\end{pmatrix}\in G_{2n}\ |\ A,B,C,D\text{  are diagonal matrices}\}
\]
Consider the principal $\SL_2$ of $L$ given by 
\[\psi:\SL_2\to L\ \ \ \ \ \psi\begin{pmatrix}
a&b\\
c&d
\end{pmatrix}=\begin{pmatrix}
A&B\\
C&D
\end{pmatrix}\]
where $A=aI_n$, $B=bI_n$, etc.
The restriction of $\psi$ to the 
torus $\bG_m\subset\SL_2$ is the 
 co-character 
\[2\rho_{L}:\bG_m\to L \ \ \
\ \ \ \ 2\rho_{L}(h)=\on{diag}(h,...,h,h^{-1},...,h^{-1})\]
 corresponding to the 
sum of the positive roots of the Levi factor $L$.
A direct computation shows that the image $\psi(\SL_2)\subset G_{2n}$ centralizes the subgroup $\delta(G_n)\subset G_{2n}$
and hence we obtain a 
homomorphism
\beq\label{SL_2}
\psi_X:\check G_X\times\SL_2\to G_{2n}\ \ \ \ \ \psi_X(g,y)\to \delta(g)\psi(y)
\eeq

\subsection{The partial Whittaker reduction}
Consider the identification $T^*G_{2n}\is G_{2n}\times\fg_{2n}^*$
by considering $\fg_{2n}^*$ as left invariant differential forms on $G_{2n}$.
The group $G_{2n}\times G_{2n}$ acts 
on $G_{2n}$ via the left and right multiplication and the 
 induced action on $T^*G_{2n}\is G_{2n}\times\fg_{2n}^*$
 is given by 
 $(g,h)(x,v)=(gxh^{-1},\on{Ad}_hv)$.
 The moment map 
 $(\mu_l,\mu_r):T^*G_{2n}\to \fg_{2n}^*\times\fg_{2n}^*$ with respect to
 the $G_{2n}\times G_{2n}$-action 
 is given by $(\mu_l,\mu_r)(x,v)=(\on{Ad}_xv,-v)$.
 
 Consider the Shalika subgroup 
 $G_n\ltimes U$ and the generic morphism 
 $\psi$ in~\eqref{Shalika}.
 Let $\fg_n\times\fu$ be the Lie algebra of $G_n\ltimes U$.
 Then one can view $\psi$ as 
an element $\psi=(0,-\tr)$ in  
$\fg_n^*\times\fu^*$
\[\psi(\begin{pmatrix}
A&0\\
0&A
\end{pmatrix},\begin{pmatrix}
0&0\\
C&0
\end{pmatrix})=-\tr(C).\]
The moment map 
for the right $G_n\ltimes U$-action on $T^*G_{2n}$ is given by
\[\mu:T^*G_{2n}\stackrel{\mu_r}\to\fg_{2n}^*\to\fg_n^*\times\fu^*\]
where  $\mu_r$ is the right moment map above and the second map the natural restriction map.
The partial Whittaker reduction $\check M$ of 
$T^*G_{2n}$ with respect to the right $G_n\ltimes U$-action 
is given by 
\[\check M=T^*(G_{2n}/G_n\ltimes U,\psi):=\mu^{-1}(\psi)/G_n\ltimes U.\]

\begin{lemma}\label{partial reduction}
There is an isomorphism 
$\check M\is G_{2n}\times^{G_n}\fg_n$
fiting into the following commutative diagram
\[\xymatrix{\check M\ar[r]^{\simeq\ \ \ \ }\ar[d]^{}&G_{2n}\times^{G_n}\fg_n\ar[d]^{}\\
\fg_{2n}^*\ar[r]^{\simeq\ \ \ }&\fg_{2n}}\]
where  the left vertical arrow is the left moment map $\mu_l$,
the bottom arrow is induced by the
trace pairing 
$(A,B)\to\on{tr}(AB)$,
 and the
right vertical map is given by
$$(x,C)\mapsto \on{Ad}_x\begin{pmatrix}
0&I_n\\
C&0
\end{pmatrix}$$
\end{lemma}
\begin{proof}
We will identify $\fg_{2n}^*$ with $\fg_{2n}$
via the trace pairing.
 The pre-image of $\psi=(0,-\on{tr})\in\fg_n^*\times\fu^*$
in $\fg_{2n}^*\is\fg_{2n}$
is given by
\[\fg_{2n,\psi}^*:=\{\begin{pmatrix}
A&-I_n\\
C&-A
\end{pmatrix}|\  A,C\in\fg_n\}\]
and it follows that 
\[\check M\is\mu^{-1}(\psi)/G_n\ltimes U\is \mu_r^{-1}(\fg_{2n,\psi}^*)/G_n\ltimes U\is
G_{2n}\times^{G_n\ltimes U} (-\fg_{2n,\psi}^*).\]
(recall that $\mu_r(x,v)=-v$). On the other hand, a direct computation shows that the action of 
$U$ on $-\fg_{2n,\psi}^*$ is free and 
any $U$-orbit on $-\fg_{2n,\psi}^*$ contains 
a unique element of the form $\begin{pmatrix}
0&I_n\\
C&0
\end{pmatrix}$, $C\in\fg_n$.\footnote{Indeed, it follows from  
$\begin{pmatrix}
I_n&0\\
X&I_n
\end{pmatrix}\begin{pmatrix}
A&I_n\\
C&-A
\end{pmatrix}\begin{pmatrix}
I_n&0\\
-X&I_n
\end{pmatrix}=\begin{pmatrix}
A-X&I_n\\
C+XA+AX-X^2&X-A
\end{pmatrix}$} Thus there is an isomorphism 
\[\check M\is G_{2n}\times^{G_n\ltimes U} (-\fg_{2n,\psi}^*)\is
G_{2n}\times^{G_n}\fg_n.\]
such that the left moment map is given by
$\mu_l(x,C)=\on{Ad}_x\begin{pmatrix}
0&I_n\\
C&0
\end{pmatrix}$. The lemma follows.
\end{proof}

\section{Constructible side}\label{constructible side}

\subsection{Twistor fibration}\label{twistor}
Consider the complex projective space 
$\mathbb P^{2n-1}$ and the quaternionic projective space $\mathbb{HP}^{n-1}$.
Recall the identification 
$\bC^{2n}\is\mathbb H^n$
sending 
\[(z,w)=(z_1,...,z_n,w_1,...,w_n)\ra z+jw=(q_1=z_1+jw_1,...,q_n=z_n+jw_n)\] 
If to each complex line in $\bC^{2n}\is\mathbb H^n$ we associate the quaternion line it generates,
we get a map
\beq\label{twistor}
f:\mathbb P^{2n-1}\to\mathbb{HP}^{n-1}\ \ \ [z,w]\to [q_1,...,q_n]
\eeq
between 
the corresponding complex and quaternionic projective spaces, to be called 
the \emph{twistor fibration} for $\mathbb{HP}^{n-1}$.
The fiber of $f$ over
a quaternion line (a copy of $\bbH\is\bC^2$) consists of all complex line in it which is a copy of $\mathbb P^1\is S^2$.
Thus the twistor fibration $f$ is a fiber bundle with fiber $\mathbb P^1$.
In the case $n=2$, we have $\mathbb{HP}^{n-1}=\mathbb{HP}^{1}\is S^4$
and the map~\eqref{twistor} is the well known twistor fibration
\[f:\mathbb P^3\to S^4\]
for $S^4$.

Consider the standard
action of the
complex torus $T_{2n}$ (resp. $T_n$)
on $\mathbb P^{2n-1}$ (resp. $\mathbb{HP}^{n-1}$):
\[x\cdot[z_1,...,z_{2n}]=[x_1z_1,...,x_{2n}z_{2n}]\ \ \ \ x=(x_1,...,x_{2n})\in T_{2n}\]
\[(\on{resp}.\ \ \ \  x\cdot[q_1,...,q_{n}]=[x_1q_1,...,x_{n}q_{n}]\ \ \ \ x=(x_1,...,x_n)\in T_n).\]
Then
the twistor map $f:\mathbb P^{2n-1}\to\mathbb {HP}^{n-1}$
is $T_n$-equivariant   where 
$T_n$ acts on $\mathbb P^{2n-1}$
through the embedding 
\[T_n\stackrel{\sim}\lra T_{n,\bbH}\subset T_{2n}\ \ \ \ (x_1,...,x_n)\to (x_1,...,x_n,\bar x_1,...,\bar x_n).\]
(Recall that $T_{n,\bbH}$ is the Cartan subgroup of 
$G_{n,\bbH}$)
Indeed,
for any $x=(x_1,...,x_n)\in T_n$, 
 we have 
\[f(x\cdot[z,w])=f([x_1z_1,...,x_{n}z_{n},\bar x_1w_1,...,\bar x_nw_n])=
[x_1z_1+j\bar x_1w_1,...,x_nz_n+j\bar x_nw_n]=\]
\[=[x_1z_1+x_1jw_1,...,x_nz_n+x_njw_n]=x\cdot[q_1,...,q_n].\]

\subsection{Equivariant cohomology of quaternionic projective spaces}
Consider the inverse action of 
$T_{2n}$  on 
$\mathbb P^{2n-1}$.
\footnote{The reason to consider the inverse of the standard 
action will become clear later, see the proof of Lemma \ref{computation}.}
Recall the following well-known description of the $T_{2n}$-equivariant cohomology
of $\mathbb{P}^{2n-1}$:
\beq
H^*_{T_{2n}}(\mathbb P^{2n-1})\is\bC[t_1,...,t_{2n}][\xi]/\prod_{i=1}^{2n}(\xi-t_i)
\eeq
where 
\[\xi=c_1^{T_{2n}}(\mO(1))\in H^*_{T_{2n}}(\mathbb P^{2n-1})\] is the first equivariant Chern class of the line bundle
$\mO(1)$ over $\mathbb P^{2n-1}$
and $H^*_{T_{2n}}(\on{pt})\is\mO(\ft_{2n})\is\bC[t_1,...,t_{2n}]$.

The imbedding $T_n\is T_{n,\bbH}\subset T_{2n}$
gives rise to a map $H^*_{T_{2n}}(\on{pt})\to H^*_{T_{n}}(\on{pt})$
and a direct computation show that, under the isomorphism 
$\bC[t_1,...,t_{2n}]\is H^*_{T_{2n}}(\on{pt})$
and $\bC[t_1,...,t_{n}]\is H^*_{T_{n}}(\on{pt})$, the map is given by
\[\bC[t_1,...,t_{2n}]\to\bC[t_1,...,t_{2n}]/(t_1+t_{n+1},t_2+t_{n+2},...,t_n+t_{2n})\is\bC[t_1,...,t_n]\]
It follows that 
\beq\label{description of equ coh}
H^*_{T_{n}}(\mathbb P^{2n-1})\is H^*_{T_{2n}}(\mathbb P^{2n-1})\otimes_{H^*_{T_{2n}}(\on{pt})}H^*_{T_{n}}(\on{pt})\is
\bC[t_1,...,t_n][\xi]/\prod_{i=1}^{n}(\xi^2-t^2_i).
\eeq

Similarly, we consider the inverse $T_n$-action on $\mathbb{HP}^{n-1}$.
Let $\mO_{\mathbb H}(-1)$ be the  tautological $\mathbb H$-line bundle 
$\mO_{\mathbb H}(-1)$  over $\mathbb{HP}^{n-1}$.
It is canonical $T_n$-equivariant and we denote by
 \[\eta=-e^{T_n}(\mO_{\mathbb H}(-1))\in H^4_{T_n}(\mathbb{HP}^{n-1})\]
the negative of the equivariant Euler class of $\mO_{\mathbb H}(-1)$.

\begin{lemma}\label{description of equ coh H}
There is an isomorphism
\[H_{T_n}^*(\mathbb {HP}^{n-1})\is\bC[t_1,...t_n][\eta]/\prod_{i=1}^n(\eta-t_i^2)\]
making the following diagram commutes
\[\xymatrix{H_{T_n}^*(\mathbb {HP}^{n-1})\ar[r]^{f^*}\ar[d]^{\simeq}&H_{T_n}^*(\mathbb {P}^{2n-1})\ar[d]^{\simeq}\\
\bC[t_1,...t_n][\eta]/\prod_{i=1}^n(\eta-t_i^2)\ar[r]&\bC[t_1,...t_n][\xi]/\prod_{i=1}^n(\xi^2-t_i^2)}\]
where the bottom arrow is the natural $\bC[t_1,...t_n]$-linear embedding sending 
$\eta$ to $\xi^2$, that is, we have 
$f^*(\eta)=\xi^2$.
\end{lemma}
\begin{proof}
The $T_n$-fixed points on $\mathbb{HP}^{n-1}$ are 
$p_0=[1,0,...,0]_\bbH$, $p_2=[0,1,0,...,0]_\bbH$,...,$p_n=[0,0,...,0,1]_\bbH$.
Write $s_i:\{p_i\}\to \mathbb{HP}^{n-1}$ for the inclusion map.
Then the equivariant localization says that we have an injective map of rings
\[Loc=\bigoplus s^*_i:H_{T_n}^*(\mathbb{HP}^{n-1})\lra R_{T_n}^{\oplus n}\]
The fiber of 
$\mO_{\mathbb H}(-1)|_{p_i}$ over $p_i$ 
is  the $\mathbb H$-line spanned by the i-th coordinate vector of 
$\mathbb H^n$ and hence the action 
 of $T_n$ factors though the i-th projection 
 $T_n\to \bG_m$, $(x_1,...,x_n)\to x_i$.
 It follows that, in terms of the coordinate 
 $\bC^2\is\mO_{\mathbb H}(-1)|_{p_i}, (z_i,w_i)\to z_i+jw_i$ (and hence a chosen orientation) the (inverse) action is given by
 $x_i (z_i,w_i)=(x_i^{-1}z_i,\bar x_i^{-1}w_i)$) and hence 
\[s_i^*(\eta)=
s_i^*(-e^{T_n}(\mO_{\mathbb H}(-1)))=
-e^T(\mO_{\mathbb H}(-1)|_{p_i})=t_i^2\]
Thus we have 
\[Loc(\eta)=Loc(-e^T(\mO_{\mathbb H}(-1)))=(t_1^2,...,t_n^2)\in R_T^{\oplus n}\]
 and it follows that  $Loc(\prod_{i=1}^n(\eta-t_i^2))=0$ and, as $Loc$ is injective,
 it implies $\prod_{i=1}^n(\eta-t_i^2)=0$.

To see $f^*(\eta)=\xi^2$, we observe that the 
preimage 
$f^{-1}(p_i)$ is isomorphic to the projection line 
$\mathbb P^1_i=[z_i,w_i]\subset\mathbb{P}^{2n-1}$.
The $T_n$-action  preserves $\mathbb P^1_i$
and is given by
$(x_1,...,x_n)[z_i,w_i]=[x_i^{-1}z_i,\bar x_i^{-1}w_i]$. 
The localization map
$Loc:H_{T_n}^*(\mathbb P^{2n-1})\to \bigoplus H_{T_n}^*(\mathbb P^{1}_i)=\bC[t_i][\xi_i]/(\xi_i^2-t_i^2)$ 
is injective and 
we have 
$Loc(\xi^2)=(\xi^2_1,...,\xi^2_n)$. On the other hand, 
we have 
\[Loc(f^*\eta)=f^*(Loc(\eta))=f^*((t_1^2,...,t^2_n))=
(t_1^2,...,t^2_n)=
(\xi^2_1,...,\xi^2_n)\in\bigoplus H_{T_n}^*(\mathbb P^{1}_i)\]
as $\xi^2_i=t_i^2$ in $H_{T_n}^*(\mathbb P^{1}_i)$.
We conclude that $Loc(\xi^2)=Loc(f^*\eta)$
and hence $\xi^2=f^*\eta$.

\end{proof}

\begin{remark}
Here is an alternative argument. One can show that 
 there is an isomorphism of 
 $T_n$-equivariant complex vector bundles 
 \[f^*\mO_{\mathbb H}(-1)\is \mO(-1)\oplus\overline{\mO(-1)}\]
 over $\mathbb P^{2n-1}$.
Here $\overline{\mO(-1)}$ is the complex conjugate of $\mO(-1)$ (note that a choice of a hermitian metric on $\mO(-1)$ induces an isomorphism 
 $\overline{\mO(-1)}\is\mO(-1)^\vee\is\mO(1)$).
 Since $e^T(\overline{\mO(-1)})=-e^T(\mO(-1))=-\xi$, it follows that 
\[f^*(\eta)=-f^*(e^T(\mO_\mathbb H(-1))=
-e^T(\mO(-1)\oplus\overline{\mO(-1)})= e^T(\mO(-1))^2=\xi^2.\]
Now the lemma follows from the fact that  
$f^*:H_{T_n}^*(\mathbb {HP}^{n-1})\to H_{T_n}^*(\mathbb {P}^{2n-1})$ is injective 
and
$f^*(\prod_{i=1}^n(\eta-t_i^2))=\prod_{i=1}^n(f^*\eta-t_i^2)=\prod_{i=1}^n(\xi^2-t_i^2)=0$
in $H_{T_n}^*(\mathbb {P}^{2n-1})$.

\end{remark}

Consider the push-forward functor 
$f_*:D_{T_n}^b(\mathbb{P}^{2n-1})\to D_{T_n}^b(\mathbb{HP}^{n-1})$.

\begin{lemma}\label{splitting}
$f_*(\bC_{\mathbb{P}^{2n-1}})\is\bC_{\mathbb{HP}^{n-1}}\oplus\bC_{\mathbb{HP}^{n-1}}[-2]$
\end{lemma}
\begin{proof}
Since $f$ is a $\mathbb P^1$-fibration we have a distinguished triangle
\[\bC_{\mathbb{HP}^{n-1}}\to f_*(\bC_{\mathbb{P}^{2n-1}})\to\bC_{\mathbb{HP}^{n-1}}[-2]\to
\bC_{\mathbb{HP}^{n-1}}[1]\]
and we need to show that it splits. But 
this follows from  
\[\Hom(\bC_{\mathbb{HP}^{n-1}}[-2],\bC_{\mathbb{HP}^{n-1}}[1])\is
\on{Ext}^3(\bC_{\mathbb{HP}^{n-1}},\bC_{\mathbb{HP}^{n-1}})\is H^3_{T_n}(\mathbb{HP}^{n-1})=0\]
\end{proof}

\subsection{Two bases}\label{two bases}
Consider the subvarieties $\mathbb P^{i-1}=\{[z_1,...,z_i,0,...,0]\}\subset\mathbb P^{2n-1}$, $i=1,...,2n$. If we write 
$[\mathbb P^{i-1}]\in H_{2i-2}^{T_{2n}}(\mathbb P^{2n-1})\is H^{4n-2i}_{T_{2n}}(\mathbb P^{2n-1})$
for the corresponding  fundamental class in the equivariant Borel-Moore homology, the the collection 
$\{[\mathbb P^{i-1}]\}_{i=1,...,2n}$ forms a basis of the free $R_{T_{2n}}$-module
$H^{*}_{T_{2n}}(\mathbb P^{2n-1})$. 
Moreover, one can check the image of the fundamental class $[\mathbb P^{i-1}]$ under the 
the identification 
\eqref{description of equ coh}
is given by
\[\Upsilon:H^{*}_{T_{2n}}(\mathbb P^{2n-1})\is\bC[t_1,...,t_{2n}][\xi]/\prod_{i=1}^{2n}(\xi-t_i)\]
\[\Upsilon([\mathbb P^{2n}])=1\ \ \ \ 
\Upsilon([\mathbb P^{i-1}])=\prod_{s=i+1}^{2n}(\xi-t_s)\ \ \ i=1,...,2n-1\]
Consider the subvarieties $\mathbb{HP}^{i-1}=\{[q_1,...,q_i,0,...,0]\}\subset\mathbb {HP}^{n-1}$, $i=1,...,n$. If we write 
$[\mathbb{HP}^{i-1}]\in H_{4i-4,T_n}(\mathbb{HP}^{n-1})\is H^{4n-4i}_{T_n}(\mathbb P^{2n-1})$
for the corresponding  fundamental class in the equivariant Borel-Moore homology, the the collection 
$\{[\mathbb{HP}^{i-1}]\}_{i=1,...,n}$ forms a basis of the free $R_{T_n}$-module
$H^{*}_{T_n}(\mathbb{HP}^{n-1})$. Moreover,
one can check that the image of 
the fundamental class $[\mathbb{HP}^{i-1}]$ under the identification
in~\eqref{description of equ coh H} is given by
\[\Upsilon_\bbH:H^{*}_{T_n}(\mathbb{HP}^{n-1})=\bC[t_1,...,t_n][\eta]/\prod_{i=1}^n(\eta-t_i^2)\]
\[\Upsilon_\bbH([\mathbb{HP}^{n-1}])=1\ \ \ \ 
\Upsilon_\bbH([\mathbb{HP}^{i-1}])=\prod_{s=i+1}^{n}(\eta-t^2_s)\ \ \ i=1,...,n-1\]

The isomorphism $f_*\bC_{\mathbb P^{2n-1}}\is\bC_{\mathbb{HP}^{n-1}}\oplus
\bC_{\mathbb{HP}^{n-1}}[-2]$, 
gives rise to a decomposition 
\[\Upsilon':H^*_{T_n}(\mathbb{HP}^{n-1})\oplus
H^{*-2}_{T_n}(\mathbb{HP}^{n-1})\is H^*_{T_n}(\mathbb P^{2n-1})\is \bC[t_1,...,t_n][\xi]/\prod_{i=1}^n(\xi^2-t_i^2)\]
and 
one can check that the image of the 
basis $\{[\mathbb{HP}^{i-1}]\}\cup\{[\mathbb{HP}^{i-1}][2]\}$ of $H^*_{T_n}(\mathbb{HP}^{n-1})\oplus H^{*-2}_{T_n}(\mathbb{HP}^{n-1})$
under the map above are 
\beq
\Upsilon'([\mathbb{HP}^{n-1}])=1,\ \ 
\Upsilon'([\mathbb{HP}^{i-1}])=\prod_{s=i+1}^{n}(\xi^2-t^2_s)\ \ \ i=1,...,n-1
\eeq
\[\Upsilon'([\mathbb{HP}^{n-1}[2]])=\xi,\ \ 
\Upsilon'([\mathbb{HP}^{i-1}[2]])=\xi\prod_{s=i+1}^{n}(\xi^2-t^2_s)\ \ \ i=1,...,n-1
.\]
\begin{lemma}\label{matrix presentation}
(1) In terms of the ordered basis $\{[\mathbb P^{0}],[\mathbb P^{1}],...,[\mathbb P^{2n}]\}$,
the 
the  cup product action $c_1^{T_{2n}}(\mO(1))\cup(-)\in\End_{R_{T_n}} (H^*_{T_n}(\mathbb P^{2n-1}))$
is given by
the element $e^{T_{2n}}$ in~\eqref{e^T_2n}:
\[e^{T_{2n}}=\begin{pmatrix} t_1 & 1 & \\ 
 0& t_2 &\ddots &  \\ 
 \vdots &  & \ddots & 1 \\
0&  \hdots&0 &t_{2n} \end{pmatrix}
\]
(2)
In terms of the ordered basis
 $\{[\mathbb{HP}^{0}][2],...,[\mathbb{HP}^{n-1}][2],
 [\mathbb{HP}^{0}],..,[\mathbb{HP}^{n-1}]\}$,
the cup product action $c_1^{T_{n}}(\mO(1))\cup(-)\in\End_{R_{T_n}} (H^*_{T_n}(\mathbb P^{2n-1}))$
is given by the
element $\tau\circ e^T_X$ in~\eqref{e^T_X}:
\[
\tau\circ e^T_X=\begin{pmatrix}0&I_n\\
C&0\end{pmatrix}
\ \ \ \ C=\begin{pmatrix} t_1^2 & 1 & \\ 
 \vdots& t_2^2 &\ddots &  \\ 
 \vdots &  & \ddots & 1 \\
0& 0& \hdots &t_n^2 \end{pmatrix}
\]

\end{lemma}
\begin{proof}
The cup product action is given by mutiplication by $\xi$ and the claim 
is a straightforward computation.
\end{proof}

\subsection{Complex and quaternionic affine Grassmannians}
\label{affine grassmannian}
We denote by
$\Gr_{2n}=\frak LG_{2n}/\frak L^+G_{2n}$
the complex affine grassmannian for $G_{2n}$
where $\frak LG_{2n}=G_{2n}(\bC((t)))$
and $\frak L^+G_{2n}=G_{2n}(\bC[[t]])$
are the Laurent loop group and Taylor arc group for $G_{2n}$ respectively.
We denote by 
$D^b(\frak L^+G_{2n}\backslash\Gr_{2n})$
the dg category of $\frak L^+G_{2n}$-equivariant constructible complexes
on $\Gr_{2n}$
and 
$\on{Perv}(\Gr_{2n})$ the abelain category of 
$\frak L^+G_{2n}$-equivariant perverse sheaves on $\Gr_{2n}$.

We denote by
$\Gr_{n,\bbH}=\frak LG_{n,\bbH}/\frak L^+G_{n,\bbH}$ 
the real affine Grassmannian 
for the quaternionic group $G_{n,\bbH}$ where $\frak LG_{n,\bbH}=G_{n,\bbH}(\bbR((t)))$ and $\frak L^+G_{n,\bbH}=G_{n,\bbH}(\bbR[[t]])$ are the real Laurent loop group and real Taylor arc group for 
$G_{n,\bbH}$. 
The $\frak L^+G_{n,\bbH}$-orbits on $\Gr_{n,\bbH}$ are of the form 
$\Gr^\lambda_{n,\bbH}=\frak L^+G_{n,\bbH}\cdot t^\lambda$ where $(\lambda:\bG_m\to S)\in\Lambda_S^+$ is a dominant real coweight. 
By \cite[Proposition 3.6.1]{Na}, each orbit $\Gr^\lambda_{n,\bbH}$ is a real vector bundle over 
the quaternionic flag manifold $G_{n,\bbH}/P^\lambda_{n,\bbH}$ of real dimension $2\langle\lambda,\rho_{2n}\rangle$. 
We denote by 
$D^b(\frak L^+G_{n,\bbH}\backslash\Gr_{n,\bbH})$
the dg category of $\frak L^+G_{n,\bbH}$-equivariant constructible complexes
on $\Gr_{n,\bbH}$.
Since  
$2\langle\lambda,\rho_{2n}\rangle=4\langle\lambda,2\rho_n\rangle\in 4\bbZ$ for all $\lambda\in\Lambda_S^+$
(in the second paring we regard $\lambda$ as element in $\Lambda_n$), all the orbits $\Gr^\lambda_{n,\bbH}$
have real even dimension, and
hence middle perversity makes sense
and we denote by 
$\on{Perv}_{}(\Gr_{n,\bbH})$ the category $\frak L^+G_{n,\bbH}$-equivariant perverse sheaves on $\Gr_{n,\bbH}$.
Note also that, as $P^\lambda_{n,\bbH}$ is connected, all the 
$G_{n,\bbH}$-equivariant local system on $\on{Perv}_{}(\Gr_{n,\bbH})$ are trivial and hence the irreducible objects in $\on{Perv}_{}(\Gr_{n,\bbH})$ are 
intersection cohomology complexes 
$\IC_\lambda=\IC(\overline{\Gr^\lambda_{n,\bbH}})$, $\lambda\in\Lambda_S^+$ for the closure $\overline{\Gr^\lambda_{n,\bbH}}\subset\Gr_{n,\bbH}$.

Like in the case of complex reductive groups,
there is a natural monoidal structure on 
$D^b(\frak L^+G_{n,\bbH}\backslash\Gr_{n,\bbH})$
given by the convolution product:
consider the convolution diagram
\[\Gr_{n,\bbH}\times\Gr_{n,\bbH}\stackrel{p}\la
\frak LG_{n,\bbH}\stackrel{q}\ra
\Gr_{n,\bbH}\tilde\times\Gr_{n,\bbH}:=\frak LG_{n,\bbH}\times^{\frak L^+G_{n,\bbH}}\Gr_{n,\bbH}\stackrel{m}\to\Gr_{n,\bbH}\]
where $p$ and $q$ are the natural quotient maps and
 $m(x,y\mod\frak L^+G_{n,\bbH})= xy\mod\frak L^+G_{n,\bbH}$.
 For any $\mF_1,\mF_2\in D^b(\frak L^+G_{n,\bbH}\backslash\Gr_{n,\bbH})$,
 the convolution 
 is defined as 
 \[\mF_1\star\mF_2=m_!(\mF_1\tilde\boxtimes\mF_2)\]
 where $\mF_1\tilde\boxtimes\mF_2\in D^b(\frak L^+G_{n,\bbH}\backslash\Gr_{n,\bbH}\tilde\times\Gr_{n,\bbH})$ is the unique complex
 such that $q^*(\mF_1\tilde\boxtimes\mF_2)\is p^*(\mF_1\boxtimes\mF_2)$.

\subsection{Real nearby cycles functor}\label{Real nearby cycles}
We shall recall the construction of the real nearby cycles functor 
in \cite{Na}.
Consider 
the Beilinson-Drinfeld Grassmannian 
$\Gr^{(2)}_{2n}\to\bC$ over the complex line $\bC$ classifying a $G_{2n}$-bundle 
$\mE\to\bC$, a point $x\in\bC$, and a section 
$\nu:\bC\setminus\{\pm x\}\to\mE|_{\bC\setminus\{\pm x\}}$.
It is well-known that there are canonical isomorphisms 
\[\Gr^{(2)}_{2n}|_{\{0\}}\is\Gr_{2n}\]
\[\Gr^{(2)}_{2n}|_{\bC\setminus\{0\}}\is\Gr_{2n}\times\Gr_{2n}\times\bC\setminus\{0\}\]
It is shown in \cite{Na} that 
the real form $G_{n,\bbH}$ of $G_{2n}$ together with real form 
$i\bbR$ 
of $\bC$ (corresponding to the complex conjugation $x\to -\bar x$
on $\bC$) defines a real form
$\Gr_{n,\bbH}^{(2)}\to i\bbR$ of $\Gr_{2n}^{(2)}$ such that 
there are canonical isomorphisms 
\[\Gr^{(2)}_{n,\bbH}|_{\{0\}}\is\Gr_{n,\bbH}\]
\[\Gr^{(2)}_{n,\bbH}|_{i\bbR\setminus\{0\}}\is\Gr_{2n}\times (i\bbR\setminus\{0\})\]
Consider the following diagram 
\[\xymatrix{\Gr_{2n}\times i\bbR\ar[r]^{\simeq}&\Gr^{(2)}_{n,\bbH}|_{i\bbR_{>0}}\ar[r]^j\ar[d]&\Gr^{(2)}_{n,\bbH}|_{i\bbR_{\geq 0}}\ar[d]&\Gr^{(2)}_{n,\bbH}|_{\{0\}}\ar[d]\ar[l]_i&\Gr_{n,\bbH}\ar[l]_{\ \ \simeq}\\
&i\bbR_{>0}\ar[r]&i\bbR_{\geq 0}&\{0\}\ar[l]&}\]
Note that the maps in the above diagram are all $K_c$-equivariant and we 
define the functor
\[\label{real nearby cycle}
\mathrm R': D^b_{}(K_c\backslash\Gr_{2n})\to D^b_{}(K_c\backslash\Gr_{n,\bbH})
\]
by the formula
\beq\label{formula for real nearby cycle}
\mathrm R'(\mF)=i^*j_*(\mF\boxtimes\bC_{i\bbR_{\geq 0}}).
\eeq
By \cite[Proposition 4.5.1]{Na}, the functor $\mathrm R'$ takes
$\frak L^+G_{2n}$-constructible complexes to 
$\frak L^+G_{n,\bbH}$-constructible complexes.
Introduce the subcategory $D^b_{\{\frak L^+G_{n,\bbH}\}}(K_c\backslash\Gr_{n,\bbH})$ (resp. $D^b_{\{\frak L^+G_{2n}\}}(G_c\backslash\Gr_{2n})$ and $D^b_{\{\frak L^+G_{2n}\}}(K_c\backslash\Gr_{2n})$)
of $D^b(K_c\backslash\Gr_{n,\bbH})$
(resp. $D^b(G_c\backslash\Gr_{2n})$ and $D^b(K_c\backslash\Gr_{2n})$)
consisting $\frak L^+G_{n,\bbH}$-constructible complexes (resp. 
$\frak L^+G_{2n}$-constructible complexes).
We have natural equivalence $D^b_{\{\frak L^+G_{n,\bbH}\}}(K_c\backslash\Gr_{n,\bbH})\is D^b(\frak L^+G_{n,\bbH}\backslash\Gr_{n,\bbH})$,
$D^b_{\{\frak L^+G_{2n}\}}(G_c\backslash\Gr_{2n})\is D^b(\frak L^+G_{2n}\backslash\Gr_{2n})$
and the nearby cycles functor $\mathrm R'$ above induces a functor
\beq\label{R'}
\mathrm R':D^b_{\{\frak L^+G_{2n}\}}(K_c\backslash\Gr_{2n})\to
D^b_{\{\frak L^+G_{n,\bbH}\}}(K_c\backslash\Gr_{n,\bbH})\is D^b(\frak L^+G_{n,\bbH}\backslash\Gr_{n,\bbH})
\eeq
Finally, the real nearby cycles functor is defined as 
\beq\label{real nearby cycle}
\mathrm R: D^b_{}(\frak L^+G_{2n}\backslash\Gr_{2n})\is D^b_{\{\frak L^+G_{2n}\}}(G_c\backslash\Gr_{2n})\to 
D^b_{\{\frak L^+G_{2n}\}}(K_c\backslash\Gr_{2n})\stackrel{\mathrm R'}\to
D^b(\frak L^+G_{n,\bbH}\backslash\Gr_{n,\bbH})
\eeq
where the 
middle arrow is the natural forgetful functor.

The following properties of $\on{Perv}_{}(\Gr_{n,\bbH})$
and $\mathrm R$ can be deduced from \cite{Na}:

\begin{prop}\label{abelian Satake}
\begin{enumerate}
\item There is a tensor equivalence 
$\on{Rep}(G_n)\is\on{Perv}_{}(\Gr_{n,\bbH})$
sending irreducible representation $V_\lambda$ of $G_n$ with highest weight 
$\lambda\in\Lambda_S^+$ to $\IC_\lambda$.
\item The real nearby cycle functor $\mathrm R$ preserves semi-simplicity, that is,
we have  
\[\mathrm R(\mF)\is\bigoplus_{n\in\bZ} {^pH^n}\mathrm R(\mF)[-n]\]
for any semisimple complexe $\mF$ in $D_{}^b(\frak L^+G_{2n}\backslash\Gr_{2n})$.
\item
Consider the monoidal subcategory 
\[\on{Perv}(\Gr_{n,\bbH})_\bbZ:=\bigoplus_{n\in\bZ}\on{Perv}(\Gr_{n,\bbH})[n]\subset D_{\frak L^+G_{n,\bbH}}^b(\Gr_{n,\bbH}).\]
The real nearby cycle functor restricts to a monoidal functor
\[^p\mathrm R:\on{Perv}(\Gr_{2n})\to\on{Perv}(\Gr_{n,\bbH})_\bbZ\]
such that there is a commutative diagram
\[\xymatrix{\on{Perv}(\Gr_{2n})\ar[r]^{^p\mathrm R\ \ \ }\ar[d]^{\simeq}&\on{Perv}(\Gr_{n,\bbH})_\bbZ\ar[d]^{\simeq}\\
\on{Rep}(G_{2n})\ar[r]&\on{Rep}(G_n\times\bG_m)}\]
where the vertical tensor equivalences come from the
complex and quaternionic Satake isomorphisms (part (1)) and the bottom arrow 
is the restriction map to the subgroup $G_n\times\bG_m\subset G_{2n}$.
\end{enumerate}
\end{prop}
\begin{proof}
Expect part (2), all the other claims are in \cite{Na}.
To prove part (2), it suffices to show that 
$\mathrm R(\IC_\lambda)$ is semisimple for all dominant $\lambda$.
It is shown \cite{Na} that  $\mathrm R$ is monoidal and 
given two semisimple objects $\mF_1,\mF_1$ in the essential image 
of $\mathrm R$, the convolution $\mF_1\star\mF_2$ is again semisimple.
Since $\on{Rep}(G_{2n})$ is tensor generated by
the standard representation
$V_{\omega_1}=\bC^{2n}$ and the determinant character 
$\on{det}_{2n}$, it suffices to show that 
$\mathrm R(\IC_{\omega_1})$
and $\mathrm R(\IC_{\det_{2n}})$ are semisimple.
We have $\mathrm R(\IC_{\det})\is\IC_{\det^{\otimes 2}}$
where $\det_n^{\otimes 2}$ is the square of the determinant character of 
$G_n$ (double check) and hence is simple.
On the other hand, 
$\mathrm R(\IC_{\omega_1})$ admits a filtration with associated graded
given by $\IC_{\omega_1}[1]\oplus\IC_{\omega_1}[-1]$.
where $\IC_{\omega_1}$
is the $\IC$-complex of $\Gr_{n,\bbH}^{\omega_1}\is\mathbb{HP}^{n-1}$.
 Since $\on{Ext}^1(\IC_{\omega_1}[-1],\IC_{\omega_1}[1])=
 \on{Ext}^3(\bC_{\mathbb{HP}^{n-1}},\bC_{\mathbb{HP}^{n-1}})\is
 H^3_{G_{n,\bbH}}(\mathbb{HP}^{n-1})
 \subset H^3_{T_n}(\mathbb{HP}^{n-1})=0$, it follows that the filtration splits and 
 hence $\mathrm R(\IC_{\omega_1})\is\IC_{\omega_1}[1]\oplus\IC_{\omega_1}[-1]$
 is semisimple.

\end{proof}
In the course of the proof, together with Lemma \ref{splitting}, we show that 
\begin{corollary}
There is an isomorphism 
$\mathrm R(\IC_{\omega_1})\is f_*(\IC_{\mathbb P^{2n-1}})\is 
\IC_{\mathbb{HP}^{n-1}}[1]\oplus\IC_{\mathbb{HP}^{n-1}}[-1]
$
\end{corollary}

\begin{remark}
In Theorem \ref{spectral nearby cycle}, we will give a spectral description of nearby cycle functor $\mathrm R$
on the whole derived categories (not just its restriction $^p\mathrm R$ to the subcategory of perverse sheaves).
\end{remark}

Recall the nearby cycles functor $\mathrm R':
D_{\{\frak L^+G_{2n}\}}^b(K_c\backslash\Gr_{2n})\to D^b_{\{\frak L^+G_{n,\bbH}\}}(
K_c\backslash\Gr_{n,\bbH})$ in~\eqref{R'}.
It extends to the ind-completion (denoted again by $\mathrm R'$)
\[\mathrm R':
\Ind\ D_{\{\frak L^+G_{2n}\}}^b(K_c\backslash\Gr_{2n})\to \Ind\ D^b_{\{\frak L^+G_{n,\bbH}\}}(
K_c\backslash\Gr_{n,\bbH}).\]
\begin{lemma}\label{left adjoint}
The functor $\mathrm R'$ admits left adjoint
\[^L\mathrm R':  \Ind\ D^b_{\{\frak L^+G_{n,\bbH}\}}(
K_c\backslash\Gr_{n,\bbH})
\to\Ind\ D_{\{\frak L^+G_{2n}\}}^b(K_c\backslash\Gr_{2n}) \]
Moreover, we have $^L\mathrm R'(\bC_{\Gr_{n,\bbH}})\is\bC_{\Gr_{2n}}$.
\end{lemma}
\begin{proof}
By \cite[Proposition 4.5.1]{Na}, the ind-proper family  
$f:\Gr_{n,\bbH}^{(2)}\to i\bbR_{\geq0}$ is a Thom stratified map
with respect to a Whitney stratification $\calT$ on $\Gr_{n,\bbH}^{(2)}$
and the stratification $i\bbR_{>0}\cup\{0\}$ on $i\bbR_{\geq0}$ 
such that $\calT$ restricts to the 
 $\frak L^+G_{2n}$-orbits stratification on the 
generic fiber $\Gr_{2n}$ and to the 
$\frak L^+G_{n,\bbH}$-obits stratification
on the special fiber 
$\Gr_{n,\bbH}$. 
The construction in \cite[Section 6]{GM} 
together with the results in \cite[Theorem 1.1]{PW}
(extending 
Mather's theory of control data to the 
the equivariant setting)
implies that 
 the nearby cycles functor $\mathrm R'$  is isomorphic to the functor given by $*$-push-forward 
along a $K_c$-equivariant specialization map 
$\psi:\Gr_{2n}\to\Gr_{n,\bbH}$, and hence 
admits left adjoint
given by $*$-pull-back $\psi^*$.
It is clear that  $\psi^*$ sends constant sheaf to constant sheaf.
The lemma follows.

\end{proof}

\quash{
\begin{remark}
Here is an alternative argument.
Since $\mathrm R(\mF)=i^*j_*(\mF'\boxtimes\bC_{i\bbR_{>0}})$
and both the forgetful functor $\mF\to\mF'$ and 
the functor $j_*$ admit left adjoint 
it suffices to show that the functor 
$i^*:\on{Ind}D^b(\Gr_{n,\bbH}^{(2)}|_{i\bbR_{\geq0}})\to\on{Ind}D^b(\Gr_{n,\bbH})$
admits a left adjoint
on object of the form $j_*(\mF'\boxtimes\bC_{i\bbR_{>0}})$.
By \cite[Proposition 4.5.1]{Na}, the family  
$\Gr_{n,\bbH}^{(2)}\to i\bbR_{\geq0}$ is a Thom stratified map
with respect to a Whitney stratification $\calT$ on $\Gr_{n,\bbH}^{(2)}$
and the stratification $i\bbR_{>0}\cup\{0\}$ on $i\bbR_{\geq0}$ 
such that the induced stratification $\calT|_{0}$ on the special fiber 
$\Gr_{n,\bbH}$ is the $\frak L^+G_{n,\bbH}$-obits stratification.
Now the 
 existence of left adjoint of  follows from \cite[A.1.11 Lemma]{NY}.
\end{remark}
}

\subsection{Equivariant homology and cohomology
of affine Grassmannians}
\subsubsection{}
We review the description of the 
equivariant homology $H_*^{T_{2n}}(\Gr_{2n})$ 
and $H_*^{T_c}(\Gr_{n,\bbH})$
of $\Gr_{2n}$ and $\Gr_{n,\bbH}$ in \cite{O,YZ}.
Recall that for an ind-proper 
semi-analytic set $Y=\underset{I}{\on{colim}}\ Y_i$
acting real analytically by a Lie group $G$, the 
$G$-equivariant homology $H_*^G(Y)$ of $Y$ is defined as 
$H_*^G(Y):=\underset{I}{\on{colim}}\ H^*_G(Y_i,\omega_{i})$, where 
$\omega_{i}\in D(G\backslash Y_i)$ is the dualizing sheaf of $G\backslash Y_i$
and the colimit is induced by the natural adjunction map 
$(\iota_{i,i'})_*\omega_{i}\is(\iota_{i,i'})_!\omega_{i}\to\omega_{i'}$,
and the $G$-equivariant cohomology $H^*_G(Y)$ of $Y$ is defined as 
$H^*_G(Y):=\underset{I}{\on{lim}}\ H^*_G(Y_i,\bC)$, where 
the limit is induced by the natural restriction map 
$H^*_G(Y_{i'},\bC)\to H^*_G(Y_i,\bC)$.

Let $\mL$ be the determinant line bundle on $\Gr_{2n}$
and let $c_1^{T_{2n}}(\mL)\in H^2_{T_{2n}}(\Gr_{2n})$ be its equivariant 
first Chern class. 
It is shown in \cite[Lemma 2.2]{YZ}
that there is an isomorphism of functors
\beq\label{MV filtration}
H^*_{T_{2n}}(\Gr_{2n},-)\is H^*(\Gr_{2n},-)\otimes R_{T_{2n}}:
\on{Perv}(\Gr_{2n})\to R_{T_{2n}}\on{-mod}.
\eeq
induced by the canonical splitting of the MV-filtration
associated to the semi-infinite orbits 
$S^\lambda_{2n}$, the $\frak L N_{2n}$-orbits through $\lambda\in\Lambda_{2n}$.
Moreover, the isomorphism respect the 
natural monoidal structures on $H^*_{T_{2n}}(\Gr_{2n},-)$  coming from fusion and  the one on $H^*(\Gr_{2n},-)\otimes R_{T_{2n}}$ induced from
$H^*(\Gr_{2n},-)$.
The cup product  action  $\wedge c_1^{T_{2n}}(\mL)$
on $H^*_{T_{2n}}(\Gr_{2n},\mF)$, $\mF\in\on{Perv}(\Gr_{2n})$
gives rise to a tensor endomorphism of $H^*_{T_{2n}}(\Gr_{2n},-)$ and hence 
by the Tannakian formalism gives rise 
to an element 
$c^{T_{2n}}\in\fg_{2n}\otimes R_{T_{2n}}$.
One can regard the element $c^{T_{2n}}$ as
a map
\beq\label{c^T_2n}
c^{T_{2n}}:\ft_{2n}\to\fg_{2n}.
\eeq
The equivariant homology $H_*^{T_{2n}}(\Gr_{2n})$ 
is a commutative and cocommutative Hopf algebra over $R_{T_{2n}}$
and 
there is isomorphism of group schemes 
\beq
\on{Spec}(H_*^{T_{2n}}(\Gr_{2n}))\is
(G_{2n}\times\ft_{2n})^{c^{T_{2n}}}
\eeq
where $(G_{2n}\times\ft_{2n})^{c^{T_{2n}}}
$ is the centralizer of $c^{T_{2n}}$ in $G_{2n}\times\ft_{2n}$.

We have a similar result for 
 quaternionic grassmannians.
Let $\mL_\bbH$ be the quaternionic determinant line bundle on $\Gr_{n,\bbH}$
and
let $p^{T}(\mL_\bbH)\in H^4_{T_{c}}(\Gr_{n,\bbH})$ be its equivariant 
Pontryagin class. 
It is shown in \cite{O}
that there is an isomorphism of functors
\beq\label{real MV filtration}
H^*_{T_c}(\Gr_{n,\bbH},-)\is H^*(\Gr_{n,\bbH},-)\otimes_\bC R_{T_{}}:
\on{Perv}(\Gr_{n,\bbH})\to R_{T_{}}\on{-mod}.
\eeq
induced by the canonical splitting of the real MV filtration
associated to the real semi-infinite orbits
$S_{n,\bbH}^\lambda$, the $\frak L N_{n,\bbH}$-orbits through 
$\lambda\in\Lambda_S$.
Moreover,
the isomorphism above respects the
natural monoidal structures on $H^*_{T_{c}}(\Gr_{n,\bbH},-)$  coming from fusion and the one on $H^*(\Gr_{n,\bbH},-)\otimes R_{T_{}}$ induced from
$H^*(\Gr_{n,\bbH},-)$.
The cup product action of $p^{T}(\mL_\bbH)$
on $H^*_{T_{c}}(\Gr_{n,\bbH},\mF)$, $\mF\in\on{Perv}(\Gr_{n,\bbH})\is\on{Rep}(G_n)$
gives rise to a tensor endomorphism of $H^*_{T_{}}(\Gr_{n,\bbH},-)$ and hence 
an element $p^T_X\in\fg_n\otimes R_T$.
Let 
\beq
p^{T}_X:\ft_{}\to\fg_n.
\eeq
be the corresponding map.
The main result in \cite{O} says that 
there an isomorphism of group schemes 
\beq\label{description of equ coh Gr_H}
\on{Spec}(H_*^{T_{c}}(\Gr_{n,\bbH}))\is
(G_n\times\ft_{})^{p^{T_{}}_X}
\eeq
where $(G_n\times\ft_{})^{p^{T_{}}_X}
$ is the centralizer of $p^{T_{}}_X$ in $G_n\times\ft_{}$.

Recall the maps $e^{T_{2n}}$ and $e^T_X$
introduced in~\eqref{e^T} and~\eqref{e^T_X} respectively.

\begin{lemma}\label{computation}
We have 
$c^{T_{2n}}=e^{T_{2n}}$ and  $p^T_X=-e^T_X$.
Thus there are isomorphisms of group scheme \[\on{Spec}(H_*^{T_{2n}}(\Gr_{2n}))\is 
(G_{2n}\times\ft_{2n})^{e^{T_{2n}}}\is J_{2n}\times_{\fc_{2n}}\ft_{2n}
\]
\[ 
\on{Spec}(H_*^{T_{}}(\Gr_{n,\bbH}))\is 
(G_{n}\times\ft_{})^{e_X^{T_{}}}\is J_{n}\times_{\fc_n}\ft,\]
over $\ft_{2n}$
and $\ft$ respectively, and
isomorphisms of 
group schemes 
\[
H_*^{G_{c}}(\Gr_{2n})\is H_*^{T_{2n}}(\Gr_{2n})^{\rW_{2n}}\is
(J_{2n}\times_{\fc_{2n}}\ft_{2n})^{\rW_{2n}}\is J_{2n}
\]
\[ 
H_*^{K_{c}}(\Gr_{n,\bbH})\is H_*^{T_{c}}(\Gr_{n,\bbH})^{\rW_{}}\is
(J_{n}\times_{\fc_{n}}\ft_{})^{\rW_{}}\is J_{n}.\]
over $\fc_{2n}=\ft_{2n}//\rW_{2n}$
and $\ft//\rW\is\fc_n$ respectively.

\end{lemma}

\begin{proof}
It 
follows form the computations in \cite[Section 5]{YZ} and \cite{O}.
We give an alternative (and more direct) proof 
using the computation in Section \ref{two bases}.

Let $V_{\omega_1}$ be the standard representation of  $G_{2n}$ (resp. $G_n$).
We have $H^*_{}(\Gr_{2n},\IC_{\omega_1})\is V_{\omega_1}$
(resp. $H^*_{}(\Gr_{n,\bbH},\IC_{\omega_1})\is V_{\omega_1}$)
and 
it suffices to show that the 
element
\[c^{T_{2n}}=c_1^{T_{2n}}(\mL)\cup(-)\in\on{End}(H^*_{T_{2n}}(\Gr_{2n},\IC_{\omega_1}))\is
 \on{End}(H^*_{}(\Gr_{2n},\IC_{\omega_1})\otimes R_{T_{2n}})\is
 \on{End}(V_{\omega_1}\otimes R_{T_{2n}})\is\]
 \[\is
 \fg_{2n}\otimes R_T\] 
\[(\on{resp.}\ \ 
p^T_X=p^{T}_1(\mL_\bbH)\cup(-)\in\on{End}(H^*_{T_c}(\Gr_{n,\bbH},\IC_{\omega_1}))\is
 \on{End}(H^*_{}(\Gr_{n,\bbH},\IC_{\omega_1})\otimes R_T)\is
  \on{End}(V_{\omega_1}\otimes R_T)\is
\]
  \[\is
 \fg_{n}\otimes R_T)
\]
is given by 
$e^T_X$ (resp. $p^T$).
We have the following observations:

(1) there is a 
$T_{2n}$ (resp. $T_c$)-equivariant isomorphism 
$\Gr^{\omega_1}_{2n}\is\mathbb{P}^{2n-1}$ (resp. 
$\Gr^{\omega_1}_{n,\bbH}\is\mathbb{HP}^{n-1}$)
where $T_{2n}$ (resp. $T_c$) acts on $\mathbb{P}^{2n-1}$ (resp. $\mathbb{HP}^{n-1}$) via the inverse of the natural action\footnote{This is because 
the isomorphism $\Gr^{\omega_1}_{2n}\is\mathbb P^{2n-1}$
is given by the composition the 
canonical $T_{2n}$-equivariant isomorphism 
$\Gr^{\omega_1}_{2n}\is\Gr(2n-1,V_{\omega_1})$, where 
$\Gr(2n-1,V_{\omega_1})$ is the Grassmannian variety
of $(2n-1)$-dim subspaces of $V_{\omega_1}$,
 with
the duality $\Gr(2n-1,V_{\omega_1})\is\Gr(1,V_{\omega_1}^*)\is\mathbb{P}^{2n-1}$.},

(2) the restriction 
$\mL|_{\Gr_{2n}^{\omega_1}}$ (resp. $\mL_{\bbH}|_{\Gr_{n,\bbH}^{\omega_1}}$)
is isomorphic to $\mO(1)$ (resp. $\mO_\bbH(1)$ the $\bbH$-dual of the $\mO_\bbH(-1)$), 

(3) the composed isomorphism
\[R_{T_{2n}}\otimes V_{\omega_1}\is R_{T_{2n}}\otimes H^*_{}(\Gr_{2n},\IC_{\omega_1})\is
H^*_{T_{2n}}(\Gr_{2n},\IC_{\omega_1})\is H^*_{T_{2n}}(\mathbb P^{2n-1},\IC_{\mathbb P^{2n-1}})\]
\[(\on{resp.}\ \  R_{T_{}}\otimes V_{\omega_1}\is R_{T_{}}\otimes H^*_{}(\Gr_{n,\bbH},\IC_{\omega_1})\is
H^*_{T_{}}(\Gr_{n,\bbH},\IC_{\omega_1})\is H^*_{T_{}}(\mathbb{HP}^{n-1},\IC_{\mathbb{HP}^{n-1}}))\]
sends the vectors 
$1\otimes e_i, i=1,...,2n$ to the fundamental class 
\[[\mathbb{P}^{i-1}]\in 
H^{2n+1-2i}_{T_{2n}}(\mathbb P^{2n-1},\IC_{\mathbb P^{2n-1}})=H^{4n-2i}_{T_{2n}}(\mathbb P^{2n-1})\]
(resp. $1\otimes e_i, i=1,...,n$ to the fundamental class 
\[[\mathbb{HP}^{i-1}]\in 
H^{2n+2-4i}_{T_{}}(\mathbb{HP}^{n-1},\IC_{\mathbb{HP}^{n-1}})=H^{4n-4i}_{T_{}}(\mathbb{HP}^{n-1}))\]
From the above observations, we see that 
$c^{T_{2n}}\in\fg_{2n}\otimes R_{T_{2n}}$ (resp. $p^T_X\in\fg_n\otimes R_{T}$) is the matrix presentation of the 
cup product action $c_1^{T_{2n}}(\mO(1))\cup(-)$
(resp. $e^T(\mO_\bbH(1))\cup(-)$)\footnote{Note that the underlying complex rank $2$ bundles  of $\mO_\bbH(-1)$
and $\mO_\bbH(1)$ are complex conjugate to each other and hence 
$e^T(\mO_\bbH(1))\is(-1)^2e^T(\mO_\bbH(-1))=e^T(\mO_\bbH(-1))$.}
in the basis $\{[\mathbb P^{i-1}]\}_{i=1,...,2n}$
(resp. $\{[\mathbb{HP}^{i-1}]\}_{i=1,...,n}$)
and the desired claim follows from Lemma \ref{matrix presentation}.

\end{proof}

\subsubsection{}\label{equ coh and hom}
Recall that for any Lie group $G$
and any 
 ind-proper $G$-variety $Y$ we have a paring 
 \[H^*_{G}(Y)\times H_*^{G}(Y)\to H^*_{G}(\on{pt})\is R_{G}\]
induced by the action of cohomology on homology and then the push-forward map in the 
Borel-Moore homology 
$H_*^{G}(Y)\to H^*_{G}(\on{pt})$. 
On the other hand, for any commutative affine group scheme $H$ over $S$ there is a canonical paring
\[U(\on{Lie} H)\times\mO(H)\to\mO(S)\ \ \ \ (\xi,f)\to \xi(f)|_e\]
between the relative universal enveloping algebra $U(\Lie H)$
and the ring of functions on $H$. Here $e:S\to H$ is the unity map.

According to \cite[Remark 2.13]{BFM},
there are 
isomorphisms
\beq\label{equ cohomology}
H^*_{G_c}(\Gr_{2n})\is U(\Lie J_{2n})\ \ \ \ 
H^*_{K_c}(\Gr_{n,\bbH})\is U(\Lie J_{n})
\eeq
such that
the  paring above 
between cohomology and homology
of $Y=\Gr_{2n}$ (resp. $\Gr_{n,\bbH}$)
 becomes the paring between 
universal enveloping algebra and ring of functions
for the group scheme $H=J_{2n}$ (resp. $J_n$).

\subsection{Fully faithfulness} A key ingredient in the proof of the (complex) derived Satake theorem is the fully faithfulness of the equivariant cohomology functor $H_{\frak L^+G_{2n}}^*(\Gr_{2n}, -)$ into the category of modules over the global cohomology $H_{\frak L^+G_{2n}}^*(\Gr_{2n}, \bbC)$. In \cite{BF}, this was established using general results of Ginzburg \cite{G2}. Ginzburg's arguments appeal to Hodge theory, and therefore must be modified in the real setting. 
As in \cite{AR}, we can use parity considerations in place of Hodge theory. More precisely, we will make use of the theory of parity sheaves \cite{JMW}. 
Our first step, therefore, 
is to establish that the complexes $\IC_\lambda$ (for $\lambda \in\Lambda_S^+$) are even.

\begin{remark}
In fact, our situation is simpler than the modular setting
considered in \cite{JMW} due to the fact that the 
tensor category 
$\on{Perv}(\Gr_{n,\bbH})$ of spherical perverse sheaves 
on $\Gr_{n,\bbH}$ is semisimple (see Proposition \ref{abelian Satake}).
\end{remark}

\subsubsection{} Recall that if a coweight $\mu \in \Lambda_S^+$ is miniscule, the orbit $\Gr_{n,\bbH}^{\mu}$ is closed. Such an orbit is necessarily smooth. 

\begin{lemma}\label{affinepaving}
Let $\mu_1,\dots,\mu_k \in \Lambda_S^+$ denote miniscule coweights. Consider the convolution morphism 
$$
m: \Gr_{n,\bbH}^{\mu_{\bullet}} := \Gr_{n,\bbH}^{\mu_1} \operatorname{\widetilde{\times}} \dots \operatorname{\widetilde{\times}} \Gr_{n,\bbH}^{\mu_k} \rightarrow \Gr_{n,\bbH}. 
$$
Then, the non-empty fibers of $m$ are paved by quaternionic affine spaces. 
\end{lemma}
\begin{proof}
In the complex setting, this result is due to \cite{H}. We proceed by induction on $k$. When $k=1$, there is nothing to prove. In general, we factor $m$ as follows. 
\[\xymatrix{\Gr_{n,\bbH}^{\mu_\bullet} \ar[r]^{q \qquad } \ar[rd]_{m}& \Gr_{n,\bbH}\operatorname{\widetilde{\times}} \Gr_{n,\bbH}^{\mu_k} \ar[d]^{p}\\
&\Gr_{n,\bbH} }\]
Here, $q$ is induced by multiplying the first $k-1$ factors of $\Gr^{\mu_\bullet}_{n,\bbH}$. Since $m$ is $\mathfrak{L}^+{G}_{n,\bbH}$-equivariant, it suffices to show that each fiber $m^{-1}(t^\lambda)$ (for $\lambda \in \Lambda_S^+$) is paved by quaternionic affine spaces. By the above diagram, we have 
$$
m^{-1}(t^\lambda) = q^{-1}(p^{-1}(t^\lambda)).
$$
Let $\mu_\bullet' = (\mu_1,\dots,\mu_{k-1})$. Observe that we have a commutative diagram
\[\xymatrix{m^{-1} (t^\lambda) \ar@{^{(}->}[r] \ar[d]_{q} & \Gr_{n,\bbH}^{\mu_\bullet} \ar[r]^{\qquad} \ar[d]_{q}& \Gr_{n,\bbH}^{\mu'_{\bullet}} \ar[d]^{a}\\
p^{-1}(t^\lambda) \ar@{^{(}->}[r] & \Gr_{n,\bbH} \operatorname{\widetilde{\times}} \Gr_{n,\bbH}^{\mu_k} \ar[r]^{\quad \pi}& \Gr_{n,\bbH}. }\]
Here, $\pi$ is the projection to the first factor and $a$ is the convolution map. Observe that both horizontal compositions are closed embeddings. Hence, we obtain a Cartesian diagram
\[\xymatrix{m^{-1} (t^\lambda) \ar@{^{(}->}[r] \ar[d]_{\pi \circ q} & \Gr_{n,\bbH}^{\mu'_{\bullet}} \ar[d]^{a}\\
\pi(p^{-1}(t^\lambda)) \ar@{^{(}->}[r] & \Gr_{n,\bbH}. }\]
By induction, the fibers of $a$ are paved by quaternionic affine spaces. Hence, the same is true of $\pi \circ q$. It therefore suffices to show that $\pi(p^{-1}(t^\lambda))$ is paved by quaternionic affine spaces over which $a$ is a trivial fibration. Since $p$ is $G_{n,\bbH}(\bbR((t)))$-equivariant, we have
$$
p^{-1}(t^\lambda) = t^\lambda p^{-1}(t^0) = t^\lambda(\Gr^{-\mu_k}_{n,\bbH} \operatorname{\widetilde{\times}} \Gr_{n,\bbH}^{\mu_k}). 
$$
Here, $t^\lambda$ acts on the first factor. Hence,
$$
\pi(p^{-1}(t^\lambda)) = t^\lambda \Gr_{n,\bbH}^{-\mu_k} = t^\lambda \Gr_{n,\bbH}^{-w_0(\mu_k)},
$$
where $w_0$ is the longest element of the Weyl group. Multiplication by $t^\lambda$ is an isomorphism commuting with $a$, hence it suffices to show that $\Gr_{n,\bbH}^{-w_0(\mu_k)}$ is paved by quaternionic affine spaces over which $a$ is a trivial fibration. Let $\mu = -w_0 (\mu_k)$. The coweight $\mu$ is once again minuscule. Recall that $\Gr_{n,\bbH}^{\mu}$ is a vector bundle over a partial flag variety of $G_{n,\bbH}$. On the other hand, $\Gr_{n,\bbH}^{\mu}$ is closed, so it is a partial flag variety of $G_{n,\bbH}$. We claim that the orbits of $P^{\mu}_{n,\bbH}$ on $\Gr^\mu_{n,\bbH}$ are the desired affine spaces. 

Each such orbit has the form $P^{\mu}_{n,\bbH} t^{w(\mu)}$ for $w \in W_n$ an element of the Weyl group. Let $^{w}P^{\mu}_{n,\bbH} = w(P^{\mu}_{n,\bbH})$. Then, by the real Bruhat decomposition, there exists a unipotent subgroup $^{w}N^{\mu}_{n,\bbH}$ of $^{w}P^{\mu}_{n,\bbH}$ which acts freely and transitively on the orbit $P^{\mu}_{n,\bbH} t^{w(\mu)}$. Hence, $P^{\mu}_{n,\bbH} t^{w(\mu)} = {^w}N^{\mu}_{n,\bbH} t^{w(\mu)}$. By the $\mathfrak{L}G_{n,\bbH}$ equivariance of $a$, we have a commutative diagram
\[\xymatrix{ ^{w}N^\mu_{n,\bbH} \times a^{-1}(t^{w(\mu)}) \ar[r]  \ar[d]_{\pi_1} & a^{-1}( {^w}N_{n,\bbH}^{\mu} t^{w(\mu)}) \ar[d]_{a}\\
^{w}N^\mu_{n,\bbH} \ar[r]^{\sim} & {^w}N_{n,\bbH}^{\mu} t^{w(\mu)}. }\]
As the diagram is Cartesian, and the bottom arrow is an isomorphism, the top arrow is an isomorphism are well. By induction, $a^{-1}(t^{w(\mu)})$ is paved by quaternionic affine spaces. Therefore, it suffices to see that the unipotent subgroup $^{w}N^\mu_{n,\bbH}$ is a quaternionic affine space, which is clear. 
\end{proof}

We now recall the terminology of \cite{JMW} that we will use. For $\lambda \in \Lambda_S^+$, let 
$$
i_\lambda: \Gr^\lambda_{n,\bbH} \hookrightarrow \Gr_{n,\bbH}
$$
denote the inclusion. 
\begin{definition}\label{even}
Let $\mathcal{F} \in D^b_{\mathfrak{L}^+G_{n,\bbH}}(\Gr_{n,\bbH})$. We say that $\mathcal{F}$ is \textit{$*$-even} (resp. \textit{$!$-even}) if for all $\lambda \in \Lambda_S^+$, the $\mathfrak{L}^+G_{n,\bbH}$-equivariant sheaf $i_\lambda^*\mathcal{F}$ (resp. $i_\lambda^!\mathcal{F}$) is a direct sum of constant sheaves appearing in even degrees. If $\mathcal{F}$ is both $*$-even and $!$-even, we simply say that it is \textit{even}.

We say that $\mathcal{F}$ is \textit{$*$-odd} (resp. \textit{$!$-odd}) if $\mathcal{F}[1]$ is $*$-even (resp. $!$-even). If $\mathcal{F}$ is both $*$-odd and $!$-odd, we simply say that it is \textit{odd}.
\end{definition}

\begin{prop}\label{IC parity}
For $\lambda \in \Lambda_S^+$, the complex $\IC_\lambda$ is even. 
\end{prop}
\begin{proof}
Since $\IC_\lambda$ is self-dual, it suffices to show that it is $*$-even. Recall from \ref{abelian Satake}(1) that we have an equivalence $\mathrm{Perv}(\Gr_{n,\bbH}) \simeq \mathrm{Rep}(G_n)$ taking $\IC_\lambda$ to the highest weight module $V_\lambda$. Let $\omega_1$ (resp. $\epsilon$) denote the highest weight of the standard representation of $G_n$ (resp. the determinant character). Then, $V_\lambda$ is a direct summand of a tensor product $V_{\epsilon}^{\otimes j} \otimes V_{\omega_1}^{\otimes k}$, for some $j,k \geq 0$. Hence, $\IC_\lambda$ is a direct summand of the convolution $\IC_{\epsilon}^{\star j} \star \IC_{\omega_1}^{\star k}$. It therefore suffices to show that $\IC_{\epsilon}^{\star j} \star \IC_{\omega_1}^{\star k}$ is $*$-even. We now apply \ref{affinepaving} with $\mu_1,\dots,\mu_j = \epsilon$ and $\mu_{j+1} =\dots = \mu_{j+k} = \omega_1$. Let 
$$
m: \Gr^{\mu_\bullet}_{n,\bbH} \rightarrow \Gr_{n,\bbH}
$$
denote the convolution map. We have 
$$
\IC_{\epsilon}^{\star j} \star \IC_{\omega_1}^{\star k} \simeq m_!(\IC_\epsilon^{\widetilde{\boxtimes} j} \operatorname{\widetilde{\boxtimes}} \IC_\epsilon^{\widetilde{\boxtimes} k}). 
$$
Now let $\nu \in\Lambda_S$ and let $i_\nu: \Gr^\nu_{n,\bbH} \hookrightarrow \Gr_{n,\bbH}$ denote the inclusion. Firstly, we have
$$
H^*_{\mathfrak{L}^+G_{n,\bbH}}(i_\nu^*(\IC_{\epsilon}^{\star j} \star \IC_{\omega_1}^{\star k})) \simeq H^*_{T_c}(i_\nu^*(\IC_{\epsilon}^{\star j} \star \IC_{\omega_1}^{\star k}))^W. 
$$
Next, by proper base change,
$$
H^*_{T_c}(i_\nu^*(\IC_{\epsilon}^{\star j} \star \IC_{\omega_1}^{\star k})) \simeq H^*_{T_c}(m^{-1}(\Gr^\nu_{n,\bbH}), \IC_\epsilon^{\widetilde{\boxtimes} j} \operatorname{\widetilde{\boxtimes}} \IC_{\omega_1}^{\widetilde{\boxtimes} k}).
$$
Since $\epsilon$ and $\omega_1$ are minuscule, the orbits $\Gr^{\omega_1}_{n,\bbH}$ and $\Gr^{\epsilon}_{n,\bbH}$ are smooth. Therefore, $\IC_{\omega_1} \simeq \underline{\bbC}[2(n-1)]$ and $\IC_{\epsilon} \simeq \underline{\bbC}$. Hence,
$$
H^*_{T_c}(m^{-1}(\Gr^\nu_{n,\bbH}), \IC_\epsilon^{\widetilde{\boxtimes} j} \operatorname{\widetilde{\boxtimes}} \IC_{\omega_1}^{\widetilde{\boxtimes} k}) \simeq H^*_{T_c}(m^{-1}(\Gr^\nu_{n,\bbH}), \bbC)[2k(n-1)]. 
$$
By \ref{affinepaving}, the ordinary cohomology $H^*(m^{-1}(\Gr^\nu_{n,\bbH}), \bbC)$ is concentrated in even degrees (in fact, in degrees divisible by $4$). Hence, $m^{-1}(\Gr^\nu_{n,\bbH})$ is equivariantly formal with respect to the action of $T_c$. Therefore, $H^*_{T_c}(m^{-1}(\Gr^\nu_{n,\bbH}), \bbC)$ is concentrated in even degrees.

Now we may express $i_\nu^*(\IC_{\epsilon}^{\star j} \star \IC_{\omega_1}^{\star k})$ as a direct sum of constant sheaves. We have
$$
i_\nu^*(\IC_{\epsilon}^{\star j} \star \IC_{\omega_1}^{\star k}) \simeq \underline{V}
$$
for a complex $V \in D^b(\mathrm{Vect}_{\bbC})$. Hence,
$$
H^*_{T_c}(i_\nu^*(\IC_{\epsilon}^{\star j}) \simeq H^*_{T_c}(\Gr^\lambda_{n,\bbH}, \bbC) \otimes V. 
$$ 
We have shown that $H^*_{T_c}(i_\nu^*(\IC_{\epsilon}^{\star j})$ is concentrated in even degrees. Since $H^0_{T_c}(\Gr^\lambda_{n,\bbH}, \bbC) \neq 0$, we conclude that $V$ is concentrated in even degrees. The result follows. 
\end{proof}

As a corollary of the proof we obtain the following parity vanishing 
results.

\begin{corollary}\label{parity vanishing}
We have $\mathscr H^{i-\langle\lambda,\rho_{2n}\rangle}(\IC_\lambda)=0$ for 
$i\nmid 4$.
\end{corollary}
\begin{proof}
We have shown that 
any direct summand 
$\IC_\lambda$ of $\IC_{\epsilon}^{\star j} \star \IC_{\omega_1}^{\star k}$ satisfies
$\mathscr H^{i-2k(n-1)}(\IC_\lambda)=0$ for 
$i\nmid 4$. Since $k\omega_1-\lambda$ is a non-negative integral sum of 
positive coroots,  we have  $\langle k\omega_1-\lambda,\rho_n\rangle\in\bZ$
and hence $2k(n-1)-\langle\lambda,\rho_{2n}\rangle=
\langle k\omega_1-\lambda,\rho_{2n}\rangle=4\langle k\omega_1-\lambda,\rho_n\rangle$ is divisible by four. The desired claim follows. 

\end{proof}

\subsubsection{} Our goal is now to apply the parity vanshing above to deduce the following faithfulness result. 
\begin{prop}\label{faithfulness}
For any $\lambda, \mu\in \Lambda_S^+$, the natural map
\[\on{Ext}^\bullet_{D^b_{\frak L^+G_{n,\bbH}}(\Gr_{n,\bbH})}(\IC_\lambda,\IC_\mu)\to
\Hom^\bullet_{H^*_{\frak L^+G_{n,\bbH}}(\Gr_{n,\bbH})}(H^*_{\frak L^+G_{n,\bbH}}(\IC_\lambda),H^*_{\frak L^+G_{n,\bbH}}(\IC_\mu))\]
is an isomorphism of graded modules.
\end{prop}

We will deduce \ref{faithfulness} as a consequence of the following more general result. 

\begin{prop}\label{faithfulness and parity}
Let $\mathcal{F}, \mathcal{G} \in D^b_{\mathfrak{L}^+G_{n, \bbH}}(\Gr_{n,\bbH})$. Assume that $\mathcal{F}$ and $\mathcal{G}$ are even. Then, the natural map 
$$
\mathrm{Ext}^\bullet_{D^b_{\mathfrak{L}^{+}G_{n,\bbH}}(\Gr_{n,\bbH})}(\mathcal{F}, \mathcal{G}) \rightarrow \mathrm{Hom}^\bullet_{H^*_{\mathfrak{L}^+G_{n,\bbH}}(\Gr_{n,\bbH})}(H^*_{\mathfrak{L}^+G_{n,\bbH}}(\mathcal{F}), H^*_{\mathfrak{L}^+G_{n,\bbH}}(\mathcal{G}))
$$
is an isomorphism of graded modules. 
\end{prop}

\begin{proof}[Proof of \ref{faithfulness}]
By \ref{IC parity}, we know that $\IC_\lambda$ and $\IC_\mu$ are even complexes. The claim now follows from \ref{faithfulness and parity}. 
\end{proof}

\subsubsection{} In the proof of \ref{faithfulness and parity}, we will make use of the following terminology. Consider a triangulated functor
$$
\Omega: D^b_{\mathfrak{L}^+G_{n, \bbH}}(\Gr_{n,\bbH}) \rightarrow D^b(\mathrm{Vect}_\bbC). 
$$
We say that $\Omega$ is $*$-\textit{parity preserving} (resp. $!$-\textit{parity preserving}) if it takes $*$-even (resp. $!$-even)  complexes of sheaves to even complexes of vector spaces. If $\Omega$ is both $*$-parity preserving and $!$-parity preserving, we will simply say that it is parity preserving. We use the same terminology for functors 
$$
\Omega: D^b_{\mathfrak{L}^+G_{n, \bbH}}(\Gr_{n,\bbH})^{\mathrm{op}} \rightarrow D^b(\mathrm{Vect}_\bbC).
$$
 To check that functors are parity preserving, we will use the following criterion. 
\begin{lemma}\label{parity criterion}
Let
$$
\Omega: D^b_{\mathfrak{L}^+G_{n, \bbH}}(\Gr_{n,\bbH}) \rightarrow D^b(\mathrm{Vect}_\bbC). 
$$
be a triangulated functor. Then, $\Omega$ is $*$-even if and only if, for each $\lambda \in \Lambda_S^+$, the complex $\Omega(j_{\lambda !} \underline{\bbC})$ is even. Here, $j_{\lambda}: \Gr^\lambda_{n,\bbH} \hookrightarrow \Gr_{n,\bbH}$ is the natural inclusion. 
\end{lemma}
\begin{proof}
We assume that the latter condition holds, and prove that $\Omega$ is $*$-parity preserving. Let $\mathcal{F} \in D^b_{\mathfrak{L}^+G_{n, \bbH}}(\Gr_{n,\bbH})$ be $*$-even. We must show that $\Omega(\mathcal{F})$ is even, which we do by induction on the support of $\mathcal{F}$ (which is a finite union of $\mathfrak{L}^+ G_{n,\bbH}$ orbits). Let  
$$
j: \Gr^\lambda_{n,\bbH} \hookrightarrow \Gr_{n,\bbH}
$$
denote the inclusion of a $\mathfrak{L}^+ G_{n,\bbH}$ orbit open in the support of $\mathcal{F}$. Let 
$$
i: \overline{\Gr^\lambda_{n,\bbH}} \setminus \Gr^\lambda_{n,\bbH} \hookrightarrow \Gr_{n,\bbH}
$$
denote the complementary closed embedding. We have a triangle
$$
j_{ !}j_{}^! \mathcal{F} \rightarrow \mathcal{F} \rightarrow i_{ *}i^*_{} \mathcal{F} \rightarrow . 
$$
Applying $\Omega$ yields
$$
\Omega(j_!j^! \mathcal{F}) \rightarrow \Omega(\mathcal{F}) \rightarrow \Omega(i_*i^* \mathcal{F}) \rightarrow .
$$
By induction, we may assume that $\Omega(i_*i^* \mathcal{F})$ is even. On the other hand, $\mathcal{F}$ is $\mathfrak{L}^+G_{n,\bbH}$ equivariant and $*$-even. Hence, $j^!\mathcal{F} \simeq \underline{\bbC}[m]$, for $m \in \bbZ$. We have that $m$ is even. Hence, $\Omega(j_!j^! \mathcal{F})$ is $*$-even. Therefore, $\Omega(\mathcal{F})$ is $*$-even, and $\Omega$ is $*$-parity preserving.  

For the converse, observe that each $j_{\lambda!}\underline{\bbC}$ is $*$-even, as its only non-trivial stalk is isomorphic to $H^*_{\mathfrak{L}^+ G_{n,\bbH}}(\mathrm{pt})$. Hence, $\Omega(j_{\lambda!}\underline{\bbC})$ is $*$-even. 
\end{proof}

\begin{corollary}\label{parity examples}
(i) The functor $\Gamma_{\mathfrak{L}^+ G_{n,\bbH}}(\Gr_{n,\bbH}, -)$ is $*$-parity preserving. 

(ii) Let $\mathcal{G} \in D^b_{\mathfrak{L}^+G_{n,\bbH}}(\Gr_{n,\bbH})$ be $!$-even. Then, the functor $\mathrm{Hom}_{D^b_{\mathfrak{L}^{+}G_{n,\bbH}}(\Gr_{n,\bbH})}(-, \mathcal{G})$ is $*$-parity preserving. 
\end{corollary}
\begin{proof}
(i) We checked in the proof of \ref{parity criterion} that each $H^*_{\mathfrak{L}^+ G_{n,\bbH}}(j_!\underline{\bbC})$ is even. By the conclusion of that lemma, we conclude that $\Gamma_{\mathfrak{L}^+ G_{n,\bbH}}(\Gr_{n,\bbH}, -)$ is $*$-parity preserving. 

(ii) By \ref{parity criterion}, we must check that 
$$
\mathrm{Ext}^i_{D^b_{\mathfrak{L}^{+}G_{n,\bbH}}(\Gr_{n,\bbH})}(j_{\lambda !} \underline{\bbC}, \mathcal{G}) \simeq 0
$$
for each $\lambda \in X_{A}^+$ and $i$ odd. By adjunction, 
$$
\mathrm{Ext}^i_{D^b_{\mathfrak{L}^{+}G_{n,\bbH}}(\Gr_{n,\bbH})}(j_{\lambda !} \underline{\bbC}, \mathcal{G}) \simeq \mathrm{Ext}^i_{D^b_{\mathfrak{L}^{+}G_{n,\bbH}}(\Gr_{n,\bbH})}(\underline{\bbC}, j_{\lambda}^!\mathcal{G}) \simeq H_{\mathfrak{L}^{+}G_{n,\bbH}}^i(j_{\lambda}^! \mathcal{G}). 
$$
The claim follows from the assumption that $\mathcal{G}$ is $!$-even. 
\end{proof}

\begin{lemma}\label{parity splitting}
Let $\mathcal{F}\in D^b_{\mathfrak{L}^+G_{n, \bbH}}(\Gr_{n,\bbH})$ be $*$-even. Suppose that $X \subseteq \Gr_{n,\bbH}$ is a closed finite union of $\mathfrak{L}^+{G_{n,\bbH}}$ orbits, such that $X$ contains the support of $\mathcal{F}$. Let $Z \subseteq X$ denote a $\mathfrak{L}^+{G_{n,\bbH}}$-stable closed subset, and $U = X \setminus Z$ its open complement. Let $j: U \hookrightarrow X$ and $i: Z \hookrightarrow X$ denote the natural inclusions. We have a triangle
$$
j_!j^! \mathcal{F} \rightarrow \mathcal{F} \rightarrow i_*i^* \mathcal{F} \rightarrow .
$$
Now let $\Omega: D^b_{\mathfrak{L}^+G_{n, \bbH}}(\Gr_{n,\bbH}) \rightarrow D^b(\mathrm{Vect}_\bbC)$ denote a $*$-parity preserving functor. Then, the triangle  
$$
\Omega(j_!j^! \mathcal{F}) \rightarrow \Omega(\mathcal{F}) \rightarrow \Omega(i_*i^* \mathcal{F}) \rightarrow
$$
is split. 
\end{lemma}
\begin{proof}
We must show that the boundary map $\delta: \Omega(i_*i^* \mathcal{F}) \rightarrow \Omega(j_!j^! \mathcal{F})[1]$ is zero. Observe that the functors $j_!j^!$ and $i_*i^*$ take $*$-even sheaves to $*$-even sheaves. Since $\Omega$ is $*$-parity preserving, the complexes $\Omega(i_*i^* \mathcal{F})$ and $\Omega(j_!j^! \mathcal{F})$ are even. Hence, $\Omega(j_!j^! \mathcal{F})[1]$ is odd. Therefore, $\delta$ induces the zero map in cohomology, so is zero. 
\end{proof}

\begin{lemma}\label{key induction}
Let $\mathcal{F}, \mathcal{G} \in D^b_{\mathfrak{L}^+G_{n, \bbH}}(\Gr_{n,\bbH})$. We make the following assumptions:
\begin{enumerate}
\item $\mathcal{F}$ is $*$-even. 
\item $\mathcal{G}$ is $!$-even.
\item For any $\mu \in \Lambda_S^+$, the map 
$$
H^*_{\mathfrak{L}^+G_{n,\bbH}}(\mathcal{F}) \rightarrow H^*_{\mathfrak{L}^+G_{n,\bbH}}(j_\mu^*\mathcal{F}) 
$$
is surjective. 
\item For any $\mu \in \Lambda_S^+$, the map 
$$
H^*_{\mathfrak{L}^+G_{n,\bbH}}(j_{\mu !}j_{\mu}^!\mathcal{G}) \rightarrow H^*_{\mathfrak{L}^+G_{n,\bbH}}(\mathcal{G}) 
$$
is injective. 
\end{enumerate}
Then, the natural map 
$$
\mathrm{Ext}^*_{D^b_{\mathfrak{L}^{+}G_{n,\bbH}}(\Gr_{n,\bbH})}(\mathcal{F}, \mathcal{G}) \rightarrow \mathrm{Hom}^*_{H^*_{\mathfrak{L}^+G_{n,\bbH}}(\Gr_{n,\bbH})}(H^*_{\mathfrak{L}^+G_{n,\bbH}}(\mathcal{F}), H^*_{\mathfrak{L}^+G_{n,\bbH}}(\mathcal{G}))
$$
is an isomorphism of graded modules. 
\end{lemma}
\begin{proof}
Let $Z$ denote the union of the supports of $\mathcal{F}$ and $\mathcal{G}$. We proceed by induction on $Z$. Certainly there exists an orbit $\Gr^\lambda_{n,\bbH}$ open in $Z$. Let $Y = Z \setminus \Gr^\lambda_{n,\bbH}$, and let $i_\lambda: Y \hookrightarrow Z$ denote the inclusion. We claim that the pair $i_\lambda^*\mathcal{F}$, $i_{\lambda}^!\mathcal{G}$ satisfies the hypotheses (1)-(4). 

That $i_{\lambda}^*\mathcal{F}$ is $*$-even and that $i_{\lambda}^! \mathcal{G}$ is $!$-even is evident. We verify (3) for $i_{\lambda}^*{\mathcal{F}}$. If $\Gr^\mu_{n,\bbH}$ does not lie in the support of $i_{\lambda}^*\mathcal{F}$, then there is nothing to prove. So, we may assume that $\Gr^\mu_{n,\bbH} \subseteq Y$. Therefore, we have the composition of maps
$$
H^*_{\mathfrak{L}^+G_{n,\bbH}}(\mathcal{F}) \rightarrow H^*_{\mathfrak{L}^+G_{n,\bbH}}(i_{\lambda}^*\mathcal{F} ) \rightarrow H^*_{\mathfrak{L}^+G_{n,\bbH}}(j_\mu^*i_{\lambda}^*\mathcal{F}) \simeq H^*_{\mathfrak{L}^+G_{n,\bbH}}(j_\mu^*\mathcal{F}). 
$$
The composite is surjective by assumption. Hence, the map 
$$
H^*_{\mathfrak{L}^+G_{n,\bbH}}(i_{\lambda}^*\mathcal{F} ) \rightarrow H^*_{\mathfrak{L}^+G_{n,\bbH}}(j_\mu^*i_{\lambda}^*\mathcal{F})
$$
is surjective, as needed. The proof that $i_{\lambda}^!\mathcal{G}$ satisfies (4) is similar. 

Now we proceed with the induction. To avoid overly cumbersone notation, we will surpress the subscripts on $\mathrm{Ext}^*_{D^b_{\mathfrak{L}^{+}G_{n,\bbH}}(\Gr_{n,\bbH})}$ and $\mathrm{Hom}^*_{H^*_{\mathfrak{L}^+G_{n,\bbH}}(\Gr_{n,\bbH})}$. Similarly, we will make use of the isomorphism $H^*_{\mathfrak{L}^+G_{n,\bbH}} \simeq H^*_G$ to simplify notation. Lastly, we let $H = H^*_{\mathfrak{L}^+G_{n,\bbH}}(\Gr_{n,\bbH})$. 

 Consider the triangle 
\begin{equation}\label{basic triangle}
j_{\lambda !} j^{!}_{\lambda} \mathcal{F} \rightarrow \mathcal{F} \rightarrow i_{\lambda *} i_{\lambda}^*\mathcal{F} \rightarrow .
\end{equation}
By \ref{parity splitting}, \ref{parity examples}, and adjunction, we have an exact sequence 
\begin{equation}\label{ext sequence}
0 \rightarrow \mathrm{Ext}^*(i_{\lambda}^* \mathcal{F}, i_{\lambda}^!\mathcal{G}) \rightarrow \mathrm{Ext}^*(\mathcal{F}, \mathcal{G}) \rightarrow \mathrm{Ext}^*(j_{\lambda}^! \mathcal{F}, j_{\lambda}^!\mathcal{G}) \rightarrow 0.
\end{equation}
We can also apply the functor $H^*$ to \ref{basic triangle} to obtain the exact sequence 
$$
0 \rightarrow H_G^*(j_{\lambda !}j_{\lambda}^! \mathcal{F}) \rightarrow H_G^*(\mathcal{F}) \rightarrow H_G^*(i_{\lambda}^*\mathcal{F}) \rightarrow 0. 
$$
Similarly, we have the exact sequence
\begin{equation}
0 \rightarrow H_G^*(i_{\lambda}^! \mathcal{G}) \rightarrow H_G^*(\mathcal{G}) \rightarrow H_G^*(j_{\lambda}^*\mathcal{G}) \rightarrow 0. 
\end{equation}
These two exact sequences induce a sequence
$$
0 \rightarrow \mathrm{Hom}^*(H^*_G(i_{\lambda}^*\mathcal{F}), H^*_G(i_{\lambda}^!\mathcal{G})) \rightarrow \mathrm{Hom}^*(H^*_G(\mathcal{F}), H^*_G(\mathcal{G})) \rightarrow \mathrm{Hom}^*(H^*_G(\mathcal{F}), H^*_G(j_{\lambda}^*\mathcal{G})).
$$
The second map is clearly an injection. We claim that the sequence is also exact in the middle. It suffices to show that any $H$-linear map 
$$
\alpha: H_G^*(\mathcal{F}) \rightarrow H_G^*(i_{\lambda}^! \mathcal{G})
$$
factors through $H_G^*(i_{\lambda}^*\mathcal{F})$. Consider the compactly supported cohomology $H^*_{G,c}(\Gr^{\lambda}_{n,\bbH})$. Let $\mathfrak{c}_\lambda \in H_{G,c}^{\langle 2\rho_{2n}, \lambda \rangle}(\Gr^{\lambda}_{n,\bbH})$ denote a lift of a generator to $G$-equivariant cohomology; it maps to an element $\mathfrak{c}_{\lambda} \in H$. Since $\mathfrak{c}_{\lambda}$ maps to $0 \in H^*_G(Y)$, it acts trivially on $H^*(i_{\lambda}^!\mathcal{G})$. Since $\alpha$ is $H$-linear, it suffices to show that $H^*(j_{\lambda !}j_\lambda^! \mathcal{F})$ lies in the image of 
$$
\mathfrak{c}_{\lambda}: H_G^*(\mathcal{F}) \rightarrow H^*_G(\mathcal{F}[-\langle 2 \rho_{2n}, \lambda \rangle]). 
$$
To do so, we note that by Poincar\'e duality for the smooth manifold $\Gr_{n,\bbH}^\lambda$, cupping with $\mathfrak{c}_{\lambda}$ induces an isomorphism 
$$
\mathfrak{c}_{\lambda}: H_G^*(j_{\lambda !}j^!_{\lambda}\mathcal{F}) \rightarrow H_G^*(j_{\lambda *} j_{\lambda}^*\mathcal{F}[-\langle 2 \rho_{2n}, \lambda \rangle]). 
$$
Thus we obtain a commutative diagram 
\[\xymatrix{H_G^*(\mathcal{F}) \ar[r]^{\mathfrak{c}_\lambda \qquad}& H_G^*(\mathcal{F}[-\langle 2 \rho_{2n}, \lambda \rangle]) \ar[d]\\
H_G^*(j_{\lambda !}j^!_{\lambda}\mathcal{F}) \ar[u] \ar[r]^{\mathfrak{c}_\lambda \qquad} & H_G^*(j_{\lambda *} j_{\lambda}^*\mathcal{F}[-\langle 2 \rho_{2n}, \lambda \rangle]).}\]
As the bottom arrow is an isomorphism, it suffices to show that the right vertical map is surjective. But this is assumed in (3). 

Next, we observe that any $H$-linear map $\beta: H_G^*(\mathcal{F}) \rightarrow H_G^*(j^*_{\lambda}\mathcal{G})$ factors through $H_G^*(j^*_{\lambda} \mathcal{F})$. The proof is similar to that of the previous step, using (4) in place of (3), and is therefore omitted. 

Hence, we have an exact sequence 
$$
0 \rightarrow \mathrm{Hom}^*(H_G^*(i_{\lambda}^*\mathcal{F}), H_G^*(i_{\lambda}^!\mathcal{G})) \rightarrow \mathrm{Hom}^*(H_G^*(\mathcal{F}), H_G^*(\mathcal{G})) \rightarrow \mathrm{Hom}^*(H_G^*(j_{\lambda}^*\mathcal{F}), H_G^*(j_{\lambda}^*\mathcal{G})).
$$
To conclude, we observe that this exact sequence fits into the following commutative diagram with \ref{ext sequence}: 
\[\xymatrix{\mathrm{Ext}^*(i_\lambda^* \mathcal{F}, i_\lambda^!\mathcal{G}) \ar[r] \ar[d]_{f} & \mathrm{Ext}^*(\mathcal{F}, \mathcal{G}) \ar[r]^{\qquad} \ar[d]_{g}& \mathrm{Ext}^*(j_\lambda^*\mathcal{F}, j_\lambda^*\mathcal{G})  \ar[d]^{h}\\
\mathrm{Hom}_{H}^*(H^*_{G}(i_\lambda^* \mathcal{F}), H^*_G(i_\lambda^!\mathcal{G})) \ar[r] & \mathrm{Hom}_{H}^*(H^*_{G}(\mathcal{F}), H^*_{G}(\mathcal{G})) \ar[r] & \mathrm{Hom}_{H}^*(H^*_{G}(j_\lambda^*\mathcal{F}), H^*_{G}(j_\lambda^*\mathcal{G})). }\]
The map $f$ is an isomorphism by induction, and $h$ is easily seen to be an isomorphism. Hence, $g$ is an isomorphism, as claimed. 
\end{proof}
\begin{proof}[Proof of \ref{faithfulness and parity}]
It suffices to verify that $\mathcal{F}$ and $\mathcal{G}$ satisfy the hypotheses of \ref{key induction}. The properties (1) and (2) are assumed. We will show that (3) holds; the proof of (4) is similar. We must show that for each $\lambda \in \Lambda_S^+$, the map 
$$
H^*_G(\mathcal{F}) \rightarrow H^*_{G}(j_\lambda^*\mathcal{F})
$$
is surjective. It identifies with 
$$
H^*_{T_c}(\mathcal{F})^{W_n} \rightarrow H^*_{T_c}(j_\lambda^*\mathcal{F})^{W_n}. 
$$
Since the coefficient field has characteristic zero, the functor of $W_n$-invariants is exact, and it suffices to show that the restriction map $H^*_{T_c}(\mathcal{F}) \rightarrow H^*_{T_c}(j_\lambda^*\mathcal{F})$ is surjective. We let 
$$
k_{\lambda}: (\Gr^\lambda_{n,\bbH})^{T_c} \hookrightarrow \Gr_{n,\bbH}
$$
denote the inclusion of the $T_c$-fixed locus in $\Gr^\lambda_{n,\bbH}$. Now, consider the composition
\beq\label{surj}
H^*_{T_c}(\mathcal{F}) \rightarrow H^*_{T_c}(j_\lambda^*\mathcal{F}) \rightarrow H^*_{T_c}(k_\lambda^*\mathcal{F}). 
\eeq
Observe that $j_\lambda^*\mathcal{F}$ is a constant sheaf, and that $\Gr^\lambda_{n,\bbH}$ is an equivariantly formal $T_c$-manifold. Hence, the second map above is injective by the localization theorem. The surjectivity of the first map is then reduced to that of the composition. Now, $k_\lambda$ is a \textit{closed} inclusion, so $k_{\lambda *}k_{\lambda}^{*} \mathcal{F}$ is $*$-even. Thus, \ref{parity splitting} shows that the restriction map $H^*_{T_c}(\mathcal{F}) \rightarrow H^*_{T_c}(k_\lambda^*\mathcal{F})$ is indeed surjective.
\end{proof}

\subsection{Ext algebras}
The tensor equivalence 
$\on{Rep}(G_n)\is\on{Perv}(\Gr_{n,\bbH})$ gives rise to 
a monoidal action of $\on{Rep}(G_n)$ on
$D^b_{}(\frak L^+G_{n,\bbH}\backslash\Gr_{n,\bbH})$.
We compute the de-equivariantized Extension algebra 
\[\on{Ext}^*_{D^b_{}(\frak L^+G_{n,\bbH}\backslash\Gr_{n,\bbH})}(\IC_0,\IC_0\star\mO(G_n)).\]

\begin{prop}\label{computation of Ext}
There is a  $G_n$-equivariant isomorphism of graded algebras
\[\on{Ext}^*_{D^b_{\frak L^+G_{n,\bbH}}(\Gr_{n,\bbH})}(\IC_0,\IC_0\star\mO(G_n))\is\mO(\fg_{n}[4])\is\on{Sym}(\fg_n[-4])\]
\end{prop}
\begin{proof}
By Proposition \ref{faithfulness}, 
taking equivariant cohomology induces  a $G_n$-equivariant isomorphism of 
graded algebras
\[\on{Ext}^*_{D^b_{}(\frak L^+G_{n,\bbH}\backslash\Gr_{n,\bbH})}(\IC_0,\IC_0\star\mO(G_n))
\is(\Hom^*_{H^*_{T_c}(\Gr_{n,\bbH})}(H^*_{T_c}(\on{pt}),H^*_{T_c}(\Gr_{n,\bbH},\IC_0\star\mO(G_n))^\rW\is\]
\[\is
(\Hom^*_{H^*_{T_c}(\Gr_{n,\bbH})}(\mO(\ft),\mO(G_n\times\ft)))^\rW\is 
(\mO(G_n\times\ft)^{\on{Spec}(H_*^{T_c}(\Gr_{n,\bbH}))})^\rW\]
where $\mO(G_n\times\ft)^{\on{Spec}(H_*^{T_c}(\Gr_{n,\bbH}))}\subset\mO(G_n\times\ft)$ is the subspace consisting of 
functions 
that are invariant 
(relative over $\ft$) with respect to the left  action of the group scheme 
$\on{Spec}(H_*^{T_c}(\Gr_{n,\bbH}))\is(G_n\times\ft)^{e^T_X}$ on $G_n\times\ft$.
Since 
$\mO(\fg_n^{\on{reg}}\times_{\fc_n}\ft)=\mO(\fg_n^{\on{}}\times_{\fc_n}\ft)$ and the 
 map
\beq\label{torsor}
\nu:G_n\times\ft\to \fg_n^{\on{reg}}\times_{\fc_n}\ft,\ \  (g,t)\to (\on{Ad}_{g^{-1}}e^T_X(t),t)
\eeq
realizes $G_n\times\ft$  as a $(G_n\times\ft)^{e^T_X}$-torsor over $\fg_n^{\on{reg}}\times_{\fc_n}\ft$, we obtain an isomorphism of algebra
\[\on{Ext}^*_{D^b_{}(\frak L^+G_{n,\bbH}\backslash\Gr_{n,\bbH})}(\IC_0,\IC_0\star\mO(G_n))\is(\mO(G_n\times\ft)^{\on{Spec}(H_*^{T_c}(\Gr_{n,\bbH}))})^\rW\is\mO(\fg_n^{\on{reg}}\times_{\fc_n}\ft)^\rW\is\]
\[\is
\mO(\fg_n^{\on{}}\times_{\fc_n}\ft)^\rW\is\mO(\fg_n).\]
It remain to check that the isomorphism above is compatible with the desired gradings.
 By \cite[Theorem 8.5.1]{Na}, 
 for any $\lambda\in\Lambda_S$
 and $\mF\in\on{Perv}(\Gr_{n,\bbH})$,
the compactly supported cohomology $H^*_c(S_{n,\bbH}^\lambda,\mF)$ 
along the real semi-infinite orbit $S^\lambda_{n,\bbH}$
is non-zero only in degree
$\langle\lambda,\rho_{2n}\rangle$.
Note that 
$\langle\lambda,\rho_{2n}\rangle=4\langle\lambda,\rho_{n}\rangle$
where in the second paring we regard $\lambda$ as an element in $\Lambda_n$.
Thus the grading on $H^*(\Gr_{n,\bbH},\IC_\lambda)$ corresponds, under the 
geometric Satake equivalence, to the grading on $V_\lambda$
given by cocharacter $4\rho_n$
and it follows that 
the grading on $H^*_{T_c}(\Gr_{n,\bbH},\IC_0\star\mO(G_n))\is\mO(G_n\times\ft)$
is induced by the $\bG_m$ action on $G_n\times\ft$ 
given by $x(g,t)=(4\rho_n(x)g,x^{-2}t)$ (note that the generators of $\mO(\ft)$ are in degree $2$).
We claim that the map $\nu$ in~\eqref{torsor} is $\bG_m$-equivaraint 
with respect to the above action on $G_n\times\ft$ 
and the action on $\fg_n^{\on{reg}}\times_{\fc_n}\ft$ given by
$x(v,t)=(x^{-4}v,x^{-2}t)$.
Indeed,
we have 
\[\on{Ad}_{4\rho_n(x^{-1})}e_X^T(x^{-2}t)=\on{Ad}_{4\rho_n(x^{-1})}\begin{pmatrix} x^{-4}t_1^2 & 1 & \\ 
 \vdots& x^{-4}t_2^2 &\ddots &  \\ 
 \vdots &  & \ddots & 1 \\
0& 0& \hdots &x^{-4}t_n^2 \end{pmatrix}=\begin{pmatrix} x^{-4}t_1^2 & x^{-4} & \\ 
 \vdots& x^{-4}t_2^2 &\ddots &  \\ 
 \vdots &  & \ddots & x^{-4} \\
0& 0& \hdots &x^{-4}t_n^2 \end{pmatrix}=\]
\[=
x^{-4}e_X^T(t)\]
and hence 
\[\nu(x(g,t))=
\nu(4\rho_n(x)g,x^{-2}t)=(\on{Ad}_{g^{-1}}\on{Ad}_{4\rho_n(x)^{-1}}e_X^T(x^{-2}t),x^{-2}t)=\]
\[=(x^{-4}\on{Ad}_{g^{-1}}e_X^T(t),x^{-2}t)=x(\on{Ad}_{g^{-1}}e_X^T(t),t)=x\nu(g,t).\]
Thus the pull-back along the map $\nu$ induces an isomorphism 
of graded algebras
\[(\mO(G_n\times\ft)^{\on{Spec}(H_*^{T_c}(\Gr_{n,\bbH}))})^\rW\is\mO(\fg_n^{\on{reg}}[4]\times_{\fc_n}\ft[2])^\rW\is\mO(\fg_n[4]).\]
This finishes the proof of the theorem.

\end{proof}

\subsection{IC-stalks, $q$-analog of weight multiplicity, and Kostka-Foulkes polynomials}
In this section we shall prove Theorem \ref{IC stalks, intro} (2).
We will follow Ginzburg's approach \cite{G1} (see also \cite[Section 5]{Z}) using
techniques of equivariant cohomology.

Let $V\in\on{Rep}(G_n)$.
Consider the Brylinski-Kostant filtration $F_iV:=\on{ker} e_n^{i+1}, i\geq 0$ on $V$ associated to the regular nilpotent 
element $e_n$. For any $\mu\in\Lambda_n$, we denote by 
$V(\mu)$ the $\mu$-weight space of $V$ (since $G_n$ is self dual we can 
 view $\Lambda_n$
as the weight lattice of $G_n$). The filtration $F_iV$ induces a filtration on the weight space:
\[F_iV(\mu)=F_iV\cap V(\mu).\]
Let
\[P_\mu(V,q)=\sum_i\dim(F_iV(\mu)/F_{i-1}V(\mu))q^i\]
be the $q$-analogue of weight multiplicity polynomial.

From now now 
we will identify $\Lambda_n$ with the set 
$\Lambda_S$
of 
real co-weights 
and denote by $s_\mu:\{\mu\}\to \Lambda_n\is\Lambda_S\subset\Gr_{n,\bbH}$
the inclusion map

\begin{thm}\label{IC-stalks}
Let $\mF\in\on{Perv}(\Gr_{n,\bbH})$
and let $V=H^*(\Gr_{n,\bbH},\mF)$ be the 
corresponding representation of $G_n$.
 We have 
\[P_{\mu}(V,q)=\sum_i\dim H^{-4i-4\langle\mu,\rho_n\rangle}(s_\mu^*\mF)q^i=
\sum_i\dim H^{4i+4\langle\mu,\rho_n\rangle}(s_\mu^!\mF)q^i\]
\end{thm}
The theorem above implies Theorem \ref{IC stalks, intro} (2)
in the case of quaternionic affine Grassmannian. Indeed,
if $\mu,\lambda\in\Lambda_n^+$ and $V=V_\lambda$ is the irreducible representation of 
highest weight $\lambda$, then it is known that 
$P_\mu(V_\lambda,q)=K_{\lambda,\mu}(q)$ is the 
Kostka-Foulkes polynomial associated to $\lambda,\mu$ (see, e.g., \cite{B}).
Thus for any $x\in\Gr_{n,\bbH}^\mu$, we have 
\[K_{\lambda,\mu}(q)=\sum_i\dim H^{-4i-4\langle\mu,\rho_n\rangle}(s_\mu^*\mF)q^i=\sum_i\dim\mathscr H^{-4i-4\langle\mu,\rho_n\rangle}_x(\IC_\lambda)q^i\]
and it follows that 
\[q^{\langle\lambda-\mu,\rho_n\rangle}K_{\lambda,\mu}(q^{-1})=
\sum_i\dim\mathscr H^{-4i-4\langle\mu,\rho_n\rangle}_x(\IC_\lambda)q^{-i-
\langle\mu,\rho_n\rangle+\langle\lambda,\rho_n\rangle}=
\sum_i\dim\mathscr H^{4i-4\langle\lambda,\rho_n\rangle}_x(\IC_\lambda)q^{i}.\]

The case of $\frak LK$ orbits on $\Gr_{2n}$ follows from the 
 \cite[Proposition 6.10 (3)]{CN2}
saying that 
there is a stratified $K_c$-equivariant homeomorphism 
between 
$\Omega K_c\backslash\Gr_{2n}$ and $\Gr_{n,\bbH}$
(where $\Omega K_c$ is the based loop group of $K_c$)
with stratifications given by images of 
$\frak L K$-orbits on $\Gr_{2n}$ in
the quotient $\Omega K_c\backslash\Gr_{2n}$ and the 
$\frak L^+G_{n,\bbH}$-orbits on $\Gr_{n,\bbH}$.

\subsubsection{Proof of Theorem \ref{IC-stalks}}
We follow closely the presentation in \cite[Section 5]{Z}.
For any $t\in\ft$ we denote by 
$\kappa(t)$ the residue field of $t$.
The specialized cohomology 
\[H_t(\Gr_{n,\bbH},\mF):=H^*_{T_c}(\Gr_{n,\bbH},\mF)\otimes_{R_T}\kappa(t)\]
carries a canonical filtration
\[H^{\leq i}_t(\Gr_{n,\bbH},\mF):=\on{Im}(\sum_{j\leq i} H^j_{T_c}(\Gr_{n,\bbH},\mF)\to H_t(\Gr_{n,\bbH},\mF))\]
Let us identify 
$H_t(\Gr_{n,\bbH},\mF)\is (H^*(\Gr_{n,\bbH},\mF)\otimes R_T)\otimes_{R_T}\kappa(t)\is V$ via the 
canonical splitting in~\eqref{real MV filtration}.
As explained in the proof of Proposition \ref{computation of Ext},
the cohomological grading  on
$H^*(\Gr_{n,\bbH},\mF)$ corresponds to the grading on the representation $V$
given by 
the eigenvalues of $4\rho_n$. 
It follows that 
the filtration $H^{\leq i}_t(\Gr_{n,\bbH},\mF)$ 
corresponds to  the increasing filtration on $V$ given by
the eigenvalues of $4\rho_n$ (see, e.g., \cite[Theorem 5.2.1]{G1}).
 
Fix a generic element $t=(t_1,...,t_n)\in\ft$ away from the root hyperplanes. The localization theorem 
implies that there is an isomorphism
\beq\label{local}
\bigoplus_{\mu\in\Lambda_n} H_t(s_\mu^!\mF)\is H_t(\Gr_{n,\bbH},\mF)
\eeq
Recall the description of  
the equivariant homology 
$\on{Spec}(H_*^{T_c}(\Gr_{n,\bbH}))\is (G_n\times\ft)^{e_X^T}$
in Lemma \ref{computation}.
The fiber of the group scheme $(G_n\times\ft)^{e_X^T}$ over $t$ is the 
centralizer subgroup $ (G_n)^{e_X^T(t)}\subset G_n$ of the element 
$e_X^T(t)\in\fg_n$ in~\eqref{e^T_X} .
Note that $e_X^T(t)$ is  
conjugate to the diagonal matrix $\on{diag}(t_1^2,...,t_n^2)\in\ft_n$ and 
hence is regular semi-simple (we have $t_i^2\neq t_j^2$ for $i\neq j$
as the Weyl group 
$\rW=\rW_n\ltimes\{\pm1\}^n$ acts freely on $t$).
Thus $ (G_n)^{e_X^T(t)}$ is a maximal torus and 
there is a canonical isomorphism $ (G_n)^{e_X^T(t)}\is T_n$ 
given by $x\to\on{Ad}_ux$, where $x\in (G_n)^{e_X^T(t)}$ and
$u\in G_n$ is any element satisfies 
$\on{Ad}_ue_X^T(t)=\on{diag}(t_1^2,...,t_n^2)$.\footnote{it is easy to see that the 
isomorphism is independent of the choice of $u$.}
It is shown in \cite{O}, that the decomposition in~\eqref{local} corresponds to the 
weight decomposition under $(G_n)^{e_X^T(t)}$:

\begin{lemma}
The decomposition in~\eqref{local} corresponds, under the canonical isomorphism 
$H_t(\Gr_{n,\bbH},\mF)\is V$, the weight decomposition 
$V=\oplus_{\mu\in\Lambda_n} V(\mu_t)$ with respect to the 
action the maximal tours $(G_n)^{e_X^T(t)}$.
Here $V(\mu_t)$ is the weight space associated to the character 
$\mu_t:(G_n)^{e_X^T(t)}\is T_n\stackrel{\mu}\to\bC^\times$.
\end{lemma}

Choose $t\in\ft$ such that 
$e_X^T(t)=e_n+2\rho_n$. Let $u$ be the unique element in $N_n$ such that
$\on{Ad}_u(e_n+2\rho_n)=2\rho_n$.
\begin{lemma}
We haves 
\[H_t^{\leq4i+2m}(\Gr_{n,\bbH},\mF)\cap\bigoplus_{\mu\in\Lambda_n,2\langle\mu,\rho_n\rangle=m}H_t(s^!_\mu\mF)=F_iV\cap \bigoplus_{\mu\in\Lambda_n,2\langle\mu,\rho_n\rangle=m} V(\mu_t)\]
\end{lemma}
\begin{proof}
Let $V=\bigoplus V^1(i)$ and $V=\bigoplus V^2(i)$ be two gradings 
on $V$ given by the cocharacter $2\rho_n$ and $\on{Ad}_{u^{-1}}2\rho_n$ respectively.
Let $F^1_iV$ and $F^2_iV$ be the two filtrations on $V$ given by
$F^1_iV=\oplus_{j\leq i} V^1(j)$ and $F^2_iV=\on{ker}(e_n^{i+1})$.
We have \[F_iV\cap \bigoplus_{\mu\in\Lambda_n,2\langle\mu,\rho_n\rangle=m} V(\mu_t)=F^2_iV\cap V^2(m)\]
and \[H_t^{\leq4i+2m}(\Gr_{n,\bbH},\mF)\cap\bigoplus_{\mu\in\Lambda_n,2\langle\mu,\rho_n\rangle=m}H_t(s^!_\mu\mF)=F^1_{2i+m}(V)\cap V^2(m)\]
and the desired claim follows from \cite[Lemma 5.5]{Z}.
\end{proof}

Note that we have shown in~\eqref{surj} that the natural map 
$H^*_{T_c}(\Gr_{n,\bbH},\mF)\to H^*_{T_c}(s_\mu^*\mF)$
is a surjective map of free $R_T$-modules and it implies the dual map 
$ H^*_{T_c}(s_\mu^!\mF)\to H^*_{T_c}(\Gr_{n,\bbH},\mF)$ is a splitting injective map of free $R_T$-modules.
Thus we have 
\[H^{\leq i}_{t}(\Gr_{n,\bbH},\mF)\cap H_t(s_\mu^!\mF)=H^{\leq i}_t(s_\mu^!\mF)\]
On the other hand,
the element $u\in N_n$ above  maps 
$V(\mu_t)$ to $V(\mu)$ and preserves the filtration $F_iV$, and hence 
$\on{dim}(F_iV(\mu)/F_{i-1}V(\mu))=\dim(F_iV(\mu_t)/F_{i-1}V(\mu_t))$.
Now the lemma above implies 
\[P_\mu(V,q)=\sum_i\dim(F_iV(\mu)/F_{i-1}V(\mu))q^i=
\sum_i\dim(F_iV(\mu_t)/F_{i-1}V(\mu_t))q^i=\]
\[=\sum_i
\on{dim}(H^{\leq4i+4\langle\mu,\rho_n\rangle}_{t}(\Gr_{n,\bbH},\mF)\cap H_t(s_\mu^!\mF)/H^{\leq4(i-1)+4\langle\mu,\rho_n\rangle}_{t}(\Gr_{n,\bbH},\mF)\cap H_t(s_\mu^!\mF))q^i=\]
\[=
\sum_i
\on{dim}(H_t^{\leq4i+4\langle\mu,\rho_n\rangle}(s_\mu^!\mF)/H_t^{\leq4(i-1)+4\langle\mu,\rho_n\rangle}(s_\mu^!\mF))q^i.\]
To conclude the proof, we observe that under the canonical isomorphism 
$H_t(s_\mu^!\mF)\is H^*(s_\mu^!\mF)$ the canonical filtration on the left hand side corresponds to the 
cohomological degree filtration on the right hand side and hence we obtain
\[P_\mu(V,q)=
\sum_i
\on{dim}(H_t^{\leq4i+4\langle\mu,\rho_n\rangle}(s_\mu^!\mF)/H_t^{\leq4(i-1)+4\langle\mu,\rho_n\rangle}(s_\mu^!\mF))q^i=\sum_i\on{dim } H^{4i+4\langle\mu,\rho_n\rangle}(s_\mu^!\mF)q^i.\]

\section{Main results}\label{Main results}

\subsection{Formality}
The goal of this section is to show that the dg-algebra 
\[\on{RHom}_{D^b_{}(\frak L^+G_{n,\bbH}\backslash\Gr_{n,\bbH})}(\IC_0,\IC_0\star\mO(G_n))\] is formal.

The proof is based on the following 
key proposition.
The existence of left adjoint of nearby cycles functor in Lemma \ref{left adjoint}
gives rise to a map between $K_c$-equivariant cohomology
\beq
H^*_{K_c}(\Gr_{n,\bbH})\is
\on{Ext}^*(\bC_{\Gr_{n,\bbH}},\bC_{\Gr_{n,\bbH}})\stackrel{^L\mathrm R'}\lra\on{Ext}^*(\bC_{\Gr_{2n}},\bC_{\Gr_{2n}})\is H^*_{K_c}(\Gr_{2n}).
\eeq
By taking the graded dual (see Section \ref{equ coh and hom}), we get a map between 
equivariant homology 
\beq\label{map between hom}
H_*^{K_c}(\Gr_{2n})\to H_*^{K_c}(\Gr_{n,\bbH}).
\eeq

\begin{prop}\label{key prop}
We have a commutative digram
\[\xymatrix{\on{Spec}(H_*^{K_c}(\Gr_{n,\bbH}))\ar[r]^{}\ar[d]^\simeq&
\on{Spec}(H_*^{K_c}(\Gr_{2n}))\ar[d]^\simeq\\
J_n\ar[r]^{}&J_{2n}|_{\frak c_n}}\]
where the bottom arrow $J_n\to J_{2n}|_{\fc_n}$ is the morphism introduced in
~\eqref{key diagram}.  
\end{prop}

\begin{proof}
We shall  verify the statement for 
$T_c$-equivariant homology, that is, we have a commutative diagram
\[\xymatrix{\on{Spec}(H_*^{T_c}(\Gr_{n,\bbH}))\ar[r]^{}\ar[d]^\simeq&
\on{Spec}(H_*^{T_c}(\Gr_{2n}))\ar[d]^\simeq\\
(G_{n}\times\ft)^{e^T_X}\is J_n
\times_{\fc_n}\ft\ar[r]^{}&(G_{2n}\times\ft)^{e^T}\is J_{2n}\times_{\fc_{2n}}\ft}\]
where the bottom arrow is the map~\eqref{base change to t}.
All the maps above are compatible with the natural $\rW$-actions and taking $\rW$-invariants we get
the desired claim.

Let $V_{\omega_1}$ be the standard representation of $G_{2n}$.
Recall the isomorphisms 
\beq\label{complex}
H^*_{T_c}(\Gr_{2n},\IC_{\omega_1})\is
H^*_{}(\Gr_{2n},\IC_{\omega_1})\otimes R_T\is V_{\omega_1}\otimes R_T
\eeq
\beq\label{real}
H^*_{T_c}(\Gr_{n,\bbH},\mathrm R(\IC_{\omega_1}))\is H^*_{}(\Gr_{n,\bbH},\mathrm R(\IC_{\omega_1}))\otimes R_T\is
V_{\omega_1}\otimes R_T
\eeq
induce by the complex and real MV filtrations.
 Together with the canonical isomorphism 
\beq\label{can iso}
H^*_{T_c}(\Gr_{n,\bbH},\mathrm R(\IC_{\omega_1}))\is H^*_{T_c}(\Gr_{2n},\IC_{\omega_1})
\eeq
we get an automorphism 
\beq\label{change of basis}
V_{\omega_1}\otimes R_T\is H^*_{T_c}(\Gr_{n,\bbH},\mathrm R(\IC_{\omega_1}))\is H^*_{T_c}(\Gr_{2n},\IC_{\omega_1})\is
V_{\omega_1}\otimes R_T
\eeq
and hence an element 
\beq\label{Phi'}
\Phi'\in \GL(V_{\omega_1})\otimes R_T\is G_{2n}\otimes R_T
\eeq
Note that 
the isomorphisms~\eqref{complex} and~\eqref{real}
map the standard basis  $\{e_1\otimes 1,...,e_{2n}\otimes 1\}$ of $V_{\omega_1}\otimes R_{T}$
to the basis 
\[\{b_1,...,b_{2n}\}=\{[\mathbb P^{0}],...,[\mathbb P^{2n-1}]\}\]
of  $H^*_{T_c}(\Gr_{2n},\IC_{\omega_1})\is 
H^{*}_{T_c}(\mathbb P^{2n-1})$
(up to a constant degree shift)
and the basis 
\[\{c_1,...,c_{2n}\}=\{[\mathbb{HP}^0][2],[\mathbb{HP}^0],[\mathbb{HP}^1][2],[\mathbb{HP}^1], ...,[\mathbb{HP}^{n-1}][2],[\mathbb{HP}^{n-1}]\}\]
of 
$H^*_{T}(\Gr_{n,\bbH},\mathrm R(\IC_{\omega_1}))\is 
H^{*}_{T_c}(\mathbb{HP}^{n-1})\oplus H^{*-2}_{T_c}(\mathbb{HP}^{n-1})$
(up to a constant degree shift) respectively, and the  element $\Phi'$
is the matrix for the linear map sending
$c_i\to b_i$ in the basis $c_1,...,c_{2n}$
(which is not the identity element).

By Lemma \ref{computation}, 
there is a commutative diagram
\beq\label{key com diag}
\xymatrix{(G_n\times\ft)^{e^T_X}\ar[r]^{\simeq\ \ \ }&\on{Spec}H^*_{T_c}(\Gr_{n,\bbH})\ar[d]\ar[r]&
\on{GL}(H^*_{T_c}(\Gr_{n,\bbH},\mathrm R(\IC_{\omega_1})))\ar[r]^{\ \ \ \ \ \ \ \ \ \simeq}_{\ \ \ \ \ \ \ \ \eqref{real}}\ar[d]^{\simeq}_{\eqref{can iso}}&G_{2n}\times\ft\ar[d]^{\simeq}_{\on{Ad}_{\Phi'}}\\
(G_{2n}\times\ft)^{e^T}
\ar[r]^{\simeq\ \ \ }&\on{Spec}H^*_{T_c}(\Gr_{2n})\ar[r]^{}&\on{GL}(H^*_{T_c}(\Gr_{2n},\IC_{\omega_1}))\ar[r]^{\ \ \ \ \ \ \ \ \simeq}_{\ \ \ \ \ \ \ \eqref{complex}}&G_{2n}\times\ft}
\eeq
where the upper and lower middle arrows are given by the 
 co-action of  $H_*^{T_c}(\Gr_{n,\bbH})$ and $H_*^{T_c}(\Gr_{2n})$
on $H^*_{T_c}(\Gr_{n,\bbH},\mathrm R(\IC_{\omega_1}))$
and $H^*_{T_c}(\Gr_{2n},\IC_{\omega_1})$,
and right vertical isomorphism is given by the conjugation action
\[\on{Ad}_{\Phi'}:G_{2n}\times\ft\to G_{2n}\times\ft \ \ \ \ (g,t)\to (\on{Ad}_{\Phi'(t)}g,t)\]

Note that in the above diagram the lower composed map
$(G_{2n}\times\ft)^{e^T}\to G_{2n}\times\ft$
is the natural embedding 
and the upper composed map
$(G_n\times\ft)^{e^T_X}\to G_{2n}\times\ft$
is the restriction of the map 
\[\on{Ad}_P\circ\delta:G_n\times\ft\to G_{2n}\times\ft\to G_{2n}\times\ft\ \ \ \ (g,t)\to (P\delta(g)P^{-1},t)\]
to $(G_n\times\ft)^{e^T_X}$ where 
$P\in G_{2n}$ is the permutation matrix which sends the the ordered basis 
$\{e_1,e_3,...,e_{2n-1},e_2,e_4,...,e_{2n}\}$ to the ordered basis
$\{e_1,...,e_{2n}\}$ (see Section \ref{Kostant section}).

Thus in view of the description of the map
$(G_{n}\times\ft)^{e^T_X}\to (G_{2n}\times\ft)^{e^T}$ in~\eqref{base change to t}
we need to show that the 
element
\[\Phi:=\Phi'\circ P\in G_{2n}\otimes R_T\]
satisfies
\beq\label{change of bases}
e^T=\Phi (\tau\circ e^T_X) \Phi^{-1}\in\fg_{2n}\otimes R_T
\eeq
To this end, 
we observe that, by Lemma \ref{matrix presentation},
the elements 
$\tau\circ e^T_X$
and $e^T$ in $\fg_{2n}\otimes R_T$ 
 are the matrices of 
the cup product map
$c_1^T(\mL)\cup(-):H^*_{T_c}(\Gr_{2n}, \IC_{\omega_1})\to H^*_{T_c}(\Gr_{2n}, \IC_{\omega_1})$
 in the bases 
$\{d_1,...,d_{2n}\}=
\{[\mathbb{HP}^0][2],...,[\mathbb{HP}^{n-1}][2],[\mathbb{HP}^0],...,[\mathbb{HP}^{n-1}]\}$  and $\{b_1,...,b_{2n}\}$
respectively.
On the other hand, the element 
$\Phi=\Phi'\circ P$ is the matrix for the linear map sending
$d_i\stackrel{}\to c_i\stackrel{}\to b_i$ in the basis $d_1,...,d_{2n}$, and hence ~\eqref{change of bases} holds. 
This completes the proof of the proposition.

\end{proof}

\begin{remark}\label{construction of Phi}
The proof  gives a canonical construction of the element $\Phi$ in~\eqref{key property}.
\end{remark}

\begin{proposition}\label{formality}
The dg-algebra $\on{RHom}_{D^b_{}(\frak L^+G_{n,\bbH}\backslash\Gr_{n,\bbH})}(\IC_0,\IC_0\star\mO(G_n))$ is formal.
\end{proposition}
\begin{proof}
Consider the following dg-algebras 
\[A=\on{RHom}_{D^b_{}(\frak L^+G_{2n}\backslash\Gr_{2n})}(\IC_0,\IC_0\star\mO(G_{2n})),
\ B=\on{RHom}_{{D^b_{}(\frak L^+G_{n,\bbH}\backslash\Gr_{n,\bbH})}}(\IC_0,\IC_0\star\mO(G_n\times\bG_m)).\] 
Here 
we regard
$\mO(G_n\times\bG_m)\is\bigoplus_{j\in\bZ}\IC_{\mO(G_n)}[j]$ via the 
monoidal functor
$\on{Rep}(G_n\times\bG_m)\is
\bigoplus_{j\in\bbZ}\on{Perv}(\Gr_{n,\bbH})[j]\subset
D^b_{}(\frak L^+G_{n,\bbH}\backslash\Gr_{n,\bbH})$. 
Proposition \ref{abelian Satake} (3) implies that the 
nearby cycle functor gives rise to a map of dg-algebras
\[\phi:A=\on{RHom}_{D^b_{}(\frak L^+G_{2n}\backslash\Gr_{2n})}(\IC_0,\IC_0\star\mO(G_{2n}))\stackrel{\mathrm R}\to
\on{RHom}_{D^b_{}(\frak L^+G_{n,\bbH}\backslash\Gr_{n,\bbH})}(\IC_0,\IC_0\star(\on{Res}^{G_{2n}}_{G_n\times\bG_m}\mO(G_{2n})))\]
\[\to B=\on{RHom}_{{D^b_{}(\frak L^+G_{n,\bbH}\backslash\Gr_{n,\bbH})}}(\IC_0,\IC_0\star\mO(G_n\times\bG_m))\]
where the last arrow is induced by the quotient map 
$\on{Res}^{G_{2n}}_{G_n\times\bG_m}\mO(G_{2n})\to\mO(G_n\times\bG_m)$
(in the category of $\on{Rep}(G_n\times\bG_m)$).
The right regular representations of $G_{2n}$ on $G_n\times\bG_m$
induce natural  $G_{2n}$
and $G_n\times\bG_m$-action on $A$ and $B$
and their restriction to 
the subgroup $\bG_m\subset G_n\times\bG_m\subset G_{2n}$
gives rise to a $\bG_m$-weight decomposition $A=\oplus_{j\in\bZ} A_j$ and $B=\oplus_{j\in\bZ} B_j$.
Note that the zero weight spaces $A_0$ and $B_0$ are dg-subalgebras of $A$ and $B$
and $B_0=\on{RHom}_{D^b_{}(\frak L^+G_{n,\bbH}\backslash\Gr_{n,\bbH})}(\IC_0,\IC_0\star\mO(G_n))$.

According to \cite{BF}, the dg-algebra $A$ is formal moreover we have  $A\is H^*(A)\is\mO(\fg_{2n}[2])$. Note the map $\phi:A\to B$ above respects the $\bG_m$-action and hence 
restricts to a map $\phi_0:A_0\to B_0$ fitting into the following diagram  
\[\xymatrix{A_0\ar[r]^{\phi_0}\ar[d]&B_0\ar[d]\\
A\ar[r]^{\phi_{\on{}}}&B}\]
We claim that the map 
$H^*(\phi_0):H^*(A_0)\to H^*(B_0)$ between cohomology is surjective.
Since $A_0$ is formal with generators in even degree 
and 
$H^*(B_0)\is H^*(\on{RHom}_{D^b_{}(\frak L^+G_{n,\bbH}\backslash\Gr_{n,\bbH})}(\IC_0,\IC_0\star\mO(G_n)))\is\mO(\fg_n[4])$
which is a polynomial ring with generators in even degree (see Lemma \ref{computation of Ext}),  Lemma \ref{criterion}
 below implies that $B_0$ is formal. The proposition follows.

Proof of the claim.
 To show the surjectivity of $H^*(\phi_0):H^*(A_0)\to H^*(B_0)$, we can ignore the 
 grading and view $H^*(\phi_0)$ as maps between plain algebras.
We have a commutative diagram 
\[\xymatrix{ H^*(A)\ar[r]^{\simeq\ \ \ \ \ \ \ \ \ \ \ \ \ \ \ \ \ \ \ \ \ \ \ \ \ \ \ \ }\ar[d]^{H^*(\phi)}& \on{Hom}^*_{H^*_{G_{2n}}(\Gr_{2n})}(H^*_{G_{2n}}(\IC_0),H^*_{G_{2n}}(\IC_0\star\mO(G_{2n}))\ar[d]\\
H^*(B)
\ar[r]^{\simeq\ \ \ \ \ \ \ \ \ \ \ \ \ \ \ \ \ \ \ \ \ \ \ \ \ \ \ \ \ }&\on{Hom}^*_{H^*_{K_c}(\Gr_{n,\bbH})}(H^*_{K_c}(\IC_0),H^*_{K_c}(\IC_0\star\mO(G_n\times\bG_m)))}\]
where the horizontal isomorphisms
are given by the functor of equivariant cohomology,
see Proposition \ref{faithfulness}.
Note that $\IC_0\star\mO(G_n\times\bG_m)\is\bigoplus_{j\in\bZ}\IC_{\mO(G_n)}[j]$
is a direct sum of shifts of IC-complexes
and hence 
Proposition \ref{faithfulness} is applicable. On the other hand, using Proposition \ref{key prop}, we can identify the right vertical arrow as
\[\xymatrix{ \on{Hom}^*_{H^*_{G_{2n}}(\Gr_{2n})}(H^*_{G_{2n}}(\IC_0),H^*_{G_{2n}}(\IC_0\star\mO(G_{2n}))\ar[d]\ar[r]&\mO(G_{2n}\times\fc_{2n})^{J_{2n}}\ar[d]\\
\on{Hom}^*_{H^*_{K_c}(\Gr_{n,\bbH})}(H^*_{K_c}(\IC_0),H^*_{K_c}(\IC_0\star\mO(G_n\times\bG_m)))\ar[r]&\mO(G_n\times\bG_m\times\fc_n)^{J_n}}\]
where the right vertical map above is induced by the embeddings 
$\tau:\fc_n\to\fc_{2n}$ in~\eqref{c_X} and
$\delta\times2\rho_L:G_n\times\bG_m\to G_{2n}$ in~\eqref{SL_2}.
The groups schemes $J_{2n}$ and $J_n$ act on 
$\mO(G_{2n}\times\fc_{2n})$ and $\mO(G_{n}\times\bG_m\times \fc_n)$
via the identification $J_{2n}\is (G_{2n}\times\fc_{2n})^{\on{Ad}_P^{-1}\circ\kappa_{2n}}$
and $J_n\is (G_{n}\times\fc_n)^{\tau\circ\kappa_n}$ 
where $\on{Ad}_{P}^{-1}\circ\kappa_{2n}:\fc_{2n}\stackrel{\kappa_{2n}}\to\fg_{2n}^{\on{reg}}\stackrel{\on{Ad}_{P}^{-1}}\to\fg_{2n}^{\on{reg}}$ and
$\tau\circ\kappa_n:\fc_n\stackrel{\kappa_n}\to\fg^{\on{reg}}_{n}\stackrel{\tau}\to\fg^{\on{reg}}_{2n}$ are the maps in~\eqref{diagram ks}.
Thus we have a commutative diagram
\[\xymatrix{\mO(G_{2n}\times\fc_n)^{J_{2n}}\ar[d]\ar[r]^{\simeq\ \ \ \ \ }&\mO(\fg_{2n}^{\on{reg}})\is\mO(\fg_{2n})\ar[d]\\
\mO(G_n\times\bG_m\times\fc_n)^{J_n}\ar[r]^{\simeq\ \ \ \ \ }&\mO(\fg_n^{\on{reg}}\times\bG_m)\is\mO(\fg_n\times\bG_m)}\]
where the right vertical arrow is
 given by pull-back of functions along the map
\beq\label{action}
\fg_n\times\bG_m\to\fg_{2n},\ \ \ \ (C,t)\to \on{Ad}_{2\rho_L(t)^{-1}}\tau(C)=
\begin{pmatrix}0&t^{-2}I_n\\
t^{2}C&0\end{pmatrix}.
\eeq
All together we can identify $H^*(\phi):H^*(A)\to H^*(B)$
with the  map $\mO(\fg_{2n})\to \mO(\fg_n\times\bG_m)$ 
(as map between non-graded algebras)
and we need to show that the induced map
\beq\label{zero wt space}
\mO(\fg_{2n})_0\to\mO(\fg_n\times\bG_m)_0=\mO(\fg_n)
\eeq
between the zero $\bG_m$-weight spaces is surjective.
For this we observe that the map 
\beq\label{section}
\fg_{2n}\to\fg_n\ \ \ \begin{pmatrix}A&B\\
C&D\end{pmatrix}\to BC
\eeq 
is $\bG_m$-equivariant 
($\bG_m$ acts trivially on $\fg_n$)
and the composition $\fg_n\times\bG_m\stackrel{~\eqref{action}}\to\fg_{2n}\stackrel{~\eqref{section}}\to\fg_n$
is the projection map $(C,t)\to C$. Thus the pull-back 
map
$\mO(\fg_n)\to\mO(\fg_{2n})_0$ 
along~\eqref{section}
defines a section of~\eqref{zero wt space}.
We are done.

\end{proof}

\begin{lemma}\label{criterion}
Let $\phi:A_1\to A_2$ be a map of dg-algebras.
Assume that (1) $H^*(A_1)$ commutative with 
with generators in even degree and 
$H^*(A_2)$
is a commutative polynomial ring 
with generators in even degree and (2) 
the map $H^*(\phi):H^*(A_1)\to H^*(A_2)$ is surjective.
Then $A_1$ is formal implies $A_2$ is formal.
\end{lemma}
\begin{proof}
Let  
$x_1,...,x_l$ be the set of generators of $H^*(A_2)$ in even degree such
that 
$\bC[x_1,...,x_l]\is H^*(A_2)$. Since $H^*(\phi):H^*(A_1)\to H^*(A_2)$
is surjective, one can find homogeneous elements $ y_1,...,y_l\in H^*(A_1)$
such that $H^*(\phi)(y_i)=x_i$ for $i=1,...,l$. Assume $H^*(A_1)\is A_1$
is formal, then we have map of dg-algebras
$k[z_1,...,z_l]\to H^*(A_1)\is A_1$ sending 
$z_i$ to $y_i$. Then the composition 
$\gamma:\bC[z_1,...,z_l]\to H^*(A_1)\is A_1\stackrel{\phi}\to A_2$
defines a dg-algebra morphism such that 
$H^*(\gamma):\bC[z_1,...,z_l]\is\bC[x_1,...,x_l]\is H^*(A_2)$ is the isomorphism sending
$z_i$ to $x_i$. The lemma follows.

\end{proof}

\subsection{Derived geometric Satake equivalence for the quaternionic groups}
\quash{
Denote by $D_{\frak L^+G_{n,\bbH}}(\Gr_{n,\bbH})$  the dg-category 
of constructible complexes on the real analytic stack $\frak L^+G_{n,\bbH}\backslash\Gr_{n,\bbH}$.
It is known that $D_{\frak L^+G_{n,\bbH}}(\Gr_{n,\bbH})$ is compactly generated, i.e., 
we have 
$D_{\frak L^+G_{n,\bbH}}(\Gr_{n,\bbH})\is\on{Ind}(D_{\frak L^+G_{n,\bbH}}(\Gr_{n,\bbH})^{c})$
where the ride hand side is the ind-completion of the 
the full subcategory $D_{\frak L^+G_{n,\bbH}}(\Gr_{n,\bbH})^{c}$ 
 of compact objects. 
We have a natural inclusion $D_{\frak L^+G_{n,\bbH}}(\Gr_{n,\bbH})^c\subset 
D^b_{\frak L^+G_{n,\bbH}}(\Gr_{n,\bbH})$ but in general it is not an equivalence. In fact, $D^b_{\frak L^+G_{n,\bbH}}(\Gr_{n,\bbH})$
is the full subcategory of $D_{\frak L^+G_{n,\bbH}}(\Gr_{n,\bbH})$
consisting those objects that become compact after applying 
the forgetful functor $D_{\frak L^+G_{n,\bbH}}(\Gr_{n,\bbH})\to
D_{}(\Gr_{n,\bbH})$.
Denote by $\on{Ind}(D^b_{\frak L^+G_{n,\bbH}}(\Gr_{n,\bbH}))$ the ind-completion of $D^b_{\frak L^+G_{n,\bbH}}(\Gr_{n,\bbH})$. 
}
Denote by
$D^{G_n}(\on{Sym}(\fg_n[-4]))$ be the dg-category of  $G_n$-equivariant dg-modules 
over the dg-algebra $\on{Sym}(\fg_n[-4])$ (equipped with trivial differential).
It is known that $D^{G_n}(\on{Sym}(\fg_n[-4]))$ is compactly generated and the full subcategory 
$D^{G_n}(\on{Sym}(\fg_n[-4]))^c$ of compact objects 
coincides with 
be the full subcategory $D^{G_n}(\on{Sym}(\fg_n[-4]))^c=D^{G_n}_{\on{perf}}(\on{Sym}(\fg_n[-4]))$
consisting of perfect modules.
Denote by $D^{G_n}(\on{Sym}(\fg_n[-4]))_{\on{Nilp}(\fg_n)}$
and  $D^{G_n}_{\on{perf}}(\on{Sym}(\fg_n[-4]))_{\on{Nilp}(\fg_n)}$
the full subcategory of $D^{G_n}(\on{Sym}(\fg_n[-4]))$
and $D^{G_n}_{\on{perf}}(\on{Sym}(\fg_n[-4]))$ respectively
consisting of modules that are set theoretically supported on the 
nilpotent cone $\on{Nilp}(\fg_n)$ of $\fg_n$.

Note that the category $D^{G_n}(\on{Sym}(\fg_n[-4]))$ 
(resp $D^{G_n}_{\on{perf}}(\on{Sym}(\fg_n[-4]))$, $
D^{G_n}(\on{Sym}(\fg_n[-4]))_{\on{Nip}(\fg_n)}$, 
$
D^{G_n}_{\on{perf}}(\on{Sym}(\fg_n[-4]))_{\on{Nip}(\fg_n)}$)
has a natural
monoidal structure  given by the
(derived) tensor product: 
$(\mF_1,\mF_2)\to \mF_1\otimes\mF_2:=
\mF_1\otimes^L_{\on{Sym}(\fg_n[-4])}\mF_2$.

\begin{thm}\label{main}
\

\begin{enumerate}
\item There is a canonical equivalence of monoidal categories
\[\on{Ind}(D^b_{}(\frak L^+G_{n,\bbH}\backslash\Gr_{n,\bbH}))\is D^{G_n}_{}(\on{Sym}(\fg_n[-4]))\]
which induces a monoidal equivalence 
\[D^b_{}(\frak L^+G_{n,\bbH}\backslash\Gr_{n,\bbH})\is D^{G_n}_{\on{perf}}(\on{Sym}(\fg_n[-4]))\]
between the corresponding (non-cocomplete) full subcategory of compact objects.
\item
There is a canonical equivalence of monoidal categories
\[D_{}(\frak L^+G_{n,\bbH}\backslash\Gr_{n,\bbH})\is D^{G_n}_{}(\on{Sym}(\fg_n[-4]))_{\on{Nilp}(\fg_n)}\]
which induces 
a monoidal equivalence \[D_{}(\frak L^+G_{n,\bbH}\backslash\Gr_{n,\bbH})^c\is D^{G_n}_{\on{perf}}(\on{Sym}(\fg_n[-4]))_{\on{Nilp}(\fg_n)}\]
between the corresponding (non-cocomplete) full subcategory of compact objects.

\end{enumerate}
\end{thm}
\begin{proof}
Proof of (1).
Write $\calC:=\on{Ind}(\D^b(\frak L^+G_{n,\bbH}\backslash\Gr_{n,\bbH}))$.  The 
dg-category $\calC$
is a module category for the dg-category 
$D(\on{QCoh}(BG_n))$ of quasi-coherent sheaves on $BG_n$
and
we can 
form the de-equivariantized category $\calC_{\on{deeq}}:=\calC\times_{BG}\{\on{pt}\}$ 
with objects $\on{Ob}(\calC_{\on{deeq}})=\on{Ob}(\calC)$ and (dg)-morphisms
\[\on{Hom}_{\calC_{\on{deeq}}}(\mF_1,\mF_2)=\on{Hom}_{\calC}(\mF_1,\mF_2\star\mO(G_n))=
\on{RHom}_{\on{Ind}(D^b_{}(\frak L^+G_{n,\bbH}\backslash\Gr_{n,\bbH}))}(\mF_1,\mF_2\star\mO(G_n)).\]
Every object $\mF\in\calC_{\on{deeq}}$ carries a natural action of $G_n$ and we can recover $\calC$ by taking $G_n$-equivariant objects in $\calC_{\on{deeq}}$.
The fact that $\IC_0$ is compact and generates $\calC$ under the action of $D\on{QCoh}(BG_n)$
implies that $\IC_0$, viewed as an object in $\calC_{\on{deeq}}$, is a compact generator.
Hence the Barr-Beck-Lurie theorem implies 
that the assignment $\mF\to \on{Hom}_{\calC_{\on{deeq}}}(\IC_0,\mF)$
defines  an equivalence of categories 
\[\calC_{\on{deeq}}\is D(\on{Hom}_{\calC_{\on{deeq}}}(\IC_0,\IC_0)^{\on{op}})\ \ \  \ (resp.\ \  \calC_{}
\is D^{G_n}(\on{Hom}_{\calC_{\on{deeq}}}(\IC_0,\IC_0)^{\on{op}}))
\]
where $\on{Hom}_{\calC_{\on{deeq}}}(\IC_0,\IC_0)^{op}$ is the opposite of the dg-algebra of endomorphism 
of $\IC_0$ and 
$D(\on{Hom}_{\calC_{\on{deeq}}}(\IC_0,\IC_0)^{\on{op}})$
(resp. $D^{G_n}(\on{Hom}_{\calC_{\on{deeq}}}(\IC_0,\IC_0)^{\on{op}})$)
are the corresponding dg-categories of dg-modules (resp. $G_n$-equivariant dg-modules).
Now Proposition \ref{computation of Ext} and Proposition \ref{formality} implies 
that the dg-algebra $\on{Hom}_{\calC_{\on{deeq}}}(\IC_0,\IC_0)^{\on{op}}$ is formal and there is 
$G_n$-equivariant
isomorphism
\[\on{Hom}_{\calC_{\on{deeq}}}(\IC_0,\IC_0)^{\on{op}}\is\on{RHom}_{\on{Ind}(\D^b(\frak L^+G_{n,\bbH}\backslash\Gr_{n,\bbH}))}(\IC_0,\IC_0\star\mO(G_n))^{op}\is\]
\[\is\on{Ext}^*_{\on{Ind}(\D^b(\frak L^+G_{n,\bbH}\backslash\Gr_{n,\bbH}))}(\IC_0,\IC_0\star\mO(G_n))^{op}
\is\on{Sym}(\fg_n[-4])^{op}\is\on{Sym}(\fg_n[-4])\footnote{The last isomorphism 
follows from the fact that the $\on{Sym}(\fg_n[-4])$ is commutative with grading in even degree.}\]
 and hence we conclude that there is an equivalence
\[\on{Ind}(D^b_{}(\frak L^+G_{n,\bbH}\backslash\Gr_{n,\bbH}))\is D^{G_n}(\on{Sym}(\fg_n[-4])), \ \ \ \mF \to \on{RHom}_{\on{Ind}(D^b(\frak L^+G_{n,\bbH}\backslash\Gr_{n,\bbH}))}(\IC_0,\mF\star\mO(G_n))\]
The monoidal structure on the constructed equivalence will be proved in Section \ref{monoidal}.
This finishes the proof of Part (1).
Part (2) follows from the general discussion in \cite[Section 12]{AG}.
\end{proof}

\subsection{Spectral description of  nearby cycles functors}

\subsubsection{Shift of grading}\label{shift of grading}
Let $(A=\oplus_{} A^i,d)$ be a dg-algebra 
equipped with an action of  $G=H\times\bG_m$.
We will write $A^i=\oplus A^i_j$ where 
the lower index $j$ refers to the $\bG_m$-weights coming from the $\bG_m$-action. 
Assume the $\bG_m$-weights are  even, that is, 
we have $A^i_j=0$ if $j\in 2\bZ+1$.
Following \cite[Appendix A.2]{AG}, 
one can introduce a new 
dg-algebra $(\tilde A=\bigoplus \tilde A^i_j,d)$ where 
\[\tilde A^{i}_j=A^{i+j}_j\]
such that the map 
sending a $G$-equivariant dg-module $(M=\bigoplus M^i_j,d)$ 
over $A$ 
to the dg-module
$(\tilde M=\bigoplus\tilde M^i_j,d)$
over $\tilde A$
with
\[\tilde M^i_j=M^{i+j}_j\]
induces 
an equivalence of triangulated categories
\[D^G_{}(A)\is D^G_{}(\tilde A)\ \ \ (resp.\ \ 
D^G_{\on{perf}}(A)\is D^G_{\on{perf}}(\tilde A))\]

\begin{example}\label{shift example}
Consider the dg-algebra 
$(A=\on{Sym}(\fg_{2n}),d=0)$. 
The subgroup 
 $\check G_X\times\bG_m\subset G_{2n}$ acts on 
 the generators 
 $\fg_{2n}$ of $A$ via the adjoint action and 
if we write the elements in $\fg_{2n}$ 
in the form 
\[\fg_{2n}=\{\begin{pmatrix}
A&B\\
C&D
\end{pmatrix}|A,B,C,D\in\fg_n\}
\]
then
$A,D$ are of weight zero, $B$ is of wright $2$, and $C$ is of weight $-2$.
It follows that 
\[\tilde A\is
\on{Sym}(\tilde\fg_{2n})\]
where $\tilde\fg_{2n}$ consists of elements of the form
\[\tilde\fg_{2n}=\{\begin{pmatrix}
A[0]&B[-2]\\
C[2]&D[0]
\end{pmatrix}
|A,B,C,D\in\fg_n\}.\]

\end{example}

\subsubsection{}
It follows from Example \ref{shift example} that 
we have an equivalence of categories
\beq\label{shift equ}
D^{G_n\times\bG_m}_{}(\on{Sym}(\fg_{2n}[-2]))\is D^{G_n\times\bG_m}_{}(\on{Sym}(\tilde\fg_{2n}[-2]))
\eeq
On the other hand, the natural $G_n$-equivariant map 
$\fg_{n}[4]\to \tilde\fg_{2n}[2]$
sending 
\[C[4]\to\begin{pmatrix}
0&I_n\\
C[4]&0
\end{pmatrix}\ \ \ \ C\in\fg_n
\]
gives rise to 
a map of dg-algebras
\beq\label{pull-back map}
\on{Sym}(\tilde\fg_{2n}[-2])\is\mO(\tilde\fg_{2n}[2])\lra\mO(\fg_{n}[4])\is\on{Sym}(\fg_{n}[-4])
\eeq
(here we identify the graded duals of $\tilde\fg_{2n}[-2]$ and $\fg_{n}[4]$
with 
$\tilde\fg_{2n}[2]$ and $\fg_{n}[-4]$ via the trace form)
and hence a functor
\beq\label{res}
D^{\GL_n}_{}(\on{Sym}(\tilde\fg_{2n}[-2]))\to D^{\GL_n}_{}(\on{Sym}(\fg_{n}[-4]))\ \ \ \ 
M\to \on{Sym}(\fg_{n}[-4])\otimes^L_{\on{Sym}(\tilde\fg_{2n}[-2])}M
\eeq
Finally, let us consider the functor
\beq\label{spectral}
\Phi:D^{G_{2n}}_{}(\on{Sym}(\fg_{2n}[-2]))\stackrel{F}\to D^{G_n\times\bG_m}_{}(\on{Sym}(\fg_{2n}[-2]))\stackrel{\eqref{shift equ}}\is D^{\GL_n\times\bG_m}_{}(\on{Sym}(\tilde\fg_{2n}[-2]))\stackrel{F}\to
\eeq
\[\to D^{G_n}_{}(\on{Sym}(\tilde\fg_{2n}[-2]))\stackrel{\eqref{res}}\to D^{G_n}_{}(\on{Sym}(\fg_{n}[-4]))
\]
where $F$ are the natural forgetful functors.
\begin{thm}\label{spectral nearby cycle}
The following square is commutative
\[\xymatrix{\on{Ind}D^b(\frak L^+G_{2n}\backslash\Gr_{2n})
\ar[r]^{\mathrm R}\ar[d]_{\Psi}^{\simeq}&\on{Ind}D^b(\frak L^+G_{n,\bbH}\backslash\Gr_{n,\bbH})\ar[d]_{\Psi_\bbH}^{\simeq}\\
D^{G_{2n}}_{}(\on{Sym}(\fg_{2n}[-2]))
\ar[r]^{\Phi}&D^{G_n}_{}(\on{Sym}(\fg_{n}[-4]))
}
\]
where $\Psi$ and $\Psi_\bbH$ are the complex and quaternionic Satake equivalences respectively. It induces a  similar commutative diagram 
for the subcategories of compact objects.
\end{thm}
\begin{proof}
We shall construct a natural transformation 
$\Phi\circ\Psi\to\Psi_\bbH\circ\mathrm R $. 
Write $A=
\RHom(\IC_0,\IC_{\mO(G_{2n})})\is
\on{Sym}(\fg_{2n}[-2])$, $B=\RHom(\IC_0,\IC_{\mO(G_{n})})\is\on{Sym}(\fg_n[-4])$,
and $A'=
\RHom(\IC_0,\mathrm R(\IC_{\mO(G_{2n})}))$.
Since $\mathrm R(\IC_{\mO(G_{2n})})$ 
is an algebra object in $\on{Ind}(D(\frak L^+G_{n,\bbH}\backslash\Gr_{n,\bbH}))$, 
the (dg) Hom space 
$A'$ is naturally a dg-algebra.

For any $\mF$, we have a map 
dg-modules for the dg-algebra $A$
\[\Psi(\mF)\is
\RHom(\IC_0,\mF\star\IC_{\mO(G_{2n})})\stackrel{\mathrm R}\lra
\RHom(\IC_0,\mathrm R(\mF)\star\mathrm R(\IC_{\mO(G_{2n})})):=
\Psi(\mF)'\]
where 
$A$ acts on $\Psi(\mF)'$ via the 
 dg-algebra map 
\[A=\RHom(\IC_0,\IC_{\mO(G_{2n})})\stackrel{\mathrm R}\lra A'=\RHom(\IC_0,\mathrm R(\IC_{\mO(G_{2n})})).\]
The right regular $\bG_m$-action on 
$G_{2n}$ via the 
co-character $2\rho_L:\bG_m\to G_n\times\bG_m\subset G_{2n}$
induces a $\bG_m$-action on the dg-algebras
$A$ and $A'$ (with even  weights)
and also the dg-modules $\Psi(\mF)$ and $\Psi(\mF)'$. Thus we can 
perform the shift of grading operation in Section \ref{shift of grading}
and obtain a map of dg-modules
for the dg-algebra $\tilde A$ 
\beq\label{shift}
\widetilde{\Psi(\mF)}\to\widetilde{\Psi(\mF)'}
\eeq
where $\tilde A$ acts on $\widetilde{\Psi(\mF)'}$ via the map
$\tilde A\to\tilde A'$.
By Example \ref{shift example}, we have 
\[\tilde A\is\on{Sym}(\tilde\fg_{2n}[-2]).\]
On the other hand,
by Proposition \ref{abelian Satake} (3), we have\footnote{We have the shift $[j]$ instead of $[-j]$ because we consider right regular action of $\bG_m$.}  
\[\mathrm R(\IC_{\mO(G_{2n})})\is\bigoplus_{j\in\bZ}\IC_{\on{Res}^{G_{2n}}_{G_n}\mO(G_{2n})_j}[j]\]
 where 
 \[\on{Res}^{G_{2n}}_{G_n}(\mO(G_{2n}))\is\bigoplus_{j\in\bZ}\on{Res}^{G_{2n}}_{G_n\times\bG_m}\mO(G_{2n})_j
 \]
 is the $\bG_m$-weight decomposition of $\on{Res}^{G_{2n}}_{G_n}(\mO(G_{2n}))$, and it follows that 
\[\tilde A'\is\RHom(\IC_0,\IC_{\on{Res}^{G_{2n}}_{G_n}(\mO(G_{2n}))})\ \ \ \ \ \widetilde{\Psi(\mF)'}\is
\RHom(\IC_0,\mathrm R(\mF)\star\IC_{\on{Res}^{G_{2n}}_{G_n}(\mO(G_{2n}))}).\]
Since the natural algebra map 
$\on{Res}^{G_{2n}}_{G_n}(\mO(G_{2n}))\to\mO(G_n)$
of algebra objects in $\on{Rep}(G_n)$ coming from the 
embedding $G_n\to G_{n}\times\bG_m, g\to (g,e)$
induces a map 
\[\iota:\IC_{\on{Res}^{G_{2n}}_{G_n\times\bG_m}(\mO(G_{2n}))}\to\IC_{\mO(G_n)}\]
between the corresponding algebra objects in 
$\on{Perv}(\Gr_{n,\bbH})$, we obtain a map of 
dg-algebras
\beq\label{tilde A to B}
\tilde A\to \tilde A'\is\RHom(\IC_0,\IC_{\on{Res}^{G_{2n}}_{G_n}(\mO(G_{2n}))})\stackrel{\iota}\to\RHom(\IC_0,\IC_{\mO(G_n)})=B
\eeq
and a map
dg-modules
over the dg-algebra $\tilde A$
\beq\label{natural trans'}
\widetilde{\Psi(\mF)}\stackrel{\eqref{shift}}\to\widetilde{\Psi(\mF)'}\is
\RHom(\IC_0,\mathrm R(\mF)\star\IC_{\on{Res}^{G_{2n}}_{G_n}(\mO(G_{2n}))})\stackrel{\iota}\to\RHom(\IC_0,\mathrm R(\mF)\star\IC_{\mO(G_n)})\is\Psi_\bbH\circ\mathrm R(\mF)
\eeq
where $\tilde A$ acts on $\Psi_\bbH\circ\mathrm R(\mF)$
via the morphism~\eqref{tilde A to B}. 
Moreover, the proof of Proposition \ref{formality} implies that the map~\eqref{tilde A to B} is equal to the map in~\eqref{pull-back map}.
Thus by the universal property of the tensor product, the
map~\eqref{natural trans'} gives rise to a map of dg-modules over the dg-algebra $B$
\beq\label{natural trans}
\Phi\circ\Psi(\mF)\is B\otimes^L_{\tilde A}\widetilde{\Psi(\mF)}\to
\Psi_\bbH\circ\mathrm R(\mF)
\eeq
This finished the construction of the desired natural transformation map.

Now to finish the proof, it suffices to check that~\eqref{natural trans}
is an isomorphism when $\mF$ of the form 
$\mF\is\IC_{\on{V}}$ with $\on{V}\in\on{Rep}(G_{2n})$.
For this we observe that, if $\on{V}=\oplus_{j\in\bZ}\on{V}_j$
is the $\bG_m$-weight decomposition, then we have 
\beq\label{rhs}
\Psi_\bbH\circ\mathrm R(\IC_{\on{V}})\is\bigoplus_{j\in\bZ}
\Psi_\bbH(\IC_{\on{V}_j})[-j]\is \bigoplus_{j\in\bZ}B\otimes_\bC\on{V}_j[-j]\eeq
On the other hand, we have
\[\widetilde{\Psi(\IC_{\on{V}})}\is \widetilde{(A\otimes_\bC\on{V})}\is\bigoplus_{j\in\bZ}\tilde{A}\otimes_\bC\on{V}_j[-j]\]
and hence
\beq\label{lhs}
\Phi\circ\Psi(\IC_{\on{V}})\is
B\otimes_{\tilde A}\widetilde{\Psi(\IC_{\on{V}})}\is B\otimes_{\tilde A}(\bigoplus_{j\in\bZ}\tilde{A}\otimes_\bC\on{V}_j[-j])\is
\bigoplus_{j\in\bZ}B\otimes_\bC\on{V}_j[-j]
\eeq
It follows from the construction that
the map~\eqref{natural trans} 
is given by 
 $
 \Phi\circ\Psi(\IC_{\on{V}})\stackrel{\eqref{lhs}}\is \bigoplus_{j\in\bZ}B\otimes_\bC\on{V}_j[-j]\stackrel{\eqref{rhs}}\is\Psi_\bbH\circ\mathrm R(\IC_{\on{V}})$
and hence an isomorphism.
This completes the proof of the proposition.

\end{proof}

\subsection{Monoidal structures}\label{monoidal}
We construct a monoidal structure on the equivalence
$\Psi_\bbH:\on{Ind}(D^b_{}(\frak L^+G_{n,\bbH}\backslash\Gr_{n,\bbH}))\is D^{G_n}_{}(\on{Sym}(\fg_n[-4]))$
in Theorem \ref{main}. Consider the monoidal structure 
 on $D^{G_n}_{}(\on{Sym}(\fg_n[-4]))$:
\[M_1\otimes'M_2:=\Psi_\bbH(\Psi_\bbH^{-1}(M_1)\star\Psi_\bbH^{-1}(M_2))\]
induces from the monoidal structure 
on $\on{Ind}(D^b_{}(\frak L^+G_{n,\bbH}\backslash\Gr_{n,\bbH}))$
via the equivalence $\Psi_\bbH$. 
We would like to show that $\otimes'$ is isomorphic the 
natural tensor monoidal structure.
The square in Theorem \ref{spectral nearby cycle}
together with the fact that the
derived Satake equivalence $\Psi$ is monoidal implies 
the functor 
$\Phi:D^{G_{2n}}_{}(\on{Sym}(\fg_{2n}[-2]))\ra D^{G_n}_{}(\on{Sym}(\fg_{n}[-4]))$ in \emph{loc.cit.}
is monoidal with respect to the natural tensor monoidal structure 
on $D^{G_{2n}}_{}(\on{Sym}(\fg_{2n}[-2])$ and the above monoidal structure $\otimes'$ on  $D^{G_n}_{}(\on{Sym}(\fg_{n}[-4]))$.
Now the desired claim follows from the following lemma.

\begin{lemma}
Equip $D^{G_n}_{}(\on{Sym}(\fg_{n}[-4]))$ with its natural tensor monoidal structure.

Then  the natural tensor monoidal structure on  $D^{G_n}_{}(\on{Sym}(\fg_{n}[-4]))$ is the unique (up to equivalence) monoidal structure on  $D^{G_n}_{}(\on{Sym}(\fg_{n}[-4]))$
such that the $\on{Rep}(G_{2n})$-module functor
\[\xymatrix{
\Phi:D^{G_{2n}}_{}(\on{Sym}(\fg_{2n}[-2]))
\ar[r] &D^{G_n}_{}(\on{Sym}(\fg_{n}[-4]))
}
\]
may be compatibly lifted to a monoidal functor. Moreover, the compatible monoidal structure on $\Phi$ is unique  (up to equivalence).
\end{lemma}

\begin{proof}
Returning to its construction~\eqref{spectral}, recall $\Phi$ factors into the sheared forgetful functor
\[
D^{G_{2n}}_{}(\on{Sym}(\fg_{2n}[-2])) \to  D^{G_n}_{}(\on{Sym}(\tilde\fg_{2n}[-2])) 
\]
followed by the restriction
\[
D^{G_n}_{}(\on{Sym}(\tilde\fg_{2n}[-2])) \to D^{G_n}_{}(\on{Sym}(\fg_{n}[-4]))
\qquad
M  \mapsto \on{Sym}(\fg_{n}[-4])\otimes^L_{\on{Sym}(\tilde\fg_{2n}[-2])}M
\]

First,  $\on{Sym}(\fg_{2n}[-2]))$ is the unit of $D^{G_{2n}}_{}(\on{Sym}(\fg_{2n}[-2]))$, so 
$
\on{Sym}(\fg_{n}[-4]) \simeq \Phi(\on{Sym}(\fg_{2n}[-2]))
$
 must be the unit of  $D^{G_n}_{}(\on{Sym}(\fg_{n}[-4]))$.

Next, recall that $D^{G_n}_{}(\on{Sym}(\fg_{n}[-4]))$ is compactly generated by $V \otimes \on{Sym}(\fg_{n}[-4])$ where $V$ is a finite-dimensional representation of $G_n$. Note every such $V$ is a direct summand in the restriction of a 
a finite-dimensional representation of $G_{2n}$.
Since  $\Phi$ is a $\on{Rep}(G_{2n})$-module map, this determines the monoidal product  on  $D^{G_n}_{}(\on{Sym}(\fg_{n}[-4]))$ as well as its coherent associativity  structure.

Finally, since the  monoidal structures on $\Phi$ must be compatible with its $\on{Rep}(G_{2n})$-module structure, it is  determined by its restriction to the unit $\on{Sym}(\fg_{2n}[-2]))$ where there are no choices.
\end{proof}

{}

\end{document}